\documentclass[pdflatex,sn-mathphys-num]{sn-jnl}
\usepackage{graphicx}%
\usepackage{multirow}%
\usepackage{amsmath,amssymb,amsfonts}%
\usepackage{amsthm}%
\usepackage{mathrsfs}%
\usepackage[title]{appendix}%
\usepackage{xcolor}%
\usepackage{textcomp}%
\usepackage{manyfoot}%
\usepackage{booktabs}%
\usepackage{algorithm}%
\usepackage{algorithmicx}%
\usepackage{algpseudocode}%
\usepackage{listings}%

\usepackage{tikz}
\usepackage{amsmath}
\usepackage{amssymb}
\usepackage{amsthm}
\usepackage{quiver}
\usepackage{tikz-cd}

\newtheorem{theorem}{Theorem}
\newtheorem{example}{Example}
\newtheorem{lemma}{Lemma}
\newtheorem{proposition}{Proposition}
\newtheorem{remark}{Remark}
\newtheorem{corollary}{Corollary}
\newtheorem{definition}{Definition}
\numberwithin{equation}{section}

\newcommand\inlineeqno{\stepcounter{equation}\ (\theequation)}

\newcommand{\tot}{\ensuremath{\mathrm{Tot}}}

\newcommand{\weil}{\ensuremath{\mathrm{Weil}}}

\newcommand{\SmMan}{\ensuremath{\mathrm{SmMan}}}
\newcommand{\ParMfld}{\ensuremath{\sf{ParMfld}}}

\newcommand{\Adj}{\ensuremath{\mathrm{Ad}}}
\newcommand{\downsquigarrow}{\mathbin {\raisebox{.1em}{\rotatebox[origin=c]{-90}{$\rightsquigarrow$}}}}

\usepackage[english]{babel}

\title{Lie groups in tangent join restriction categories}

\author*[1]{\fnm{Robin} \sur{Cockett}}\email{robin@ucalgary.ca}
\author*[1]{\fnm{Florian} \sur{Schwarz}}\email{florian.schwarz@ucalgary.ca}
\affil[1]{University of Calgary}

\begin{document}
\abstract{
Principal bundles have at least three different definitions, depending on the category of geometric objects studied.

In Differential Geometry, they are defined as locally trivial projection map of smooth manifolds with an atlas whose transition maps are given by group multiplication.
In Topology they are $G$-equivariantly trivial $G$-spaces.
In Algebraic Geometry, they are Étale locally isotrivial geometric quotients of $G$-varieties.

The goal of this work is to have a categorical notion that recovers all of them.
While they are different structures, they are all locally isomorphic to the Cartesian product of a base space with a group.

There are a variety of other results on group objects and their tangent bundle. In particular we show that the tangent bundle is the product of the tangent space and the group object and that the tangent space has an external Lie algebra structure, generalizing the correspondence between Lie groups and Lie algebras.

In order to give a purely categorical definition of a principal bundle, we formulate this notion in the language of join restriction categories. Restriction categories were developed by Cockett and Lack to generalize partial maps (maps defined only on a subset of the domain) and have since then found applications in mathematics and computer science.
Join restriction categories, as described by Guo are restriction categories where local restrictions can be joined to obtain a global map.
Together with a manifold construction due to Grandis, that allows us to glue together objects, we can describe principal bundles entirely in the language of join-restriction categories.
} 

\maketitle
\tableofcontents 

\section{Introduction}

Tangent categories \cite{rosicky,Cockett2014DifferentialST} are categories equipped with an endofunctor and certain natural transformations which capture the behaviour of the tangent bundle in the category of smooth manifolds. They are a minimal setting for studying abstract differential geometry and, indeed, many aspects of differential geometry have been successfully reformulated in tangent categories. For example, vector fields, vector bundles \cite{cockett2016diffbundles,McAdam},  connections \cite{cockett2017connections}, de Rham cohomology \cite{Cruttwell2018}, and differential equations \cite{cockett2019diffeq}, all have been reformulated in tangent categories.

An important result in differential geometry which is studied in this paper is the correspondence between Lie groups and Lie algebras, described for example in \cite{duistermaat_2000}. One direction of this correspondence is the observation that a Lie group -- a group object in the category of manifolds and smooth maps -- induces a Lie algebra by considering the tangent space over the identity.  This is called the Lie algebra of the Lie group and it is classically used as a description of a Lie groups local structure.

The paper starts by introducing the various categorical structures used in the paper: tangent categories and join restriction categories.   This allows a general description of the Lie algebra of a group object in a tangent category.  The construction of the Lie algebra requires that the pullback of the projection of the tangent bundle of the group over the group unit exists. This relatively mild assumption suffices to allow one to prove the classical result that differential bundle of a group object is ``trivial'' in the sense of being the product of the Lie algebra and the group object, see Theorem \ref{thm:group_tangent}.

Classically there is an alternate description of the Lie algebra of a Lie group as the space of left-invariant vector fields of the Lie group (see Section II.4.11 of \cite{michor2008topics} or in Chapter 6 of \cite{Waldmann2021}).  In Proposition \ref{prop:left-invariant_isomorphic}, we show that this correspondence holds generally for group objects in any tangent category. 

Interestingly, we do not need to assume the existence of negatives in the tangent bundles:  negatives are a consequence of the group structure and the existence of the required pullback.  Thus, these classical results from differential geometry hold in great generality in any tangent category with the appropriate pullbacks.

Similar results for the tangent bundle of a group object in a tangent category have been developed in \cite{Loryaintablian2025differentiablegroupoidobjectsabstract}, however this paper takes a different direction by exploring the tangent bundle of group objects and principal bundles in restriction categories.

The second main objective of the paper is to develop the abstract theory of principal bundles.   These, as we explain, can be described in any join restriction category which permits the construction of manifolds. In classical differential geometry, given a Lie group $G$, a principal $G$-bundle is a map of manifolds $E \to M$ that is locally given by the projection $U \times G \to M$ for some open subset $U \subset M$, and where the gluing between the different local pieces is mediated by the action of the group itself. 

A principal $G$-bundle over a base space $M$ is a total space $E$ with a projection $E \xrightarrow{q} M$.  Examples of principal $G$-bundles include the trivial bundle, $M \times G \xrightarrow{\pi_0} M$ and also spaces which are twisted over the base as in the second example in Figure \ref{fig:twist}.

\begin{figure}
    \centering
    \includegraphics[width=0.3\linewidth]{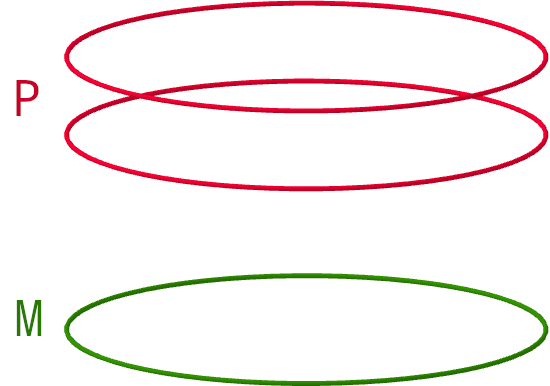}\qquad 
    \includegraphics[width=0.3\linewidth]{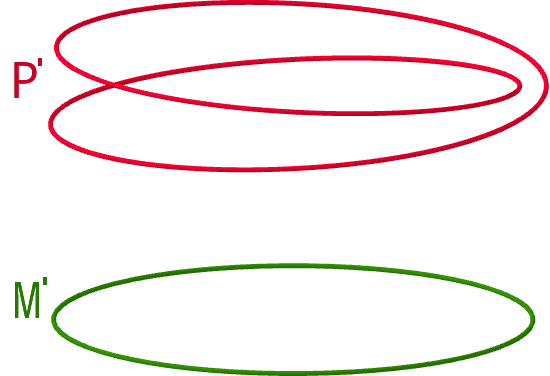}
    \caption{Two examples of principal $C_2$-bundles over $S^1$}
    \label{fig:twist}
\end{figure}

Locally, structure, such as that of $G$-bundles, may be expressed as a requirement on partial maps and, as such, one can use the theory of restriction categories \cite{cockett2002restriction} to explore this sort of structure.  Restriction categories provides a algebraic formulation of partial map categories.  A partial map is a map that is defined only on a subset of its domain.  These occur frequently and fundamentally in both analysis and geometry.  For example, the function $f(x)=\frac{1}{x}$ is not defined on all real numbers: it is a partial map $f: \mathbb{R} \to \mathbb{R}$.

Restriction categories were systematically developed in \cite{cockett2002restriction}, albeit the ideas themselves had a long provenance with origins in semigroup theory. Menger and his students  \cite{Menger} developed the basic idea of describing partial maps by what is here called a restriction. Subsequently, although independently, Di Paolo and Heller \cite{diP&H} used partiality to in developing a categorical approach to recursion theory. They based their formulation of partiality on the behavior of the product. Robinson and Rosolini \cite{robinson&rosolini} abstracted this notion of partiality based on the partial product by introducing p-categories. The first fully general formulations of restriction categories were given by Hans-Jürgen Hoehnke, who called them ``fractal categories" in \cite{Hoehnke_certain_fractal_categories}, and by Marco Grandis, who called them ``e-cohesive categories'' \cite{Grandis1990CohesiveCA}.  Significantly for the development below, Marco Grandis also developed the manifold construction for join restriction categories.

Join restriction categories were further studied in \cite{cockett&manes} and then in Xiuzhan Guo's thesis \cite{guo_joins_2012}. Join restriction categories are restriction categories in which morphisms that coincide on the intersection of their domains can be extended to a morphism defined on the union of their domains. This property allows the abstract construction of manifolds using atlases and gluing (see Section \ref{sec:res_cat} below).  Join restriction categories and their manifolds will be the basic setting in which we formulate the notion of a principal bundle.

The main objective of this work is to establish the theory of principal bundles in tangent categories. To achieve this we shall end up working in tangent join restriction categories (see Section \ref{ssec.tangent-join-restcats}).

The first results in this paper are about the tangent space of a group object in a Cartesian tangent category, significantly Theorem \ref{thm:group_tangent} establishes the triviality of the tangent bundle $T(G)$. Table \ref{tab:translations}, then describes all the structure maps of $T(G)$ in terms of the local triviality structure. Theorem \ref{thm:addition_and_multiplication_on_TG} establishes the interaction of tangent vector addition and group multiplication, and Theorem \ref{thm:Lie_Algebra} shows the external Lie Algebra structure on the tangent space.

The next main results pertain to principal bundles in restriction categories and restriction tangent categories, as defined in Definition \ref{def:G-bundle}. Theorems \ref{thm:right_action_principal}, \ref{thm:total_right_torsor} and \ref{thm:vertical_bundle} show that they fulfill the desired properties. Proposition \ref{prop:diffgeo_principal_bundles} and Proposition \ref{prop:topology_principal_bundles} show that they recover the classical notions from differential geometry and point-set topology.

\textbf{Acknowledgements.} For their ideas, comments, questions and discussions we would like to thank Geoff Vooys, Geoffrey Cruttwell and Rory Lucyshyn-Wright. Thanks go also to Kalin Krishna for pointing out the related work in \cite{Loryaintablian2025differentiablegroupoidobjectsabstract}. Special thanks goes to Kristine Bauer for her extensive writing support.
Florian Schwarz was funded by the Alberta Innovates Graduate Student Scholarship. Robin Cockett was partially funded by NSERC.

\section{Tangent categories}\label{sec:tan_cat}

 In differential geometry every manifold $M$ has a tangent bundle $T(M)$ which is also a manifold. It consists of a vector space over each point of the original manifold, this data can be packaged as a natural projection map $p:T(M) \rightarrow M$ where, for $m \in M$, $T(M)_m = p^{-1}(m)$ is the vector space of tangent vectors at the point $m$. 

 An axiomatization of tangent bundles was provided by \cite{rosicky} and then later refined by \cite{Cockett2014DifferentialST}. In this section we review tangent categories.

The behavior of the tangent bundle, $T^2(M) := T(T(M))$ is important  for axiomatizing the properties of the tangent bundle.  There are two important maps involving this iterated tangent bundle: the {\em vertical lift}, $\ell$, and the {\em canonical flip}, $c$.  For every vector in the tangent bundle one can consider its ``lift'' to the tangent bundle of the tangent bundle:
$$\ell: T(M) \rightarrow T(T(M)); (m,v)  \mapsto (m,0,0,v)$$
where here an element of $T(M)$ is regarded as a position $m$ and a vector $v \in T(M)_m$ in the space  over $m$.  The double tangent space $T(T(M))$ then can be described by four coordinates: a position in $T(M)$ and a vector of double the dimension.
Similarly, one can swap the tangent vectors of the iterated tangent bundle to obtain the ``flip'':
$$ c:T^2M \rightarrow T^2M; (m,v,w,u) \mapsto (m,w,v,u)$$
the definition of tangent categories is determined by the special properties of these natural transformations. 

In preparation for defining tangent categories it is useful to consider additive bundles:

\subsection{Basic definitions}
\begin{definition}[\cite{Cockett2014DifferentialST}, Definition 2.1]\label{def:additive_bundle}
Let $\mathbb{X}$ be a category and $A\in \mathbb{X}$ an object. An additive bundle over $A$ is
\begin{itemize}
    \item An object $X$ together with a morphism $p: X\rightarrow A$ such that pullback powers of $p$ exist.
    \item Morphisms $+ : X_2 \rightarrow X$ and $0: A \rightarrow X$ such that
    \begin{itemize}
        \item $+p= \pi_0 p = \pi_1 p$ and $0p= 1_A$
        \item the associative, commutative and unit diagrams below commute:
    \end{itemize}
\end{itemize}
{
\begin{tikzpicture}
\path (0,1.5) node(a) {$X_3$}
(2.5,1.5) node (b) {$X_2$}
(0,0) node (c) {$X_2$}
(2.5,0) node (d) {$X$};
\draw [->] (a) -- node[above] {$\langle \langle \pi_0, \pi_1\rangle+,\pi_2\rangle$} (b);
\draw [->] (a) -- node[right] {$\langle \pi_0 , \langle \pi_1 , \pi_2 \rangle + \rangle $} (c);
\draw [->] (c) -- node[below] {$+$} (d);
\draw [->] (b) -- node[right] {$+$} (d);
\end{tikzpicture}
\begin{tikzpicture}
\path (0,1.5) node(a) {$X_2$}
(0,0) node (c) {$X_2$}
(2.5,0) node (d) {$X$};
\draw [->] (a) -- node[above] {$+$} (d);
\draw [->] (a) -- node[left] {$\langle \pi_1 , \pi_0 \rangle $} (c);
\draw [->] (c) -- node[below] {$+$} (d);
\end{tikzpicture}
\begin{tikzpicture}
\path (0,1.5) node(a) {$X$}
(0,0) node (c) {$X_2$}
(2.5,0) node (d) {$X$};
\draw [->] (a) -- node[above] {$1_X$} (d);
\draw [->] (a) -- node[left] {$\langle p0 , 1_X \rangle $} (c);
\draw [->] (c) -- node[below] {$+$} (d);
\end{tikzpicture}
}\end{definition}

A morphism of additive bundles consists of morphisms which commute with this structure:

\begin{definition}[\cite{Cockett2014DifferentialST}, Definition 2.2]
For two additive bundles $X \xrightarrow{p} A$ and $Y \xrightarrow{q} B$ an additive bundle morphism is a pair of maps $f: X \rightarrow Y$ and $g: A \rightarrow B$ such that the following diagrams commute:
$$
\begin{tikzpicture}
\path (0,1.5) node(a) {$X$}
(2.5,1.5) node (b) {$Y$}
(0,0) node (c) {$A$}
(2.5,0) node (d) {$B$};
\draw [->] (a) -- node[above] {$f$} (b);
\draw [->] (a) -- node[left] {$p$} (c);
\draw [->] (c) -- node[above] {$g$} (d);
\draw [->] (b) -- node[left] {$q$} (d);
\end{tikzpicture}
\begin{tikzpicture}
\path (0,1.5) node(a) {$X_2$}
(2.5,1.5) node (b) {$Y_2$}
(0,0) node (c) {$X$}
(2.5,0) node (d) {$Y$};
\draw [->] (a) -- node[above] {$\langle \pi_0 f , \pi_1 f \rangle $} (b);
\draw [->] (a) -- node[left] {$+$} (c);
\draw [->] (c) -- node[above] {$f$} (d);
\draw [->] (b) -- node[left] {$+$} (d);
\end{tikzpicture}
\begin{tikzpicture}
\path (0,1.5) node(a) {$A$}
(2.5,1.5) node (b) {$B$}
(0,0) node (c) {$X$}
(2.5,0) node (d) {$Y$};
\draw [->] (a) -- node[above] {$g$} (b);
\draw [->] (a) -- node[left] {$0$} (c);
\draw [->] (c) -- node[above] {$f$} (d);
\draw [->] (b) -- node[left] {$0$} (d);
\end{tikzpicture}
$$
\end{definition}
Equipped with additive bundles, we can now define the notion of a tangent category.

\begin{definition}[\cite{Cockett2014DifferentialST}, Definition 2.3]\label{def:tan_cat}
A tangent category ($\mathbb X , T , p , 0, \ell , c$) consists of a category $\mathbb X$ and the following data:
\begin{itemize}
\item A functor $T: \mathbb X \rightarrow \mathbb X$, called the tangent functor with a natural transformation $p: T \rightarrow 1$ such that pullback powers of $p_M$ exist for all $M \in \mathbb X_0$ and $T$ preserves these pullback powers.
\item Natural transformations $+: T_2 \rightarrow T$ and $0: 1 \rightarrow T$ making each $p_M: TM \rightarrow M$ an additive bundle.
\item  A natural transformation $\ell: T \rightarrow T^2$, called the vertical lift, such that 
$$
(\ell_M,0_M):(p_M,+,0_M) \rightarrow (Tp_M,T(+_M),T(0_M)) 
$$
is an additive bundle morphism.
\item a natural transformation $c: T^2 \rightarrow T^2$ with $c^2 = 1$ and $lc = l$, such that 
$$
(c_M,1):(Tp_M,T(+_M),T(0_M)) \rightarrow (p_{TM},+_{TM},0_{TM})
$$
is an additive bundle morphism.
\end{itemize}
In addition the following diagrams must commute:
$$
\begin{tikzpicture}
\path (0,1.5) node(a) {$T$}
(2,1.5) node (b) {$T^2$}
(0,0) node (c) {$T^2$}
(2,0) node (d) {$T^3$};
\draw [->] (a) -- node[above] {$\ell$} (b);
\draw [->] (a) -- node[left] {$\ell$} (c);
\draw [->] (c) -- node[above] {$\ell_T$} (d);
\draw [->] (b) -- node[left] {$T(\ell)$} (d);
\end{tikzpicture}
\begin{tikzpicture}
\path (0,1.5) node(a) {$T^3$}
(2,1.5) node (b) {$T^3$}
(4,1.5) node (c) {$T^3$}
(0,0) node (d) {$T^3$}
(2,0) node (e) {$T^3$}
(4,0) node (f) {$T^3$};
\draw [->] (a) -- node[above] {$T(c)$} (b);
\draw [->] (b) -- node[above] {$c_T$} (c);
\draw [->] (d) -- node[above] {$T(c)$} (e);
\draw [->] (e) -- node[above] {$c_T$} (f);
\draw [->] (a) -- node[left] {$c_T$} (d);
\draw [->] (c) -- node[right] {$T(c)$} (f);
\end{tikzpicture}
\begin{tikzpicture}
\path (0,1.5) node(a) {$T^2$}
(2,1.5) node(e) {$T^3$}
(4,1.5) node (b) {$T^3$}
(0,0) node (c) {$T^2$}
(4,0) node (d) {$T^3$};
\draw [->] (a) -- node[above] {$\ell_T$} (e);
\draw [->] (e) -- node[above] {$T(c)$} (b);
\draw [->] (a) -- node[left] {$c$} (c);
\draw [->] (c) -- node[above] {$T(\ell)$} (d);
\draw [->] (b) -- node[left] {$c_T$} (d);
\end{tikzpicture}$$
and such that the following is an equalizer diagram:
\begin{center}
\begin{tikzpicture}
\path (-3,0) node (a) {$T_2M$}
(2,0) node (b) {$T^2M$}
(5,0) node (c) {$TM$};
\path (2.2,0.1) node (b0) {}
(4.8,0.1) node (c0) {};
\path (2.2,-0.1) node (b1) {}
(4.8,-0.1) node (c1) {};
\draw[->] (b0) -- node[above] {$T(p)$} (c0);
\draw[->] (b1) -- node[below] {$T(p)p 0$} (c1);
\draw[->] (a) -- node[above] {$u=\langle \pi_0 \ell , \pi_1 0_T\rangle T(+)$} (b);
\end{tikzpicture}
\end{center}
\end{definition}
Here the morphism $\langle \pi_0 \ell , \pi_1 0_T\rangle$ is the induced morphism determined by the universal property of pullbacks:
\begin{center}
\begin{tikzpicture}
\path (0,0) node (a) {$T_2M$}
(2,1.5) node (b0) {$TM$}
(2,-1.5) node (b1) {$TM$}
(6,1.5) node (c0) {$T^2M$}
(6,-1.5) node (c1) {$T^2M$}
(4,0) node (d) {$TT_2M$}
(8,0) node (e) {$TM$};
\draw[->] (a) -- node[above] {$\pi_0$} (b0);
\draw[->] (a) -- node[below] {$\pi_1$} (b1);
\draw[->] (b1) -- node[above] {$0_T$} (c1);
\draw[->] (b0) -- node[above] {$\ell$} (c0);
\draw[<-] (c0) -- node[below] {$~~T (\pi_0)$} (d);
\draw[<-] (c1) -- node[above] {$~~T (\pi_1)$} (d);
\draw[->,dashed] (a) -- node[above] {$u$} (d);
\draw[->] (c0) -- node[right] {$T(p)$} (e);
\draw[->] (c1) -- node[right] {$T(p)$} (e);
\end{tikzpicture}
\end{center}
Note that we can use the universal property of the pullback as $T$ preserves pullback powers. In order to show that the outer diagram commutes, we need to use that 
$$(\ell,0):(p,+,0) \rightarrow (Tp, T+, T0)$$
is a morphism of additive bundles. This ensures that it preserves the projection, that is following the diagram 
commutes:
\begin{center}
\begin{tikzpicture}
\path (0,1.5) node(a) {$TM$}
(2,1.5) node (b) {$T^2M$}
(0,0) node (c) {$M$}
(2,0) node (d) {$TM$};
\draw [->] (a) -- node[above] {$\ell$} (b);
\draw [->] (a) -- node[left] {$p$} (c);
\draw [->] (c) -- node[above] {$0$} (d);
\draw [->] (b) -- node[right] {$T(p)$} (d);
\end{tikzpicture}
\end{center}
Therefore we obtain $$\pi_0 \ell T(p) = \pi_0 p 0$$ and as $T_2M$ is the pullback of $p$ and $p$, $\pi_0 p = \pi_1p$ and thus 
$$\pi_0 \ell T(p) = \pi_0 p 0 = \pi_1 p 0$$
which together with the naturality of the $0$-map $p0 = 0 T(p)$ gives that the outer diagram commutes:
$$
\pi_0 \ell T(p) = \pi_0 p 0 = \pi_1 p 0 = \pi_1 0 T(p)
$$
Thus there is a unique induced morphism $u = \langle \pi_0 \ell , \pi_1 0_T\rangle $ which was used in the definition.

There are at least three different definitions for the functors between tangent categories. Strong morphism will be the most relevant in this development.

\begin{definition}[\cite{Cockett2014DifferentialST}, Definition 2.7] ~
\begin{enumerate}[(i)]
\item A morphism of tangent categories from $(\mathbb X, T, p , 0, \ell , c)$ to $(\mathbb X', T', p' , 0', \ell ' , c')$ is a functor $F: \mathbb X \rightarrow \mathbb X'$ and a natural transformation $\alpha: TF \rightarrow FT'$ such that the following diagrams commute:
\begin{center}
\begin{tikzpicture}
\path (0,1.5) node(a) {$TF$}
(2.5,1.5) node (b) {$FT'$}
(2.5,0) node (d) {$F$};
\draw [->] (a) -- node[above] {$\alpha$} (b);
\draw [->] (a) -- node[left] {$pF$} (d);
\draw [->] (b) -- node[left] {$Fp$} (d);
\end{tikzpicture}
\begin{tikzpicture}
\path (0,1.5) node(a) {$F$}
(0,0) node (c) {$TF$}
(2.5,0) node (d) {$FT'$};
\draw [->] (a) -- node[above] {$F0'$} (d);
\draw [->] (a) -- node[left] {$0F$} (c);
\draw [->] (c) -- node[above] {$\alpha$} (d);
\end{tikzpicture}
\begin{tikzpicture}
\path (0,1.5) node(a) {$T_2F$}
(2.5,1.5) node (b) {$FT'_2$}
(0,0) node (c) {$TF$}
(2.5,0) node (d) {$FT'$};
\draw [->] (a) -- node[above] {$\alpha_2$} (b);
\draw [->] (a) -- node[left] {$+F$} (c);
\draw [->] (b) -- node[left] {$F+'$} (d);
\draw [->] (c) -- node[above] {$\alpha$} (d);
\end{tikzpicture}
\end{center}
\begin{center}
\begin{tikzpicture}
\path (0,1.5) node(a) {$TF$}
(2.5,1.5) node (b) {$FT'$}
(0,0) node (c) {$T^2F$}
(2.5,0) node (d) {$FT'^2$};
\draw [->] (a) -- node[above] {$\alpha$} (b);
\draw [->] (a) -- node[left] {$\ell$} (c);
\draw [->] (b) -- node[left] {$F\ell'$} (d);
\draw [->] (c) -- node[above] {$(T\alpha) (\alpha T)$} (d);
\end{tikzpicture}
\begin{tikzpicture}
\path (0,1.5) node(a) {$T^2F$}
(2.5,1.5) node (b) {$FT'^2$}
(0,0) node (c) {$T^2F$}
(2.5,0) node (d) {$FT'^2$};
\draw [->] (a) -- node[above] {$(T\alpha) (\alpha T)$} (b);
\draw [->] (a) -- node[left] {$c$} (c);
\draw [->] (b) -- node[left] {$Fc'$} (d);
\draw [->] (c) -- node[above] {$(T\alpha) (\alpha T)$} (d);
\end{tikzpicture}
\end{center}
\item A morphism is called strong if $\alpha$ is invertible and $F$
 preserves equalizers and pullbacks of the tangent structure
\end{enumerate}
\end{definition}
\begin{definition}[\cite{Cockett2014DifferentialST}, Definition 2.8]
A \textbf{Cartesian tangent category} is a tangent category that has all finite products (including a final object), where the product functor
$\times: \mathbb X \times \mathbb X \to \mathbb X, (X,Y) \mapsto X \times Y$
with the natural isomorphism $\alpha: T(X) \times T(Y) \rightarrow T(X \times Y)$ forms a strong morphism of tangent categories.
\end{definition}
\subsection{Examples of Tangent categories}
\begin{example}[\cite{Cockett2014DifferentialST}, page 335]
\label{ex:manifolds_as_a_tangent_cat} The category \SmMan{} of smooth
manifolds motivated this definition, thus smooth manifolds should be expected to form a tangent category. In this case we have:
\begin{enumerate}[(i)]
    \item $p: T(M) \to M$ is the tangent bundle (and for maps the tangent map).
    \item $T_k(M)_m := p_{T_k(M)}^{-1}(m)= (p^{-1}_{T(M)}(m))^k =: (T(M)_m)^k$ has the $k^{\rm th}$-power of the tangent space $k$ times at each base point $m$.
    \item $0: M \rightarrow T(M) $ sends $m\in M$ to the tangent vector $0$ in the fibre $T(M)_m$.
    \item $+: T_2(M) \rightarrow T(M)$ has the fibre over $m$ given by: $$+_m: (T(M)_m)^2 \to T(M)_m; (v_m, w_m) \mapsto v_m + w_m$$
    \item As explained earlier, the vertical lift is given by: $$\ell:T(M) \rightarrow T(T(M)); (m,v) \mapsto (m,0,0,v)$$
    \item Also as explained earlier, the canonical flip is $$c: T(T(M)) \rightarrow T(T(M); (m,u,v,w) \mapsto (m,v,u,w).$$
\end{enumerate}
\end{example}

\begin{example}[Finite-dimensional $\mathbb{N}$-modules]\label{ex:modules_as_a_tangent_cat} A basic algebraic example is given by the category $\mathbb N^\bullet$, the category of finitely generated free commutative monoids. The objects are free finite dimensional $\mathbb{N}$-modules, $\mathbb{N}^k$ for $k \in \mathbb{N}$, and the morphisms are $\mathbb{N}$-linear maps.
This category then has tangent structure (it is a Cartesian differential category \cite{cart-diff}, which is a special kind of tangent category), given by:
\begin{enumerate}[(i)]
    \item $T(\mathbb N^k)= \mathbb N^k \times \mathbb N^k = \mathbb N^{2k}, T(f) = f \times f$, i.e. the tangent functor doubles the dimension.
    \item $p := \pi_0 : \mathbb{N}^k \times \mathbb{N}^k \to \mathbb{N}^k$ is the projection to the first component, which is the component holding the base point.
    \item $0:= \langle 1_\mathbb{N}^k , 0 \rangle: \mathbb{N}^k \to \mathbb N^k \times \mathbb{N}^k$ inserts the base point into the first component and zero into the second.
    \item $+:  \mathbb{N}^k \times \mathbb{N}^k \times \mathbb{N}^k \to \mathbb{N}^k \times \mathbb{N}^k$ is given by sending $(\vec u,\vec v,\vec w) \mapsto (\vec u, \vec v + \vec w)$.
    \item $\ell : (\mathbb{N}^k)^2 \to (\mathbb{N}^k)^4 $ is given by inserting zeros in the two middle components, i.e. $(\vec v, \vec w) \mapsto (\vec v, 0 ,0 , \vec w)$.
    \item $c : (\mathbb{N}^k)^4 \to (\mathbb{N}^k)^4$ is given by switching the two middle components, i.e. $(\vec t, \vec u , \vec v, \vec w) \mapsto (\vec t , \vec v, \vec u, \vec w)$.
\end{enumerate}

\begin{example}[Weil-algebras]
A more involved -- but important -- example is the category of Weil-algebras, $\weil$. The key building block of this category is the polynomial algebra $W:=  \mathbb{N}[x]/\langle x^2 \rangle$ called the dual numbers. We define product powers $W^n := \mathbb N [x_1 , ... , x_n]/\langle x_ix_j | i,j \in \{1, ... , n\} \rangle$ (with $W^0 := \mathbb{N}$) and coproducts $W^n \otimes W^k := \mathbb{N}[x_1, ... , x_n, y_1 , ... ,y_k] /\langle x_i^2, x_ix_j, y_i^2, y_iy_j \rangle$. Of particular significance  is the category $\weil_1$ whose objects are coproducts of product-powers of $W$. Morphisms are algebra-maps that preserve the constant terms (or augmentation) . More details on $\weil_1$ can be found in \cite{Leung2017}.

The category $\weil_1$ has a tangent structure given by the following:
\begin{enumerate}[(i)]
    \item $T(A) := A \otimes W , T(f) := f \otimes 1_\weil$, the tangent structure is given by tensoring with the dual numbers.
    \item $p:= 1_A \otimes \hat p: A \otimes W \to A \otimes \mathbb N = A$ is given by the identity of $A$ and the projection $\hat p : W \to N, a+bx \mapsto a$ that is the projection to the constant term of a polynomial.
    \item $0: = 1_A \otimes \hat 0 : A \to $ is given by the identity of $A$ and the projection $\hat{0}$ where $\hat{0}: \mathbb{N} \to W$ sends the natural number $n$ to the polynomial $n+0x$.
    \item $\ell := 1_A \otimes \hat \ell: A \otimes W \to A \otimes W \otimes W$ is given by the identity on $A$ and the map $\hat \ell: W \to W \otimes W ; a + bx \mapsto a+bxy$ (where $W \otimes W = \mathbb N[x,y]/\langle x^2 , y^2 \rangle $).
    \item $c := 1_A \otimes \hat{c}: A \otimes W \otimes W \to A \otimes W \otimes W$ is given by the identity of $A$ and the map $\hat{c}: W \otimes W \mapsto W \otimes W; a+bx + cy +dxy \mapsto a+ by+ cx+ dxy$.
\end{enumerate}
\end{example}
In \cite{Leung2017} it is shown that a tangent category can equivalently be defined as a category $\mathbb{X}$ with a strong monoidal functor $\hat T: \weil_1 \to  \mathrm{End}(\mathbb X)$ preserving certain pullbacks.
\end{example}


\section{Join restriction categories}
\label{sec:res_cat}

This section recalls the definition of join restriction categories. These will be the basis for our description of both fibre and principal bundles, see section \ref{sec:principal_bundles_part}.

Restriction categories, as defined in \cite{cockett2002restriction}, are a generalization of categories of partial maps. An important structure one can have in a restriction category is joins, which can be used to construct piecewise defined functions.


\subsection{Restriction categories}
 Restriction categories as introduced in \cite{cockett2002restriction}, capture the structure of partial maps. Here we review the definition and some results on restriction categories which will be needed in the sequel.
 
\begin{definition}[\cite{cockett2002restriction}, section 2.1.1, \cite{guo_joins_2012}, Definition 3.12 and \cite{guo_joins_2012}, Lemma 1.6.3] \label{def:restriction_categories} ~
\begin{enumerate}[(i)]
\item A \textbf{restriction category} is a category with a structure that assigns to each morphism $f:A\to B$ a morphism $\bar f : A \rightarrow A$ such that: 
\begin{center}
\begin{tabular}{llcll}
 {\bf [R.1]} & $\bar f f =f $ & ~ & {\bf [R.3]} & $\overline {{\bar g} f} = \bar g \bar f$ \\
 {\bf [R.2]} & $\bar g \bar f = \bar f \bar g$ & ~ & {\bf [R.4]} & $f \bar h = \overline{f h} f$
 \end{tabular}
 \end{center}
 
\item Two parallel morphisms $f,g: X \to Y$ in a restriction category are \textbf{compatible}, written $f \smile g$ if $\bar f g = \bar g f$
\item Two parallel morphisms in a restriction category are partially ordered by $f \leq g$ if and only if $\bar f g = f$.
\item A functor $F: \mathbb X \to \mathbb Y$ between restriction categories $\mathbb X$ and $\mathbb Y$ is a \textbf{restriction functor} if it preserves the restriction structure, i.e. $F(\bar f) = \overline{F(f)}$ for all morphisms $f$ in $\mathbb X$
\end{enumerate}
\end{definition}
\begin{lemma}
Let $f$ and $g$ be morphisms in a restriction category, then 
\begin{center}
\begin{tabular}{ll}
    (i)~~~ $\bar f \bar f = \bar f$ & ~~~~(ii)~~  $\bar {\bar f}= \bar f$ \\
    (iii)~~ $\overline{f \bar g} = \overline{fg}$ &~~~ (iv)~~~ $\bar f \smile \bar g$ \\
    (v)~~ $f \leq g \Rightarrow f \smile g$. & ~ 
\end{tabular}
\end{center}
\end{lemma}
The first point means that the restriction of a morphism is idempotent, thus, morphisms of the form $\bar f$ are called \textbf{restriction idempotent}.
The second point, $\overline{\overline{g}} = \bar g$, means that a repeated restriction is the same as a single restriction.

Equations (i)-(iii) are proven as part of Lemma 2.1 in \cite{cockett2002restriction} and (iv) and (v) are proven as part of Lemma 3.1.3 in \cite{guo_joins_2012}. We will still prove them here, as the proofs demonstrate how to work with restriction categories.
\begin{proof}~
\begin{enumerate}[(i)]
    \item $\bar f \bar f \overset{\bf [R.3]}{=} \overline{\bar f f} \overset{\bf [R.1]}{=} \bar f$.
    \item This follows from {\bf [R.3]} if $f=1$.
    \item[(iv)] $\overline{f \bar g} \overset{\bf [R.4]}{=} \overline{\overline{fg}f} \overset{\bf [R.3]}{=} \overline{fg} \bar f \overset{\bf [R.2]}{=} \bar f \overline{fg}  \overset{\bf [R.3]}{=} \overline{\bar f f g} \overset{\bf [R.1]}{=} \overline{fg}$.
    \item[(iv)] $
    \overline{f \bar g} \overset{\bf [R.4]}{=}
    \overline{\overline{fg}f} \overset{\bf [R.3]} , {\bf [R.2]|}{=}
    \bar f \overline{fg} \overset{\bf [R.3]}{=}
    \overline{\bar f f g} \overset{\bf [R.1]}{=}
    \overline{fg}$
    and thus $\overline{f \bar g} = \overline{fg}$. This implies that
    $\overline {\overline f} \bar g \overset{(i)}{=}  \bar f \bar g  \overset{\bf [R.2]}= \bar g \bar f
    \overset{(i)}{=} \bar f \overline {\overline g}$
    and thus $\bar f \smile \bar g$.
    \item[(v)] $
    \bar f g \overset{\bf [R.1]}{=} \bar f \bar g g \overset{\bf [R.2]}{=} \bar g \bar f g \overset{f \leq g}{=} \bar g f$
    and thus $f \smile g$.
\end{enumerate}
\end{proof}

\begin{remark}\label{rmk:partial_set_map_interpretations}
The motivating example of a restriction category is the category of sets and partial maps. The objects are sets and the morphisms $A \to B$ are set maps of the form $U \to B$ where $U \subset A$ is a subset of $A$. Given a partial map $f: A \to B$, its restriction idempotent $\bar f$ is defined as the inclusion $U \hookrightarrow A$ of the domain of definition $U$ into the domain $A$ of $f$. We often interpret a restriction idempotent to describe the subset $U \subset A$. In this example 
\begin{enumerate}[(i)]
    \item $\bar f g$ is the restriction of $g$ to $f$'s domain of definition,
    \item and $\overline{fg}$ describes the preimage $f^{-1}(V)$ of $g$'s domain of definition $V$ under $f$.
\end{enumerate}
With these interpretations in mind, the properties [R.1]-[R.4] are just common properties of partial maps.
\end{remark}

As shown in section 2.1.3 of \cite{cockett2002restriction}, restriction categories generalize partial map categories. The corresponding notion of being a total morphism and a partial isomorphism need to be described:
\begin{definition}\label{def:total_map_partial_inverse} ~
\begin{enumerate}[(i)]
    \item A \textbf{total morphism} is a morphism $f: A \to B$ such that $\bar f = 1_A$
    \item A \textbf{partial isomorphism} is a morphism $f:A\to B$ such that there is a morphism $f^*:B \rightarrow A$ such that $ff^* =\bar f$ and $f^*f = \bar {f^*}$. The morphism $f^*$ is called the \textbf{partial inverse} of $f$.
\end{enumerate}
\end{definition}
Partial inverses are unique and will be denoted by $f^*$.

Given a restriction category $\mathbb X$ one can take its full subcategory of total morphisms and obtain a category $\mathrm{Tot}(\mathbb X)$. On the other hand, given a category $\mathbb Y$ one can see it as a restriction category $\mathrm{Triv}(\mathbb X)$ by defining $\bar f = 1$ for all morphisms $f$. 

\begin{proposition}
The assignments 
 $\mathrm{Tot}:\sf{ResCat} \to \sf{Cat}$ and $\mathrm{Triv}:\sf{Cat} \to \sf{ResCat}$ are functors and $\mathrm{Tot}$ is right-adjoint to the inclusion $\mathrm{Triv}$. Thus the 1-category $\sf{Cat}$ of categories is a coreflective subcategory of $\sf{ResCat}$.
\end{proposition}
\begin{proof}
In order to obtain functors we first need to define $\mathrm{Tot}$ and $\mathrm{Triv}$ on morphisms. Let $F: \mathbb X \to \mathbb Y$ be a restriction functor. Then as restriction functors preserve totals we can define $\mathrm{Tot}(F):= F|_{\mathrm{Tot}(X)}:\mathrm{Tot}(\mathbb X) \to \mathrm{Tot}(\mathbb Y)$. As $\mathrm{Tot}(F)$ it is a restriciton of $F$ to a subcategory, $\mathrm{Tot}$ is automatically functorial.

Similarly, given a functor $F: \mathbb X \to \mathbb Y$ of ordinary categories, we can see it as a restriction functor $\mathrm{Triv}(F): \mathrm{Triv}(\mathbb X) \to \mathrm{Triv}(\mathbb Y)$. The functor $\mathrm{Triv}$ is automatically functorial as it does not change the underlying assignment of objects and morphisms.

Next we show the adjunction by showing that there is an isomorphism between the hom-sets
$$
\sf{ResCat}(\mathrm{Triv}(\mathbb X), \mathbb Y) \cong \sf{Cat}( \mathbb X , \mathrm{Tot}(\mathbb Y))
$$
that is natural in functors $\mathbb X \to \mathbb X'$ and restriction functors $\mathbb Y \to \mathbb Y'$.
Given a restriction functor $F:\mathrm{Triv}(\mathbb X) \to \mathbb Y$, this functor preserves total morphisms. As every morphism in $\mathrm{Triv}(\mathbb X)$ is total this means that every morphism in the image of the functor is total. Thus the assignment of objects and morphisms that is $F$ forms a functor $\varphi(F): \mathbb X \to \mathrm{Tot}(\mathbb Y)$ that sends an object $X$ of $\mathbb X$ to $F(X)$ and a morphism $f$ of $X$ to $F(f)$.

As the morphism $\varphi: \sf{ResCat}(\mathrm{Triv}(\mathbb X), \mathbb Y) \to \sf{Cat}( \mathbb X , \mathrm{Tot}(\mathbb Y))$ does not change the assignments of objects and morphisms, it is automatically natural and injective. It is surjective because every functor $\mathbb X \to \mathrm{Tot}(\mathbb Y)$ can be seen as a restriction functor.
\end{proof}

\subsection{Limits and products}
As restriction categories are enriched in partial orders, it turns out to be natural to consider lax limits:
\begin{definition}[\cite{cockett_lack_2007}, section 4.4]\label{def:restriction_limit}
Let $\mathbb X$ be a restriction category, $\mathcal C$ a finite category and $S:\mathcal C \to \mathbb X$ a diagram. 
\begin{enumerate}[(i)]
    \item A lax cone of the diagram $S$ is an object $L$ of $\mathbb X$ together with (for every object $C$ of $\mathcal C$ a morphism $q_C:L\to S(C)$) such that for every morphism $c\in \mathrm{Hom}_\mathcal{C}(C, D)$, $S(c)q_C \leq q_D$.
    \item A restriction limit is a lax cone $(L,p_C)$ with total morphisms $p_C$ such that for every other cone $(M,q_C)$ there is a unique morphism $f:M \to L$ such that for all objects $C$ of  $\mathcal C$,
$fp_C = e q_C$ holds, where $e=\prod_{C \in \mathcal C_0}\bar q_C$.
\end{enumerate}
\end{definition}
An example for this definition is:
\begin{example}
Let $\mathbb X$ be the restriction category with topological spaces as objects and partial continuous maps defined only on closed subsets as morphisms. Then the pullback of the maps $i_{[0,1]},i_{[-1,0]}:\mathbb R \to \mathbb R$ given by the inclusion of $[0,1]$ and the inclusion of $[-1,0]$ is the object $\{0\}$ with the projections $p_0,p_1$ embedding it into $\mathbb R$ as zero.
\begin{center}
\begin{tikzpicture}
\path 
(-1.5,3) node (A) {$\mathbb R$}
(2,2) node (p) {$\mathbb R$}
(0,2) node(z) {$\{0\}$}
(0,0) node (m) {$\mathbb R$}
(2,0) node (R) {$\mathbb R$}
;
\draw [->] (A) -- node[above] {$q_0 = i_0$} (p);
\draw [->] (A) -- node[left] {$q_1 = i_{[0,1]}$} (m);
\draw [->] (p) -- node[right] {$i_{[0,1]}$} (R);
\draw [->] (m) -- node[below] {$i_{[-1,0]}$} (R);
\draw [->] (z) -- node[below] {$p_0=i_0$} (p);
\draw [->] (z) -- node[right] {$p_1=i_0$} (m);
\draw [->,dashed] (A) -- node[right] {} (z);
\end{tikzpicture}
\end{center}
Now let us consider another cone $(\mathbb R, q_0, q_1)$ with $q_0:\mathbb R \to \mathbb R$ sending 0 to 0 and not defined anywhere else and $q_1: \mathbb R \to \mathbb R$ embedding the interval $[0,1]$. Then the unique map $f:\mathbb R \to 0$ fulfilling the conditions is the map sending 0 to 0 and that is nowhere else defined. 

This is an example where $f p_1 \neq q_1$ as $q_1$ is defined on $[0,1]$ whereas $f p_1$ is defined only on $\{0\}$. However, as $e$ is the restriction to $\{0\}$, $f p_1 = q_1 e$ holds.
\end{example}
In order to have a better grasp on restriction pullbacks, we will expand that definition in the special case of pullbacks of total morphisms. We will heavily rely on this Lemma in the proof of Theorem \ref{thm:vertical_bundle}.
\begin{lemma}\label{lem:restriction_pullback_of_total}
    Let $A \xrightarrow{f} B \xleftarrow{g} C$ be a cospan of total morphisms in a restriction category. Let $X$ be an object and let $p_A: X \to A$ and $p_C: X \to C$ be total morphisms such that the diagram
    \[\begin{tikzcd}
    	X & A \\
    	C & B
    	\arrow["{p_A}", from=1-1, to=1-2]
    	\arrow["{p_C}"', from=1-1, to=2-1]
    	\arrow["f", from=1-2, to=2-2]
    	\arrow["g"', from=2-1, to=2-2]
    \end{tikzcd}\]
    commutes, i.e. $p_A f = p_C g$. Then the following are equivalent:
    \begin{enumerate}[(i)]
        \item For any object $Z$ with morphisms $r_A: Z \to A$, $r_C:Z \to C$ such that the diagram
        \[\begin{tikzcd}
        	Z \\
        	& X & A \\
        	& C & B
        	\arrow["{r_A}", curve={height=-6pt}, from=1-1, to=2-3]
        	\arrow["{r_C}"', curve={height=6pt}, from=1-1, to=3-2]
        	\arrow["{p_A}", from=2-2, to=2-3]
        	\arrow["{p_C}"', from=2-2, to=3-2]
        	\arrow["f", from=2-3, to=3-3]
        	\arrow["g"', from=3-2, to=3-3]
        \end{tikzcd}\]
        commutes (i.e. $r_Af = r_Cg$), there is a unique morphism $\phi: Z \to X$ with $\phi p_A = r_A$ and $\phi p_C = r_C$.
        \item The cone $(X,(q_A,q_Af,q_C))$ is the restriction pullback of the cospan $A \xrightarrow{f} B \xleftarrow{g} C$.
    \end{enumerate}
\end{lemma}
The key difference between Lemma \ref{lem:restriction_pullback_of_total}(i) and Definition \ref{def:restriction_limit} is that in Lemma \ref{lem:restriction_pullback_of_total}(i) equalities hold whereas in Definition \ref{def:restriction_limit} cones commute only up to inequality and $f p_C$ differs from $q_C$ by a restriction idempotent. Thereby condition (i) is easier and closer to the classical definition of a limit.
\begin{proof}
    The implication $(ii) \Rightarrow (i)$ follows from Definition \ref{def:restriction_limit} since, due to the totality of $f$ and $g$, $\bar r_A = \bar r_C$ and thus $\phi p_A = \bar r_C r_A = r_A$.

    For the implication $(i) \Rightarrow (ii)$, suppose $(i)$ and a lax cone $(Z,(q_A,q_B,q_C))$. Since $f$ is total and $(Z,(q_A,q_B,q_C))$ is a cone, we know that $q_Af \leq q_B$ which implies $\bar q_A \leq \bar q_B $ and analogously $\bar q_C \leq \bar q_B$.
    Define $r_A = \bar q_C q_A$ and $r_C = \bar q_A q_C$. Since $(Z,(q_A,q_B,q_C))$ was a lax cone, they fulfill that
    $$
    r_A f = \bar q_C q_A f = \bar q_C \bar q_A q_B = \bar q_A q_C g = r_C g .
    $$
    Thus, by the assumption $(i)$, there is a unique morphism $\phi: Z \to X$ with $\phi p_A = r_A$ and $\phi p_C = r_C$. From this it follows that $\phi$ is a unique morphism $Z \to X$ fulfilling
    $$
    \phi p_A \!=\!\bar q_C q_A \!=\! \bar q_C \bar q_B \bar q_A q_A , \quad \phi p_A f \!=\! \bar q_C q_A f \!=\! \bar q_C \bar q_B \bar q_A q_B \quad  \text{and} \quad \phi p_C \!=\! \bar q_A q_C \! = \!\bar q_C \bar q_B \bar q_A q_C .
    $$
    Thus the cone $(X,(q_A,q_Af,q_C))$ fulfills Definition \ref{def:restriction_limit} proving the implication $(i) \Rightarrow (ii)$.
\end{proof}

Using restriction limits we can define additional structures that we will use in section \ref{sec:principal_bundles_part}.
A \textbf{Cartesian restriction category} is a restriction category that has all restriction products and a restriction terminal object.
A \textbf{restriction tangent category} is a restriction category that has a tangent category structure where all the limits are restriction limits. It will be defined in more detail in Definition \ref{def:tan_res_cat}.
\subsection{Joins, atlases and gluings}
Defined in \cite{guo_joins_2012},
a \textbf{join restriction category} $\mathbb X$ is a restriction category such that for each $A,B \in X$ and each set $S\subset \mathbb X(A,B)$ of compatible morphisms there is a join (least upper bound, or supremum)
$$
\bigvee_{s \in S} s \in \mathbb X(A,B)
$$
which fulfills the following properties
\begin{enumerate}[(i)]
\item $\overline{\bigvee_{s \in S} s} = \bigvee_{s \in S} \bar s$
\item $g(\bigvee_{s \in S} s) = \bigvee_{s \in S} (gs)$
\item it is a join (supremum) with respect to the partial ordering $\leq$, i.e. 
$f_i \leq g \forall\iota\Rightarrow \bigvee f_i \leq g$
\end{enumerate}

\begin{remark}\label{rmrk:commutative_diagrams_in_restriction}
We will sometimes draw commutative diagrams in restriction categories. When we say that a diagram commutes, that means that the compositions are equal as morphisms in the restriction category. In particular, their domains of definition coincide.
For example, when the diagram 
\[
\begin{tikzcd}
	A && B \\
	C && D
	\arrow["f", from=1-1, to=1-3]
	\arrow["h"', from=1-1, to=2-1]
	\arrow["g", from=1-3, to=2-3]
	\arrow["k"', from=2-1, to=2-3]
\end{tikzcd}
\]
commutes, it means that $fg=hk$, implying that both $fg \smile hk$ and $\overline{fg} = \overline{hk}$.
\end{remark}

In classical differential geometry, manifolds are described by atlases which consist of local neighbourhoods and instructions for how to glue them together. Now in order to describe objects that are built piecewise out of multiple components in a join restriction category, we define atlases to consist of objects $(U_i)_{i \in I}$ as components and morphisms $u_{ij}$ as instructions how to glue $U_i$ and $U_j$ together. This is known as the manifold construction that was introduced by Grandis in \cite{Grandis1990CohesiveCA} and specified for join restriction categories by Cockett and Cruttwell in \cite{Cockett2014DifferentialST}.
\begin{definition}[\cite{Grandis1990CohesiveCA},\cite{Cockett2014DifferentialST}]\label{def:atlas}
In a join restriction category, an \textbf{atlas} is a collection of objects $\{U_i\}_{i \in I}$ and morphisms $\{u_{ij}\}_{i,j \in I}$ with $u_{ij}: U_i \to U_j$, fulfilling 
\begin{enumerate}[(i)]
\item $u_{ii}u_{ij} = u_{ij}: U_i \to U_j$
\item $u_{ij}u_{jk}\leq u_{ik}:U_i \to U_k$
\item $u_{ij}u_{ji} = \overline{u_{ij}}: U_i \to U_i$.
\end{enumerate}
\end{definition}
\begin{definition}[\cite{Grandis1990CohesiveCA},\cite{Cockett2014DifferentialST}] \label{def:atlas-morphism}
An atlas morphism $(U_i,u_{ij}) \to (V_k,v_{kl})$ is a family of morphisms $A_{ik}:U_i \to V_k$ such that
\begin{enumerate}[(i)]
\item $u_{ii}A_{ik} = A_{ik}:U_i \to V_k$
\item $A_{ik}v_{kk} = A_{ik}:U_i \to V_k$
\item $u_{ij}A_{jk} \leq A_{ik}: U_i \to V_k$
\item $A_{ik}v_{kl} \leq A_{il}: U_i \to V_l$
\item $A_{ik}v_{kl} =\overline{A_{ik}}A_{il}: U_i \to V_l$.
\end{enumerate}
\end{definition}
Condition \ref{def:atlas-morphism}.(iv) is actually a direct consequence of condition \ref{def:atlas-morphism}.(v). 
A trivial example of an atlas is obtained by choosing an object $U=U_0$, taking the index-set to be $\{0\}$ and $u_{00}=1_U$. We will denote this atlas as $\mathrm{At}(1_U)$

An atlas is a specification of how to glue things together. The gluing construction from \cite{Grandis1990CohesiveCA} and \cite{Cockett2014DifferentialST} makes this possible.
\begin{definition}\label{def:gluing}
Let $(U_i,u_{ij})$ be an atlas in a join restriction category $\mathbb X$. Then a \textbf{gluing} of it is an object $G$ together with an atlas morphism $g:(U_i,u_{ij}) \to \mathrm{At}(1_G)$ that fulfills a universal property: for any object $X$ and atlas morphism $f: (U_i,u_{ij}) \to \mathrm{At}(1_X)$ there is a unique morphism $\check f: G \to X$ of $\mathbb X$ such that $f_i = g_i \check f$ and if $f_i\leq f_i'$ then $\check f \leq \check f'$.
\end{definition}

Equivalently but more explicitly, the gluing can be characterized by the components of the atlas morphism $g$. Since the target has only one object, the components $A_{ik}$ (as in Definition \ref{def:atlas-morphism}) can be written as $g_i$ with only one index.

\begin{remark}\label{rmk:gluing_explicit} A gluing for $\{u_{ij}\}$ is an object $G_U$ with a family of partial isomorphisms $g_i: U_i \to G_U$ (with partial inverses $g_i^*$) such that 
\begin{enumerate}[(i)]
    \item $u_{ij}g_{j}\leq g_i$,
    \item $u_{ij} = g_ig_j^*$, and
    \item $1_{G_U} = \bigvee_i g_i^* g_i$.
\end{enumerate}
\end{remark}

In fact condition $(i)$ is redundant and follows from condition $(ii)$ since $u_{ij}g_j = g_i g_j^* g_j = g_i \bar g_j \leq g_i$.

\subsection{Tangent join restriction categories} \label{ssec.tangent-join-restcats}
The category of smooth manifolds and smooth maps was the motivating example for tangent categories. This category has total maps as morphisms. The category of smooth manifolds and smooth partial maps, defined on open subsets, is a restriction category. It is the motivating example of tangent restriction categories which encode both the local structure and the tangent structure of manifolds.
Tangent join restriction categories combine the tangent category structure from Section \ref{sec:tan_cat} with the join restriction category structure from Section \ref{sec:res_cat}.
This definition was given in \cite{Cockett2014DifferentialST}.
\begin{definition}[Definition 6.14 of \cite{Cockett2014DifferentialST} ]
\label{def:tan_res_cat} ~ \\
A \textbf{tangent restriction structure} for a join restriction category $\mathbb X$ consists of a restriction preserving functor $T: \mathbb X \to \mathbb X$ and associated total transformations such that 
\begin{itemize}
\item for each $M \in \mathbb X$, we have total morphisms $p_M: TM \to M$, restriction pullbacks $T_n(M)$ and total morphisms $+_M: T_2M \to TM, 0_M : M \to TM$ such that for each $f:M \to N$, the pair $(Tf,f)$ is an additive bundle morphism;
\item for each $n,k \in \mathbb N, T^n$ preserves the restriction pullbacks of $k$ copies of $p$.
\item there is a total natural transformation $T \xrightarrow{\ell} T^2$ such that for each $M$, the pair $(\ell_M, 0_M)$ is an additive bundle morphism from $(Tp:T^2M \to TM)$ to $(p_T: T^2M \to TM)$;
\item there is a total natural transformation $T^2 \xrightarrow{c} T^2$ such that for eacch $M$ the pair $(c_M,1)$ is an additive bundle morphism from $(Tp: T^2M \to TM)$ to $(p_T: T^2M \to TM)$;
\item we have $c^2 = 1, \ell c = \ell$ and the following diagrams commute:
\[\begin{tikzcd}[column sep=scriptsize]
	T & {T^2} && {T^3} & {T^3} & {T^3} && {T^2} & {T^3} & {T^3} \\
	{T^2} & {T^3} && {T^3} & {T^3} & {T^3} && {T^2} && {T^3}
	\arrow["\ell", from=1-1, to=1-2]
	\arrow["\ell"', from=1-1, to=2-1]
	\arrow["{\ell_T}"', from=2-1, to=2-2]
	\arrow["{T(\ell)}", from=1-2, to=2-2]
	\arrow["{T(c)}", from=1-4, to=1-5]
	\arrow["{c_T}", from=1-5, to=1-6]
	\arrow["{T(c)}", from=1-6, to=2-6]
	\arrow["{c_T}"', from=2-5, to=2-6]
	\arrow["{T(c)}"', from=2-4, to=2-5]
	\arrow["{c_T}"', from=1-4, to=2-4]
	\arrow["{\ell_T}", from=1-8, to=1-9]
	\arrow["{T(c)}", from=1-9, to=1-10]
	\arrow["{c_T}", from=1-10, to=2-10]
	\arrow["c"', from=1-8, to=2-8]
	\arrow["{T(\ell)}"', from=2-8, to=2-10]
\end{tikzcd}\]
\item the diagram
\[\begin{tikzcd}
	{T_2M} && {T^2M} \\
	M && TM
	\arrow["{\langle \pi_0 \ell , \pi_10_T \rangle T(+)}", from=1-1, to=1-3]
	\arrow["{T(p)}", from=1-3, to=2-3]
	\arrow["{0_M}"',from=2-1, to=2-3]
	\arrow["{\pi_0p = \pi_1p}"', from=1-1, to=2-1]
\end{tikzcd}\]
is a restriction pullback.
\end{itemize}
If $\mathbb X$ has restriction products and $T$ preserves them, then $(\mathbb X, T)$ is a \textbf{Cartesian tangent join restriction category}.
\end{definition}
The motivating example for a cartesian tangent join restriction category is the category of smooth manifolds and partial smooth maps. We will analyze this example further in section \ref{subsubsec:classical}.

If $(\mathbb X,T)$ is any tangent category it can be turned into a tangent restriction category by defining $\bar f = 1_A$ for all morphisms $f:A \to B$ in $\mathbb X$.

If $\mathbb X$ is any restriction category it can be turned into a tangent restriction category by choosing the tangent structure $T := 1_\mathbb X$ to be the identity.

\begin{proposition}
Let $(\mathbb X, T)$ be a tangent restriction category. Then $(\mathrm{Tot}(\mathbb X), \mathrm{Tot}(T)) $ is an ordinary tangent category.
\end{proposition}
\begin{proof}
Let $(\mathbb X, T, p, 0, + , \ell, c)$ be a tangent restriction category.
We proceed by checking that $\tot(\mathbb X), \tot (T), , p, 0, + , \ell, c)$ fulfills the axioms in Definition \ref{def:tan_cat}. As $T$ was a restriction functor, $\tot (T)$ is an endofunctor $\tot(\mathbb X) \to \tot(\mathbb X)$. Since $p,0,+,\ell, c$ were total morphisms in $\mathbb X$ they are morphisms in $\tot (\mathbb X)$. For every object $X \in \mathrm{Obj}(\mathbb X)= \mathrm{Obj}(\tot (\mathbb X))$ pullback powers of $\tot(T)(X) \xrightarrow{p} X$ exist and are given by
$
\tot(T)_n(X) = T_n(X)
$
because restriction limits of total morphisms are limits in $\tot(\mathbb X)$ according to section 4.4 of \cite{cockett_lack_2007}. All the equalities hold in $\tot(\mathbb X)$ as they hold in $\mathbb X$. The universality of the vertical lift holds because restriction limits of total morphisms in $\mathbb X$ become limits in $\tot (\mathbb X)$.
\end{proof}

Interestingly we did not need to ask for the tangent functor to be compatible with joins, only with restrictions. The reason is the following lemma which says that compatibility with the restriction implies compatibility with the join for a tangent structure. We will use it later in Section \ref{sec:vertical_bundle}.

\begin{lemma}\label{lem:join_and_tangent}~
\begin{enumerate}[(i)]
    \item 
    Let $(\mathbb X , T , p , 0, \ell , c)$ be a Cartesian tangent join restriction category. Then for any parallel morphisms $f,g$ in $\mathbb X$, $f \leq g$ if and only if $T(f) \leq T(g)$ .
    \item Let $(\mathbb X , T , p , 0, \ell , c)$ be a Cartesian tangent join restriction category. Then for any parallel morphisms $f,g$ in $\mathbb X$,
    $f \smile g$ if and only if $T(f) \smile T(g)$.
    \item Let $(\mathbb X , T , p , 0, \ell , c)$ be a Cartesian tangent join restriction category and let $(f_i)_{i \in I}$ be a compatible family of morphisms in $\mathbb X$. Then
    $$
    \bigvee_{i \in I} T(f_i) =  T\left(\bigvee_{ i \in I}f_i \right).
    $$
\end{enumerate}
\end{lemma}
\begin{proof}~
\begin{enumerate}[(i)]
    \item If $f\leq g$, that means that $\bar f g = f$. Thus, applying $T$ to the equation we obtain that 
    $$
    \overline{T(f)}T(g) = T(\bar f)T(g) = T(\bar f g) = T(f)
    $$
    which means that $T(f) \leq T(g)$. Conversely, if $T(f)\leq T(g)$, that means that $\overline{T(f)}T(g) = T(f)$. Due to naturality of $p$, the equation $f = 0T(h)p$ holds for any morphism $h$. Thus we obtain that
    $$
    \bar f g = 0 T(\bar f g) p = 0 \overline{T(f)} T(g) p = 0 T(f) p = f
    $$
    which means that $f \leq g$.
    \item This is analogous to (i).
    \item Due to (i), $f_i \leq \bigvee_{i \in I} f_i$ implies that $F(f_i) \leq F (\bigvee_{i \in I} f_i)$. Thus, since the join is a supremum, this means that 
    $$
    \bigvee_{i \in I} F(f_i) \leq F \left(\bigvee_{i \in I} f_i\right).
    $$
    Next we will use that the projection $p:T \to 1_\mathbb X$ is a natural transformation with total components, thus post-composition with it does not change the restriction idempotent. Thus,
    $$
    \overline{F\left(\bigvee_{i \in I} f_i \right)} = \overline{F\left(\bigvee_{i \in I} f_i \right)p} = \overline{p \bigvee_{i \in I} f_i} = \bigvee_{i \in I} \overline{p f_i} = \bigvee_{i \in I} \overline{F(f_i) p} = \overline{\bigvee_{i \in I} F(f_i)}
    $$
    holds. This implies the required equation.
\end{enumerate}
\end{proof}
Note that the inequality $\bigvee_{i \in I} F(f_i) \leq F \left(\bigvee_{i \in I} f_i\right).$ holds for any restriction preserving functor, while equality requires the projection natural transformation.

\section{The tangent bundle of a group object}\label{sec:The tangent bundle of a group object}
In classical Lie theory, the tangent bundle of a Lie group $G=(G,\cdot, u)$ is trivial in the sense that $T(G) = G \times \mathfrak{g}$, where $\mathfrak g = T(G)_u$ is the Lie algebra of $G$.
In this section we will show that the tangent bundle of a group object in any tangent category is also trivial. This recovers the classical result, but is more general, since it now holds in many different categories.

While Section \ref{sec:principal_bundles_part} is about join restriction categories, this section does not assume join restriction structures. The results pertain to Cartesian tangent categories in general.
Let us start by stating the definition of a group object:
\begin{definition}\label{def:group_object}
A \textbf{group object} $(G,m,u,\iota)$ in a Cartesian category $\mathbb X$ is an object $G$ together with morphisms $m : G \times G \to G$, $u: 1 \to G$ and $\iota: G \to G$ (where $1$ is the terminal object) such that
\begin{enumerate}[(i)]
    \item $(m\times 1_G)m = (1_G \times m) m \quad :\quad  G \times G \times G \to G$
    \item $\langle !u , 1_G \rangle m = 1_G = \langle 1_G, !u \rangle m \quad :\quad  G \to G$
    \item $\langle \iota,1_G \rangle m = !u = \langle 1_G,\iota \rangle m  \quad :\quad  G \to G$,
\end{enumerate}
where $!:A \to 1$ is the unique morphism to the terminal object and $\langle f, g\rangle$ is the induced morphism into a pullback or product. In the special case of morphisms between products, we use the notation $f \times g := \langle \pi_0 f, \pi_1 g\rangle $.
\end{definition}
While a group object consists of $G$, $m$, $u$ and $\iota$  we often denote the group object $(G,m,u,\iota)$ by $G$ for brevity.
The concept of a group object is used in most areas of mathematics.
In the category of sets and functions, groups are exactly the group objects.
In the category of algebraic varieties and algebraic maps, algebraic groups are exactly the group objects.
In the category of topological spaces with continuous maps, topological groups are exactly the group objects.
In the category of smooth manifolds with smooth maps, Lie groups are exactly the group objects.

\begin{lemma}\label{lemma:tangent_group}
If $G$ is a group object in a Cartesian tangent category $(\mathbb X, T)$, $T(G)$ is also group object.
\end{lemma}
\begin{proof}
Since $(\mathbb X, T)$ is a Cartesian tangent category, $\langle T(\pi_0), T(\pi_1)\rangle : T(G \times G) \to T(G) \times T(G)$ is an isomorphism. The multiplication is
$
T(G) \times T(G) \overset{\langle T(\pi_0), T(\pi_1)\rangle}\cong T(G \times G) \xrightarrow{T(m)} T(G),
$
the unit is given by
$
1 \overset{0_1}{\cong} T(1) \xrightarrow{T(u)} T(G),
$
and the inverse is given by 
$
T(G) \xrightarrow{T(\iota)} T(G) .
$
These data satisfy the required identities as $T$ is functorial.
\end{proof}

\subsection{Trivializing the tangent bundle}\label{sec:trivializing_the_tangent_bundle}
The purpose of this section is to define an isomorphism between $T(G)$ and $G \times T(G)_u$. Here $T(G)_u$ is the pullback
\[\begin{tikzcd}
	{T(G)_u} && {T(G)} \\
	1 && G
	\arrow["{p_u^*}", from=1-1, to=1-3]
	\arrow["{!}"', from=1-1, to=2-1]
	\arrow["\lrcorner"{anchor=center, pos=0.125}, draw=none, from=1-1, to=2-3]
	\arrow["p", from=1-3, to=2-3]
	\arrow["u"', from=2-1, to=2-3]
\end{tikzcd}\]
which we require to exist.

While the first component of the isomorphism is simply the projection $p:T(G) \to G$, the second component is defined using the universal property of the pullback defining $T(G)_u$. The idea of the second component is to start with a tangent vector $v_g \in T(G)$ over the basepoint $g$ and use the group multiplication with $g^{-1}$ to move the tangent vector $v_g$ over the unit. Concretely, the multiplication $T(m)$ multiplies $p\iota0$ of a tangent vector to the tangent vector, as shown in the diagram below.

\[\begin{tikzcd}
	&& {T(G)} \\
	&&& {G \times T(G)} \\
	&&&& {G \times T(G)} \\
	&& {T(G)_u} &&& {T(G) \times T(G)} \\
	1 &&&& {T(G)} \\
	&& G
	\arrow["{\langle p,1_{T(G)}\rangle}", from=1-3, to=2-4]
	\arrow["{\langle p \iota0,1\rangle T(m)}"{pos=0.7}, dashed, from=1-3, to=4-3]
	\arrow[from=1-3, to=5-1]
	\arrow["{\iota \times 1_{T(G)}}", from=2-4, to=3-5]
	\arrow["{0 \times 1_{T(G)}}", from=3-5, to=4-6]
	\arrow[from=4-3, to=5-1]
	\arrow[from=4-3, to=5-5]
	\arrow["\lrcorner"{anchor=center, pos=0.125, rotate=-45}, draw=none, from=4-3, to=6-3]
	\arrow["{T(m)}", from=4-6, to=5-5]
	\arrow["u"', from=5-1, to=6-3]
	\arrow["p", from=5-5, to=6-3]
\end{tikzcd}\]
Together they form the morphism 
$$
\phi = \langle p, \langle !, \langle p \iota 0, 1 \rangle  T(m)\rangle \rangle : T(G) \to G \times T(G)_u.
$$

\begin{theorem}\label{thm:group_tangent}
Suppose $G$ is a group object in a Cartesian tangent category such that the pullback $T(G)_u$
exists. Then 
$
T(G) \cong G \times T(G)_u
$.
The isomorphism is given as
$$
\phi = \langle p, \langle !, \langle p \iota 0, 1 \rangle  T(m)\rangle \rangle : T(G) \to G \times T(G)_u
$$
and its inverse is
$$
\phi^{-1} = ( 0 \times p^*_u ) T(m) : G \times T(G)_u \to T(G).
$$
\end{theorem}
\begin{proof}
The morphism $\langle p \iota 0 , 1 \rangle T(m): T(G) \to T(G)$ satisfies 
$$
\langle p \iota 0 , 1 \rangle T(m) p = \langle p \iota 0 , 1 \rangle p m =
\langle p \iota 0 p , p \rangle m = p \langle \iota , 1 \rangle m = p ! 0 = !0
$$
and therefore it induces a morphism $\langle ! , \langle p \iota 0, 1 \rangle  T(m)\rangle : T(G) \to T(G)_u$. The composition $p! = !$, because it is the unique morphism from $T(G)$ to the terminal object $1$.
The candidate for its inverse is
$$
( 0 \times p^*_u ) T(m) : G \times T(G)_u \to T(G)
$$
where $p^*_u$ is the pullback of $p:T(G)\to G$ along $u: 1 \to G$ defining $T(G)_u$:
\[\begin{tikzcd}
	{T(G)_u} & {T(G)} \\
	{1} & G
	\arrow["{!}"', from=1-1, to=2-1]
	\arrow["{p^*_u}", from=1-1, to=1-2]
	\arrow["u"', from=2-1, to=2-2]
	\arrow["p", from=1-2, to=2-2]
\end{tikzcd}\]
Now we need to check that they are inverse to each other by checking that both ways of composing yield the identity morphism. First, we have
\begin{align*}
\langle p, \langle ! , \langle p \iota 0, 1 \rangle  T(m)\rangle \rangle ( 0 \times p^*_u ) T(m) &= \langle p0, \langle p \iota 0, 1 \rangle  T(m) \rangle T(m) 
\\
\text{(by associativity)} \qquad  
&= \langle \langle p0, p \iota 0 \rangle T(m), 1\rangle T(m)
\\
&= \langle p \langle 1, \iota \rangle m 0 , 1 \rangle T(m)
\\
&= \langle p ! u 0, 1 \rangle T(m)\\
&= 1,
\end{align*}
where the last equality holds because $u 0 = 0 T(u)$ by naturality of $0$ and this is the unit for the multiplication $T(m)$ from Lemma \ref{lemma:tangent_group}. On the other hand, 

\begin{align*}
&( 0 \times p^*_u ) T(m) \langle p, \langle !, \langle p \iota 0, 1 \rangle  T(m)\rangle \rangle 
\\
&=  \langle ( 0 \times p^*_u ) T(m) p ~ , ~   \langle !, \langle ( 0 \times p^*_u ) T(m) p \iota 0, ( 0 \times p^*_u ) T(m) \rangle  T(m)\rangle \rangle
\\
\text{(by naturality of $p$)} \qquad
&= \langle ( 0 p \times p^*_u p ) m ~ , ~  \langle !, \langle ( 0 p \times p^*_u p ) m \iota 0, ( 0 \times p^*_u ) T(m) \rangle  T(m)\rangle \rangle
\\
&= \langle ( 1 \times ! u ) m ~ , ~  \langle ! \langle ( 1 \times ! u ) m \iota 0, ( 0 \times p^*_u ) T(m) \rangle  T(m) \rangle \rangle
\\
&= \langle \pi_0 ~ , ~  \langle !, \langle \pi_0 \iota 0, ( 0 \times p^*_u ) T(m) \rangle  T(m)] \rangle
\\
\text{(by associativity)} \qquad
&=\langle \pi_0 ~ , ~  \langle !, \langle  \langle \pi_0 \iota 0, \pi_0 0 \rangle T(m), \pi_1 p^*_u \rangle T(m)\rangle \rangle
\\
&=\langle \pi_0 ~ , ~  \langle !, \langle  \pi_0 ! u 0, \pi_1 p^*_u \rangle T(m)\rangle \rangle
\\
&=\langle \pi_0 ~ , ~  \langle !, p_u^* \langle  \pi_0 ! u 0, \pi_1 p_u^* \rangle T(m)\rangle \rangle
\\
&= \langle \pi_0, \pi_1 \rangle = 1.
\end{align*}
Here again the last equality holds as $u 0 = 0 T(u)$ by naturality of $0$ and this is the unit for the multiplication $T(m)$ from Lemma \ref{lemma:tangent_group}.

As the first component of the isomorphism $\langle p, \langle ! , \langle p \iota 0, 1 \rangle  T(m)\rangle \rangle$ is $p$, it follows directly that 
$$
\langle p, \langle ! , \langle p \iota 0, 1 \rangle  T(m)\rangle \rangle \pi_0 = p
$$
and therefore $p$ corresponds to $\pi_0$.
\end{proof}
\subsection{Structure maps of the tangent bundle of a group objects}
We will now continue to analyze the behaviour of $G$ and $T(G)_u$. In particular, $G$ acts on $T(G)_u$  in a way that recovers the adjoint action of a Lie group on its Lie algebra in differential geometry. In the classical case in which $G$ is a Lie group, it is given by sending a group element $g \in G$ and a vector $v \in T(G)_u$ to $g\cdot v \cdot g^{-1}$. This provides intuition for what follows.

Recall that Definition \ref{def:group_object} guarantees that the multiplication $m: G \times G \to G$ is associative. Thus, we denote for $k \in \mathbb N^+$ the unique morphism from $G^k$ to $G$ multiplying all components together as  $m_k  : G^k \to G $. One expression of $m_k$ as a composition is $(m \times 1_G \times ... \times 1_G) ... m$, but any other composition of $m$ that leads to a morphism $G^k \to G$ is the same morphism.

\begin{definition}\label{def:adjoint_actions}~
$$
\mathrm{Ad}_\ell : G \times T(G)_u \to T(G)_u
$$
is defined as the morphism
$$
\langle ! , \langle \pi_0 0_G , \pi_1 p^*_u, \pi_0 \iota 0 \rangle T(m_3) \rangle
$$
which is the induced morphism into the pullback 
\[\begin{tikzcd}
	{G\times T(G)_u} \\
	& {T(G)_u} & {T(G)} \\
	& {1} & G
	\arrow["{!}"', from=2-2, to=3-2]
	\arrow["{p^*_u}"', from=2-2, to=2-3]
	\arrow["u"', from=3-2, to=3-3]
	\arrow["p", from=2-3, to=3-3]
	\arrow["\lrcorner"{anchor=center, pos=0.125}, draw=none, from=2-2, to=3-3]
	\arrow["{!}"', from=1-1, to=3-2]
	\arrow["{\langle \pi_0 0 , \pi_1p_u^*, \pi_0 \iota 0\rangle T(m_3)}", from=1-1, to=2-3]
\end{tikzcd},\]
where $m_3= (m \times 1) m: G \times G \times G \to G$ is the multiplication of three group elements in $G$.

The \textbf{adjoint right group action} is adjoint action with the inverse of a group element, i.e. it is defined by
$$\Adj_r =
\langle ! , \langle \pi_1 \iota 0 , \pi_0 p^*_u, \pi_1 0 \rangle T(m_3) \rangle : T(G)_u \times G \to T(G)_u.
$$
\end{definition}
Here the $\ell$ and $r$ in the index keep track which of them is a left- and which of them is a right-action. The actions $\Adj_r$ and $\Adj_\ell$ are the generalizations of $g v g^{-1}$ and $g^{-1} v g$ in the classical case.

In order to see that Definition \ref{def:adjoint_actions} is well-defined we need to check that it actually induces a morphism into the pullback, 
that is $\langle \pi_0 0 , \pi_1 p^*_u, \pi_0\iota0 \rangle T(m_3) p = !u$:
\begin{align*}
&\langle \pi_0 0 , \pi_1 p^*_u, \pi_0\iota0 \rangle T(m_3) p 
\\
\overset{\text{naturality of }p}{=}&
\langle \pi_0 0 p , \pi_1 p^*_u p, \pi_0\iota0 p \rangle m_3 
\\ \overset{0p=1 ~,~ p_u^* p = !u}{=}&
\langle \pi_0 , ! u , \pi_0 \iota \rangle m_3 
\\\overset{\iota \text{ is inversion}}{=}& !u .
\end{align*}
\begin{definition}
In a Cartesian category $\mathbb X$, given an object $A$ and a group object $G$, a \textbf{left $G$-action} on $A$ is a morphism $a: G \times A \to A$ such that the diagrams
\[\begin{tikzcd}
	A & {G \times A} & {G \times G \times A} & {G \times A} \\
	& A & {G \times A} & A
	\arrow["{u \times 1_A}", from=1-1, to=1-2]
	\arrow[from=1-1, to=2-2]
	\arrow["a", from=1-2, to=2-2]
	\arrow["{m \times 1_A}", from=1-3, to=1-4]
	\arrow["{1_G \times a}"', from=1-3, to=2-3]
	\arrow["a"', from=1-4, to=2-4]
	\arrow["a"', from=2-3, to=2-4]
\end{tikzcd}\]
commute. A \textbf{right $G$-action} on $A$ is a morphism $a: A \times G \to A$ such that the diagrams
\[\begin{tikzcd}
	A & {A \times G} & {A \times G \times G} & {A \times G} \\
	& A & {A \times G} & A
	\arrow["{1_A \times u}", from=1-1, to=1-2]
	\arrow[from=1-1, to=2-2]
	\arrow["a", from=1-2, to=2-2]
	\arrow["{1_A \times m}", from=1-3, to=1-4]
	\arrow["{a \times 1_G}"', from=1-3, to=2-3]
	\arrow["a"', from=1-4, to=2-4]
	\arrow["a"', from=2-3, to=2-4]
\end{tikzcd}\]
commute.
\end{definition}

\begin{lemma}[Left action]The morphism $\Adj{}_\ell$ is a left group action
Analogously $\Adj{}_r$ is a right group action
\end{lemma}
The proof that $\Adj_l$ is a left group action and $\Adj_r$ is a right group action follows from the associativity of $m$.

In this section we will use the trivialization from Theorem \ref{thm:group_tangent} to express the natural transformations in the tangent structure of Definition \ref{def:tan_cat} and the group structure morphisms of Definition \ref{def:group_object} on $T(G)$ through structures on $G \times T(G)_u$. The result is summarized in Table \ref{tab:translations}.

An observation from Table \ref{tab:translations} is that the tangent structure acts on the $T(G)_u$ component and not on the $G$ component. The group structure involves the $G$ component as well.

\begin{table}[ht!]
    \centering
\renewcommand{\arraystretch}{1.5} 
{\small
\begin{tabular}{c|c}
Operation on $T(G)$ & 
Operation on $G \times T(G)_u$\\ \hline
$ +: T_2 G \to T(G)$ &  $1 \times +_u: G \times T(G)_u \times T(G)_u \to G \times T(G)_u$  \\
$0: G \to T(G)$ & $\langle 1 , !0_u \rangle : G \to G \times T(G)_u$
\\
$p: T(G) \to G$ & $\pi_0 : G \times T(G)_u \to G$
\\
$\ell : T(G) \to TT(G)$ & $\langle \pi_0 , !0_u, !0_u , \pi_1 \rangle : G \!\times\! T(G)_u \to G \!\times\! T(G)_u \!\times\! T(G)_u \!\times\! T(G)_u$
\\
$c: TT(G) \to TT(G)$ &
$\langle \pi_0 , \pi_2 , \pi_1 , \pi_3 \rangle : G \times T(G)_u^3 \to G \times T(G)_u^3$
\\
$Tm\!:\!T(G) \! \times\! T(G)\! \to \! T(G)$ &
$(1 \! \times \! s \times 1)(m\!\times\!T(m)_u)\!:\! G \!\times\! T(G)_u \!\times\! G \!\times\! T(G)_u \to G \!\times\! T(G)_u$
\\
$Tu : 1 \to T(G)$ & $\langle u, 0_u\rangle : 1 \to G \times T(G)_u$
\\ $T \iota : T(G) \to T(G)$ & $s(T(\iota)_u \times \iota): G \times T(G)_u \to T(G)_u \times G $
\end{tabular}
}

\caption{Structure morphisms on $T(G)$ and their translation to the local coordinates $G \times T(G)_u$. The meanings of $0_u$, $+_u$ and $s$ are discussed in Proposition \ref{prop:trivial-translations} and Lemma \ref{lemma:the_swap}}.
\label{tab:translations}
\end{table}

Lines 1,2,3 and 7 of the table (pertaining to $+$, $0$, $p$ and $u$) follow very directly from the definition of the isomorphism $\phi$ in Theorem \ref{thm:group_tangent}. We show these equations in Proposition \ref{prop:trivial-translations}.

The next entry of Table \ref{tab:translations}, pertaining to $\ell$, is shown in Proposition \ref{prop:lift in T(G)_u world}. 
The following entry, pertaining to $c$ is shown in in Proposition \ref{prop:canonical_flip_translation}.
The following entry, pertaining to $T(m)$, is shown in Proposition \ref{prop:multiplication in T(G)_u world}. 
The last entry, pertaining to $T(\iota)$, is shown in Proposition \ref{prop:inverse in T(G)_u world}.

In all the propositions throughout the rest of this section, the vertical arrows are isomorphisms built from $\phi$ that translate the operations on $T(G)$ into operations on $G \times T(G)_u$. The horizontal arrows are the morphisms in Table \ref{tab:translations}.
\begin{proposition}\label{prop:trivial-translations}
The diagrams \\
\noindent{
\small
\begin{tabular}{cc}
$
\begin{tikzcd}
	{G \!\times \!T(G)_u \!\times\! T(G)_u} && {G \!\times\! T(G)_u} \\
	{T_2G} && TG
	\arrow["{1_G \!\times\! +_u}", from=1-1, to=1-3]
	\arrow["{\langle \langle \pi_0 , \pi_1 \rangle \phi^{-1}, \langle \pi_0, \pi_2\rangle \phi^{-1}\rangle}","\cong"', from=1-1, to=2-1]
	\arrow["{\phi^{-1}}","\cong"', from=1-3, to=2-3]
	\arrow["{+}"', from=2-1, to=2-3]
\end{tikzcd} \inlineeqno
$
&
$\begin{tikzcd}
	G && {G \!\times\! T(G)_u} \\
	G && TG
	\arrow["{\langle 1_G, !u0_u\rangle}", from=1-1, to=1-3]
	\arrow["{1_G}"',"\cong", from=1-1, to=2-1]
	\arrow["{\phi^{-1}}","\cong"', from=1-3, to=2-3]
	\arrow["0"', from=2-1, to=2-3]
\end{tikzcd} \inlineeqno$ 
\\
$\begin{tikzcd}
	{G \!\times\! T(G)_u} && G \\
	TG && G
	\arrow["{\pi_0}", from=1-1, to=1-3]
	\arrow["{\phi^{-1}}"',"\cong", from=1-1, to=2-1]
	\arrow["{1_G}","\cong"', from=1-3, to=2-3]
	\arrow["p"', from=2-1, to=2-3]
\end{tikzcd} \inlineeqno$
& 
$\begin{tikzcd}
	1 && {G \!\times\! T(G)_u} \\
	{T(1)} && TG
	\arrow["{\langle u,u0_u\rangle}", from=1-1, to=1-3]
	\arrow["{0_1}"',"\cong", from=1-1, to=2-1]
	\arrow["{\phi^{-1}}","\cong"', from=1-3, to=2-3]
	\arrow["{T(u)}"', from=2-1, to=2-3]
\end{tikzcd} \inlineeqno$
\end{tabular}}

\noindent{
commute where $\phi^{-1}: G \times T(G)_u \to T(G)$ is the isomorphism from Theorem \ref{thm:group_tangent}, $0_u := \langle ! , u 0_G\rangle: 1 \to T(G)_u$ is the zero morphism into $T(G)_u$ induced by $0_G: G \to T(G)$ and $+_u = \langle ! , \langle p_u^* , p_u^* \rangle + \rangle: T(G)_u \times T(G)_u \to T(G)_u$ is the addition on $T(G)_u$ induced by $+$.}
\end{proposition}
In this proposition the vertical arrows are isomorphisms and 
express the group action in terms of the tangent space at the unit.

Here, and and in the following proofs when we will show that a diagram
\[\begin{tikzcd}
	A && B \\
	C && D
	\arrow["f", from=1-1, to=1-3]
	\arrow["h"', from=1-1, to=2-1]
	\arrow["g", from=1-3, to=2-3]
	\arrow["k"', from=2-1, to=2-3]
\end{tikzcd}\]
commutes, we will refer to $hk$ as the ``left path" and to $fg$ as the ``right path".

\begin{proof}
Diagram 4.1 commutes because its right path $(1_G \times +_u) \phi^{-1}$ is
$$
(1_G \times +_u)\langle 0 \times p_u^*\rangle T(m)= (0 \times (p_u^*, p_u^*)+)T(m) = 
\langle \langle \pi_0 0, \pi_1 p_u^* \rangle , \langle \pi_0 0 , \pi_2 p_u^* \rangle \rangle T_2(m) + 
$$,
which equals to the left path $\langle \langle \pi_0 , \pi_1 \rangle \phi^{-1}, \langle \pi_0, \pi_2\rangle \phi^{-1}\rangle +$.
Here we used the naturality of $+$ for the second equality and the definition of $\phi^{-1}$ as $\langle \pi_0 0G, \pi_1 p_u^* \rangle T(m)$ to conclude that it equals to the left path.

Diagram 4.2 commutes because its right path is
$$
\langle 1_G , !u0_u \rangle ( 0 \times \pi_1 p_u^*) T(m) =
\langle 0 , !u0 \rangle T(m)= 0 = 1_G 0
$$
since $u0$ is the unit of the group object $T(G)$.

Diagram 4.3 commutes because its left path is 
$$
\langle \pi_0 0 , \pi_1 p_u^* \rangle T(m) p = 
\langle \pi_0 0 p , \pi_1 p_u^* p \rangle m =
\langle \pi_0 , \pi_1 ! u \rangle m = 
\pi_0.
$$

Diagram 4.4 commutes because the right path is
$$
\langle u , u 0_u \rangle \langle \pi_0 0, \pi_1 p_u^* \rangle T(m) = 
\langle u 0, u 0 \rangle T(m) = u 0 = 0 T(u).
$$
The vertical arrow $0_1$ is an isomorphism because $\mathbb X$ is a Cartesian tangent category. The vertical arrow $\langle \pi_0 0 , \pi_1 p_u^* \rangle T(m)$ is an isomorphism as we proved in Theorem \ref{thm:group_tangent}. Therefore its pullback square is also an isomorphism.
\end{proof}

In order to prove that the remaining lines in the table are correct, the following Lemma is useful. It identifies $G \times T_uG$ with $T_uG \times G$ in a way that is compatible with the group multiplication and thereby tells us how to move group elements past each other.

\begin{lemma}[Switching order]\label{lemma:the_swap}~
\begin{enumerate}[(i)]
\item  The diagram 
\[\begin{tikzcd}
	{T(G)_u \times G} && {G \times T(G)_u} \\
	{T(G) \times T(G)} && {T(G) \times T(G)} \\
	& T(G)
	\arrow["Tm", from=2-3, to=3-2]
	\arrow["Tm"', from=2-1, to=3-2]
	\arrow["{p^*_u\times 0}"', from=1-1, to=2-1]
	\arrow["{0 \times p^*_u}", from=1-3, to=2-3]
	\arrow["{\langle \pi_1, \mathrm{Ad}_r\rangle}", from=1-1, to=1-3]
\end{tikzcd}\]
commutes.
\item The composition $T(G)_u \times G \xrightarrow{\langle p^*_u, 0\rangle T(m)} T(G) \xrightarrow{\cong} 
G \times T(G)_u$ is $\langle \pi_1 , \Adj_r \rangle$.
\item The morphism $s:=\langle \pi_1 , \mathrm{Ad}_r \rangle : T(G)_u \times G \to G \times T(G)_u $ is an isomorphism and its inverse is $s^{-1} := \langle \mathrm{Ad}_\ell, \pi_0 \rangle $.
\end{enumerate}
\end{lemma}
\begin{proof}~
\begin{enumerate}[(i)]
\item Putting everything together the right path in the diagram is 
\begin{align*}
\langle \pi_1, \Adj_r \rangle (p^*_u \times 0) T(m) &= \langle \pi_1 0, \Adj_r p^*_u \rangle T(m)
\\
= \langle \pi_1 0, \pi_1\iota0 , \pi_0 p_u^* , \pi_1 0 \rangle T(m_4)
&= \langle \pi_0 p_u^* , \pi_1 0 \rangle T(m).
\end{align*}
which is exactly the left path in the diagram.
\item follows now by composing with the inverse of the right vertical arrows as the isomorphism between $T(G)$ and $G \times T(G)_u$ is $$(0 \times p_u^*)T(m): G \times T(G)_u \to T(G).$$
\item is a straightforward calculation:
\begin{align*}
\langle \pi_1 , \Adj_r \rangle \langle \Adj_\ell , \pi_0\rangle 
&=\langle \pi_0 , \Adj_r \langle \Adj_\ell , \pi_0 \rangle \rangle 
= \langle \pi_0 , \langle ! , \langle \pi_0 \iota 0 , \Adj_\ell p_u^* , \pi_0 0 \rangle T(m_3) \rangle \rangle 
\\ &= \langle \pi_0 , \langle ! , \langle \pi_0 \iota 0 , 
\langle \pi_0 0 , \pi_1 p_u^* , \pi_0 \iota 0 \rangle T(m_3) 
 , \pi_0 0 \rangle T(m_3) \rangle \rangle 
\\&= \langle \pi_0 , \langle ! , \langle \pi_0 \iota 0 , 
 \pi_0 0 , \pi_1 p_u^* , \pi_0 \iota 0 
 , \pi_0 0 \rangle T(m_5) \rangle \rangle 
\\&= \langle \pi_0 , \langle ! , \pi_1 p_u^* \rangle \rangle = \langle \pi_0 , \pi_1 \rangle = 1_{G \times T(G)_u}
\end{align*}
The composition in the inverse composition is similar:
\begin{align*}
\langle \Adj_\ell , \pi_0\rangle \langle \pi_1 , \Adj_r \rangle  
&=\langle \langle !, \langle  \pi_1  0 , \langle \pi_1 \iota 0, \pi_0 p_u^* , \pi_1 0 \rangle T(m_3) , \pi_1 \iota 0 \rangle T(m_3), \pi_1 \rangle
\\&= \langle \langle ! , \pi_0 p_u^* \rangle , \pi_1 \rangle = \langle \pi_0 , \pi_1 \rangle = 1_{G \times T(G)_u}
\end{align*}
\end{enumerate}

\end{proof}
The diagram in Lemma \ref{lemma:the_swap} means that the morphism $s := \langle \pi_1 , \Adj_r \rangle$ is the isomorphism that
identifies $T(G)_u \times G$ with $G \times T(G)_u$ in a way that is compatible with the group multiplication.
\begin{proposition}[The multiplication] \label{prop:multiplication in T(G)_u world}
The diagram
\[\begin{tikzcd}
	{G \times T(G)_u \times G \times T(G)_u} & {G \times G \times T(G)_u \times T(G)_u} & {G \times T(G)_u} \\
	{T(G) \times T(G)} && T(G)
	\arrow["{1 \times s \times 1}", from=1-1, to=1-2]
	\arrow["Tm", from=2-1, to=2-3]
	\arrow["{m \times T(m)_u}", from=1-2, to=1-3]
	\arrow["{\phi^{-1}}", from=1-3, to=2-3]
	\arrow["{(\phi^{-1})^2}"', from=1-1, to=2-1]
\end{tikzcd}\]
commutes, where $T(m)_u: T(G)_u \times T(G)_u \to T(G)_u$ is induced by the multiplication $Tm: T(G) \times T(G) \to T(G)$. 
\end{proposition}
\begin{proof}
The left path in the square is, using associativity,
$$
\langle 0 , p^*_u, 0 , p^*_u\rangle T(m_4).
$$
The right path in the square is
\begin{align*}
(1 \times s \times 1) (m \times T(m)_u) \langle 0 , p_u^* \rangle Tm &=
(1 \times s \times 1) \langle \langle \pi_0 , \pi_1 \rangle m 0 , \langle \pi_2 , \pi_3 \rangle T(m)_u p_u^* \rangle T(m) 
\\ &=
(1 \times s \times 1) \langle \langle \pi_0 0, \pi_1 0 \rangle T(m), \langle \pi_2 p_u^*, \pi_3 p_u^* \rangle T(m)  \rangle T(m)  
\\&=
(1 \times s \times 1) \langle 0 \times 0 \times p_u^* \times p_u^* \rangle T(m_4) 
\intertext{Now it follows from part (ii) of Lemma \ref{lemma:the_swap} that this equals
}
&= \langle 0, p_u^* , 0 , p_u^*\rangle T(m_4)
\end{align*}
which was the left path.
\end{proof}
We can also formulate the inverse $T(\iota)$ for $T(G)$ in terms of $G \times T(G)_u$.
\begin{proposition}[The inverse]\label{prop:inverse in T(G)_u world}
Suppose $G$ is a group object in a tangent category and $T(G)_u$ exists. Then the following diagram commutes:
\[\begin{tikzcd}
	{G \times T(G)_u} && {G \times T(G)_u} \\
	{T(G)} && {T(G)}
	\arrow["{\langle \pi_0 \iota , \Adj_\ell T(\iota)_u\rangle}", from=1-1, to=1-3]
	\arrow["{ \phi^{-1}}"', from=1-1, to=2-1]
	\arrow["{T(\iota)}"', from=2-1, to=2-3]
	\arrow["\phi"', from=2-3, to=1-3]
\end{tikzcd}\]
commutes.
\end{proposition}
\begin{proof}
In the first component the diagram commutes because
$$
(0 \times p_u^*) T(m) T(\iota) p= (0p,p_u^*p)m \iota = \pi_0 \iota.
$$
In the second component we obtain:
\begin{align*}
(0 \times p_u^*) T(m) T(\iota)\langle ! , \langle p \iota 0,1\rangle T(m)\rangle &= \langle ! , \langle (0 \times p_u^*) T(m) T(\iota)p \iota 0 , (0 \times p_u^*) T(m) T(\iota)\rangle T(m ) \rangle \\
&= \langle ! , \langle (0 \times p_u^*) T(m) , (0 p \times p_u^* p) m \iota 0 \rangle T(m ) T(\iota)\rangle \\
&= \langle ! , \langle \pi_0 0 , \pi_1 p_u^* ,\pi_0 \iota 0 \rangle T(m_3) T(\iota)\rangle = \mathrm{Ad}_l T(\iota)_u
\end{align*}
\end{proof}
It is known from Theorem 4.15 of \cite{Cockett2014DifferentialST} that $T(G)_u$ is a differential object. The following Lemma recalls what that means.

For $f: X \to TT(G)_u$, let $\{f\}: X \to T(G)_u$ denote the unique morphism such that $f = \langle \{f\} \ell_u , f p 0\rangle  T(+_u)$ where $\ell_u: T(G)_u \to T(T(G)_u)$ and $+_u : T(G)_u \times T(G)_u \to T(G)_u$ are the morphisms induced by $\ell: T(G) \to TT(G)$ and $+: T_2G \to T(G)$, following Lemma 2.10 of \cite{cockett2016diffbundles}.
Then the morphism 
$$
T(T(G)_u) \xrightarrow{\langle p, \{1\}\rangle} T(G)_u \times T(G)_u
$$
is an isomorphism with inverse
$$
T(G)_u \times T(G)_u \xrightarrow{(0 \times \ell_u )T(+_u)} T(T(G)_u)$$

The first composition is
$$
\langle p , \{1\} \rangle (0 \times \ell_u) T(+_u) = \langle 1 p 0 , \{1\} \ell_u \rangle T(+_u) = 1
$$
by definition of $\{\bullet \}$
The other composition is 
\begin{align*}
(0 \times \ell_u)  T(+_u) \langle p , \{1\} \rangle &= \langle (0 \times \ell_u) T(\sigma) p, \{(0 \times \ell_u) T(+_u) \}\rangle
\\&= 
\langle (0p \times \ell_u p) +_u ,\{ \langle \pi_0 0, \pi_1 \ell_u\rangle  T(+_u) \}\rangle
\\&= 
\langle (1 \times ! 0_u) +_u , \langle \{\pi_0 0\}, \{ \pi_1 \ell_u \}\rangle  +_u \rangle
\\&=
\langle \pi_0 , \pi_0 ! 0_u , \pi_1 \rangle +_u
\\&= \langle \pi_0 , \pi_1 \rangle = 1
\end{align*}
where $0_u: 1 \to T(G)_u$ is induced by $u0:1 \to T(G)$. There we used that $\{0\}=!0_u$ and $\{ \ell_u \} = 1$ from Lemma 2.12 of \cite{cockett2016diffbundles}.

\begin{remark}
We provided an explicit isomorphism $T(T(G)_u) \cong T(G)_u \times T(G)_u$. However, one could alternatively simply apply Theorem 4.15 of \cite{Cockett2014DifferentialST} which shows that $T(G)_u$ is a differential object. According to Definition 4.8 of \cite{Cockett2014DifferentialST} this implies that $T(T(G)_u) \cong T(G)_u \times T(G)_u$ and therefore proves that there is an isomorphism $T(T(G)_u) \cong T(G)_u \times T(G)_u$. 
\end{remark}

\begin{proposition}[The lift]\label{prop:lift in T(G)_u world}
Suppose $G$ is a group object in a tangent category and the pullback $T(G)_u$ exists and is preserved by $T$. Then the following diagram commutes:
\begin{equation*}
    \begin{tikzcd}
	{G \times T(G)_u} && {G \times T(G)_u \times T(G)_u \times T(G)_u} \\
	&& {T(G) \times TT(G)_u} \\
	T(G) && TT(G)
	\arrow["{\langle\pi_0, !0_u,!0_u,\pi_1\rangle}", from=1-1, to=1-3]
	\arrow["\ell"', from=3-1, to=3-3]
	\arrow["{\phi^{-1}}"',"\cong", from=1-1, to=3-1]
	\arrow["{T (\phi)}"',"\cong", from=3-3, to=2-3]
	\arrow["{\phi \times \langle p , \{1\}\rangle}"',"\cong", from=2-3, to=1-3]
\end{tikzcd} 
\end{equation*}
Here $\phi = \langle p , \langle ! ,\langle p \iota 0, p_u^* \rangle Tm \rangle \rangle$ is the isomorphism from Theorem \ref{thm:group_tangent}.
\end{proposition}
This shows that the inclusion in the first and last components corresponds to the vertical lift $\ell$, similar to the formulas in Examples \ref{ex:manifolds_as_a_tangent_cat} and \ref{ex:modules_as_a_tangent_cat}.
\begin{proof}
As we know that $\ell T(\varphi) \pi_0 = \ell T(p) = p 0$, it follows that the first two components are $\pi_0$ and $0_u$. The third component is zero as by naturality
$$
\varphi^{-1} \ell  T(\varphi) \pi_1 p = \varphi^{-1}  \ell p \varphi \pi_1  =  \pi_0 \langle 1 , 0 \rangle \pi_1=  ! 0_u.
$$
The fourth component is given by the following calculation:
\begin{align*}
\varphi^{-1} \ell T(\varphi) \pi_1 \{ 1 \} &= \varphi^{-1} \ell T(\langle p , \langle ! , \langle p \iota 0 , 1\rangle T(m) \rangle \rangle) \pi_1 \{ 1 \} \\
&= \varphi^{-1} \langle ! , \langle  \ell T(p) T(\iota)T(0) ,  \ell \rangle T^2(m) \rangle \{ 1 \}\\
&= \varphi^{-1} \langle ! , \langle  p 0 T(\iota)T(0) ,  \ell \rangle T^2(m) \rangle \{ 1 \}\\
&= \varphi^{-1} \langle ! , \langle  p \iota 0 \ell ,  \ell \rangle T^2(m) \rangle \{ 1 \}\\
&= \varphi^{-1} \langle ! , \langle  p \iota 0  , 1 \rangle T(m) \ell  \rangle \{ 1 \}\\
&= \varphi^{-1} \langle ! , \langle  p \iota 0  , 1 \rangle T(m) \rangle  \ell_u \{ 1 \} \\
&= \varphi^{-1} \langle ! , \langle  p \iota 0  , 1 \rangle T(m) \rangle = \varphi^{-1} \varphi \pi_1 = \pi_1
\end{align*}
Here $\ell_u \{1\} = \{\ell_u \} = 1$ as $(T(G)_u,1,!,+_u, 0_u,\ell_u)$ is a differential bundle because $T(G)_u$ is a differential object. 
\end{proof}

\begin{proposition}[The canonical flip]\label{prop:canonical_flip_translation}
Suppose $G$ is a group object in a Cartesian tangent category for which the pullback $T(G)_u$ exists and is preserved by T. Then the diagram
\begin{equation}
\begin{tikzcd}
	{G \times T(G)_u \times T(G)_u \times T(G)_u} && {G \times T(G)_u \times T(G)_u \times T(G)_u} \\
	{T(G) \times TT(G)_u} && {T(G) \times TT(G)_u} \\
	TT(G) && TT(G)
	\arrow["{\langle \pi_0, \pi_2,\pi_1,\pi_3\rangle}", from=1-1, to=1-3]
	\arrow["{\phi \times \langle p,\{1\}\rangle}","\cong"', from=2-1, to=1-1]
	\arrow["{\phi \times \langle p,\{1\}\rangle}"', "\cong", from=2-3, to=1-3]
	\arrow["T\phi","\cong"', from=3-1, to=2-1]
	\arrow["c"', from=3-1, to=3-3]
	\arrow["T\phi"',"\cong", from=3-3, to=2-3]
\end{tikzcd}\label{diagram:flip_phi}
\end{equation}
commutes, where $\phi$ is the isomorphism from Theorem \ref{thm:group_tangent}.
\end{proposition}
This shows that in our coordinates the canonical flip just flips the middle two components, analogous to the formula in Example \ref{ex:modules_as_a_tangent_cat} and the local formula in Example \ref{ex:manifolds_as_a_tangent_cat}.
\begin{proof}
For this proof we will actually use Theorem \ref{thm:eckmann_hilton_for_m_and_plus}, the fact that $+_u$ and $T(m)_u$ coincide and are commutative. Since we don't use the canonical flip when proving Theorem \ref{thm:eckmann_hilton_for_m_and_plus} this is logically sound.
We will show the following diagram commutes
\begin{equation}
\begin{tikzcd}[column sep=large]
	{G \times T(G)_u \times T(G)_u \times T(G)_u} && {G \times T(G)_u \times T(G)_u \times T(G)_u} \\
	{T(G) \times TT(G)_u} && {T(G) \times TT(G)_u} \\
	TT(G) && TT(G)
	\arrow["{\langle \pi_0, \pi_2,\pi_1,\pi_3\rangle}", from=1-1, to=1-3]
	\arrow["{\phi^{-1} \times (0 \times \ell_u)T(+_u)}"',"\cong", from=1-1, to=2-1]
	\arrow["{\phi^{-1} \times (0 \times \ell_u)T(+_u)}","\cong"', from=1-3, to=2-3]
	\arrow["{T\phi^{-1}}"',"\cong", from=2-1, to=3-1]
	\arrow["{T\phi^{-1}}","\cong"', from=2-3, to=3-3]
	\arrow["c"', from=3-1, to=3-3].
\end{tikzcd}.
\label{diagram:flip_phi_inverse}
\end{equation}
Because the vertical arrows in Diagram \ref{diagram:flip_phi_inverse} are the inverses of the vertical arrows in Diagram \ref{diagram:flip_phi}, this is equivalent to the statement we need to show. 

First we will write out the vertical arrows of Diagram \ref{diagram:flip_phi_inverse} in components
\begin{align*}
&((\phi^{-1}) \times (0 \times \ell_u) T(+_u))T(\phi^{-1}) 
\\=& \langle \langle \pi_0 0 , \pi_1 p_u^* \rangle T(m) , \langle \pi_2 0 , \pi_3 \ell_u \rangle T(+_u) \rangle \langle \pi_0 T(0) , \pi_1 T(p_u^*) \rangle T^2(m)
\\=&
\langle \langle \pi_0 0 , \pi_1 p_u^* \rangle T(m) T(0) , 
 \langle \pi_2 0 , \pi_3 \ell_u \rangle T(+_u) T(p_u^*) \rangle T^2(m)
\intertext{
Now we use Theorem \ref{thm:eckmann_hilton_for_m_and_plus} to replace $+_u$ with $T(m)_u$ and use $T(m)_u p_u^* = p_u^* Tm$, $\ell_u T(p_u^*)= p_u^* \ell$ and naturality of 0 to obtain}
&
\langle \langle \pi_0 0 , \pi_1 p_u^* \rangle T(m) T(0) , 
\langle \pi_2 0 , \pi_3 \ell_u \rangle T(T(m)_u) T(p_u^*) \rangle T^2(m)
\\=& \langle \langle \pi_0 0 T(0) , \pi_1 p_u^* T(0) \rangle T^2(m) , 
\langle \pi_2 p_u^* 0 , \pi_3 p_u^* \ell \rangle T^2(m) \rangle T^2(m)
\\=& \langle  \pi_0 0 T(0), \pi_1 p_u^* T(0),\pi_2 p_u^* 0 ,\pi_3 p_u^* \ell \rangle T^2(m_4).
\end{align*}
Now if we follow the left path of the diagram we have
\begin{align*}
((\phi^{-1}) \times (0 \times \ell_u) T(+_u))T(\phi^{-1}) c &= \langle  \pi_0 0 T(0), \pi_1 p_u^* T(0),\pi_2 p_u^* 0 ,\pi_3 p_u^* \ell \rangle T^2(m_4) c
\\&= \langle  \pi_0 0 T(0) c, \pi_1 p_u^* T(0) c,\pi_2 p_u^* 0 c ,\pi_3 p_u^* \ell c \rangle T^2(m_4)
\\&=\langle  \pi_0 0 T(0), \pi_1 p_u^* 0,\pi_2 p_u^* T(0) ,\pi_3 p_u^* \ell \rangle T^2(m_4)
\\&=\langle  \pi_0 0 T(0), \pi_1  0 T(p_u^*),\pi_2 T_u(0) T(p_u^*) ,\pi_3 p_u^* \ell \rangle T^2(m_4)
\\&= \langle  \pi_0 0 T(0),\pi_2 T_u(0) T(p_u^*) , \pi_1  0 T(p_u^*) ,\pi_3 p_u^* \ell \rangle T^2(m_4)
\\&= \langle  \pi_0 0 T(0),\pi_2 p_u^* T(0)  , \pi_1 p_u^* 0  ,\pi_3 p_u^* \ell \rangle T^2(m_4)
\\&= \langle \pi_0 , \pi_2 , \pi_1 , \pi_3 \rangle ((\phi^{-1}) \times (0 \times \ell_u) T(+_u))T(\phi^{-1})
\end{align*}
which is the right path through Diagram \ref{diagram:flip_phi_inverse}. When going from the fourth to fifth line we used that $T(m)_u$ is commutative and therefore 
$$
\langle a p_u^* , b p_u^* \rangle T(m) = \langle a , b \rangle T(m)_u p_u^* = \langle b , a \rangle T(m)_u p_u^* = \langle b p_u^* , a p_u^* \rangle T(m).
$$
\end{proof}
This concludes the verification of all the structures summarized in Table \ref{tab:translations}, with the exception of showing that $+_u = T(m)_u$, which we will verify in section \ref{section:eckmann-hilton}.

\subsection{Comparing addition and multiplication}\label{section:eckmann-hilton}
There are two monoid structures on $T(G)_u$, one of them is the group multiplication and it is a group (i.e. has an inverse). The other one is the addition of tangent vectors and it is commutative. In this section we show, that they are in fact the same monoid structure and therefore both of them are abelian group structures. This implies that the two monoid structures (addition of tangent vectors and group multiplication) on $T(G)$ coincide up to conjugation with $\phi$. Theorem \ref{thm:eckmann_hilton_for_m_and_plus}, the central result of this subsection will make this explicit.

Let $\mathbb X$ be a Cartesian tangent category and $G$ a group object in $\mathbb X$ such that the pullback $T(G)_u$ exists.
Then there are two monoid structures on $T(G)_u$, one of them is given by $T(m)_u$, the morphism induced by
\[\begin{tikzcd}
	& {T(G) \times T(G)} \\
	{T(G)_u \times T(G)_u} & {T(G)_u} & T(G) \\
	& 1 & G
	\arrow["{T(m)}", from=1-2, to=2-3]
	\arrow["{p_u^* \times p_u^*}", from=2-1, to=1-2]
	\arrow["{T(m)_u}", dashed, from=2-1, to=2-2]
	\arrow["{!}"', from=2-1, to=3-2]
	\arrow["{p_u^*}", from=2-2, to=2-3]
	\arrow["{!}"', from=2-2, to=3-2]
	\arrow["\lrcorner"{anchor=center, pos=0.125}, draw=none, from=2-2, to=3-3]
	\arrow["p", from=2-3, to=3-3]
	\arrow["u"', from=3-2, to=3-3]
\end{tikzcd}.\]
The other one is $+_u$, the morphism induced by
\[\begin{tikzcd}
	& {T_2G} \\
	{T(G)_u \times T(G)_u} & {T(G)_u} & T(G) \\
	& 1 & G
	\arrow["{+_G}", from=1-2, to=2-3]
	\arrow["{p_u^* \times_{G} p_u^*}", from=2-1, to=1-2]
	\arrow["{+_u}", dashed, from=2-1, to=2-2]
	\arrow["{!}"', from=2-1, to=3-2]
	\arrow["{p_u^*}", from=2-2, to=2-3]
	\arrow["{!}"', from=2-2, to=3-2]
	\arrow["\lrcorner"{anchor=center, pos=0.125}, draw=none, from=2-2, to=3-3]
	\arrow["p", from=2-3, to=3-3]
	\arrow["u"', from=3-2, to=3-3]
\end{tikzcd},\]
where $p_u^* \times_G p_u^*$ stands for $\langle \pi_0 p_u^* , \pi_1 p_u^* \rangle$. 
In the following theorem, we show that these operations coincide and are commutative. The unit has a convenient expression as $0_u 
= \langle !, u 0_G\rangle = \langle!, 0_1T(u)\rangle: 1 \to T(G)_u$.

\begin{theorem}\label{thm:eckmann_hilton_for_m_and_plus}
The two binary operations on $T(G)_u$ satisfy the following:
\begin{enumerate}[(i)]
    \item They coincide, $+_u = T(m)_u$.
    \item They are commutative, $T(m)_u = \langle \pi_1, \pi_0 \rangle T(m)_u$.
    \item Their unit is $0_u$.
    \item Their inverse is $T(\iota)_u$.
\end{enumerate}
\end{theorem}
\begin{proof}
The morphism $+_u$ is commutative, since it is induced by the commutative monoid addition $+_G$ on $T_2(G)$.
The morphism $0_u$ is the unit for $T(m)_u$ because the unit for $T(m)$ is $0 T(u) = u 0$ and $T(m)_u$ is induced by $T(m)$. The morphism $0_u$ is the unit for $+_u$ because $0_G$ is the unit for $+_G$.
Since $+_G: T_2G \to T(G)$ is natural the morphism $+_u: T(G)_u^2$ is also natural. That implies the naturality diagram of $+_u$ with respect to the morphism $T(m)_u: T(G)_u \times T(G)_u \to T(G)_u$, the diagram
\[\begin{tikzcd}[column sep=large]
	{T(G)_u^2 \times TuG^2} &&& {T(G)_u^2} \\
	{T(G)_u \times T(G)_u} &&& {T(G)_u}
	\arrow["{(T(m)_u)^2}", from=1-1, to=1-4]
	\arrow["{+_u}", from=1-4, to=2-4]
	\arrow["{T(m)_u}"', from=2-1, to=2-4]
	\arrow["{+_u \times +_u}"', from=1-1, to=2-1]
\end{tikzcd},\]
commutes. Now an Eckmann–Hilton style argument implies that the two operations coincide:
\begin{align*}
T(m)_u \overset{\text{neut.}}{=} \langle \langle \pi_0 , 0_u \rangle +_u, \langle 0_u , \pi_1 \rangle +_u  \rangle T(m)_u 
\overset{\text{nat.}}{=} \langle \langle \pi_0 , 0_u \rangle T(m)_u , \langle 0_u , \pi_1 \rangle T(m)_u \rangle +_u \overset{\text{neut.}}{=} +_u.
\end{align*}
Since the two monoid structures coincide and $T(\iota)_u$ is an inverse for $T(m)_u$, it is an inverse for both of them.
\end{proof}

From this we will conclude in the following theorem that the additions on $T(G)$ have a negation.
This result is surprising, as we did not assume that the addition of tangent vectors has negatives. However, for the tangent vectors of the group object $G$ it is defined. It is only defined on $T(G)$, it is not a natural transformation.

\begin{theorem}\label{thm:addition_and_multiplication_on_TG}
Let $\mathbb X$ be a Cartesian tangent category and let $G$ be a group object in $\mathbb X$ such that the pullback $T(G)_u$ exists.
\begin{enumerate}[(i)]
\item Then the addition $+_G : T_2G \to T(G)$ is explicitly given by the composition
$$
T_2G \xrightarrow{\phi_2^{-1}} G \times T(G)_u \times T(G)_u \xrightarrow{(1_G \times T(m)_u)} G \times T(G)_u \xrightarrow{\phi} T(G)
$$
where $\phi: G \times T(G)_u \to T(G)$ is the isomorphism from theorem \ref{thm:group_tangent}.
\item The addition $+_G : T_2(G) \to T(G)$ has a negation that is given by the composition
$$
T(G) \xrightarrow{\phi^{-1}} G \times T(G)_u \xrightarrow{(1_G \times T(\iota)_u)} G \times T(G)_u  \xrightarrow{\phi} T(G) . 
$$
\item The addition $+_{T(G)}: T_2^2G \to T^2G$ has a negation.
\end{enumerate}
\end{theorem}
\begin{proof}
The addition $+_G : T_2G \to T(G)$ can be explicitly described as the composition
$$
G \times T(G)_u \times T(G)_u \xrightarrow{\phi_2} T_2G \xrightarrow{+_G} T(G) \xrightarrow{\phi^{-1}}G \times T(G)_u
$$
where $\phi: G \times T(G)_u \to T(G)$ is the isomorphism from Theorem \ref{thm:group_tangent}. The second component of this composition can be explicitly calculated as the following (the first component is just $p$, as $+$ preserves the basepoint): 
\begin{align*}
\phi_2 +_G \phi^{-1} \pi_1 &= \langle \langle \pi_0 0, \pi_1 p^*_u \rangle T(m), \langle \pi_0 0, \pi_2 p^*_u\rangle T(m)\rangle + \langle ! , \langle p \iota 0, 1 \rangle T(m) \rangle
\\&= 
\langle \langle \pi_0 0, \pi_1 p^*_u \rangle , \langle \pi_0 0, \pi_2 p^*_u \rangle \rangle +^2 T(m) \langle ! , \langle p!0, 1 \rangle T(m) \rangle
\\&=
\langle \langle \pi_0 0, \pi_0 0 \rangle + , \langle \pi_1 p^*_u, \pi_2 p^*_u \rangle + \rangle T(m) \langle ! , \langle p \iota 0, 1 \rangle T(m) \rangle
\\&=
\langle \pi_0 0 , \langle \pi_1 p^*_u, \pi_2 p^*_u \rangle T(m) \rangle T(m) \langle ! , \langle p \iota 0, 1 \rangle T(m) \rangle
\\&=
\langle ! , \langle \langle \pi_0 , !u\rangle m \iota 0, \langle \pi_0 0 , \langle \pi_1 p^*_u, \pi_2 p^*_u \rangle T(m) \rangle T(m) \rangle T(m)\rangle 
\\&=
\langle ! , \langle \pi_1 p^*_u , \pi_2 p^*u \rangle T(m) \rangle 
\\&= 
\langle \pi_1, \pi_2 \rangle T(m)_u
\end{align*}
This calculation now tells us that $ \phi_2 +_G \phi^{-1} = 1_G \times T(m)_u$
Pre- and postcomposing this equality with the isomorphism $\phi$ results in
$$
+_G = \phi_2^{-1} (1_G \times T(m)_u) \phi, 
$$
the equality we needed to show for (i).
This explicit formula shows directly that 
$$
- := \phi^{-1} (1 \times T(\iota)_u) \phi
$$
defines a negation on $T(G)$ implying part (ii) of the theorem.

It follows that $T^2G$ has negatives with respect to $+_{T(G)}$ since $T(G)$ is a group object and
\[\begin{tikzcd}
	{T(G)_u \times T(G)_u} & {T_2G} & {T^2G} & {T^2G} \\
	{1} & G & T(G) & T(G)
	\arrow["{\langle\pi_0,\pi_1\rangle}", from=1-1, to=1-2]
	\arrow["{!}"', from=1-1, to=2-1]
	\arrow["u"', from=2-1, to=2-2]
	\arrow["{\pi_0 p}"', from=1-2, to=2-2]
	\arrow["{T(p)}"', from=1-3, to=2-3]
	\arrow["0"', from=2-2, to=2-3]
	\arrow["p"', from=1-4, to=2-4]
	\arrow["c", from=1-3, to=1-4]
	\arrow["1"', from=2-3, to=2-4]
	\arrow["\nu"', from=1-2, to=1-3]
	\arrow["\lrcorner"{anchor=center, pos=0.125}, draw=none, from=1-3, to=2-4]
	\arrow["\lrcorner"{anchor=center, pos=0.125}, draw=none, from=1-2, to=2-3]
	\arrow["\lrcorner"{anchor=center, pos=0.125}, draw=none, from=1-1, to=2-2]
\end{tikzcd}\]
is a pullback composed of three pullbacks. This negation can be explicitly described by using the fact that the canonical flip is an additive bundle morphism and thus $+ = c_2 T(+) c: T^2_2G \to T^2G$. From this formula we can see that if $-: T(G) \to T(G)$ is a negation with respect to $+: T(G) \to T(G)$, then  $c T(-) c: T^G \to T^2G$ is a negation with respect to $+: T^2G \to T^2G$, proving part (iii).
\end{proof}
The negation $-$ is only defined for the tangents of the group object $G$, it is not a natural transformation.
However the negation is compatible with multiplication and addition in the following sense.
\begin{lemma}\label{lemma:minuscommutes}
The diagrams
\[\begin{tikzcd}
	{T(G) \times T(G)} && {T(G) \times T(G)} && {T_2G} && {T_2G} \\
	T(G) && T(G) && T(G) && T(G)
	\arrow["{- \times -}", from=1-1, to=1-3]
	\arrow["{-}"', from=2-1, to=2-3]
	\arrow["{T(m)}"', from=1-1, to=2-1]
	\arrow["{T(m)}", from=1-3, to=2-3]
	\arrow["{- \times_G -}", from=1-5, to=1-7]
	\arrow["{-}"', from=2-5, to=2-7]
	\arrow["{+}"', from=1-5, to=2-5]
	\arrow["{+}", from=1-7, to=2-7]
\end{tikzcd}\]
commute.
\end{lemma}
\begin{proof}
In the diagram
\[\begin{tikzcd}
	{T(G) \times T(G)} && {T(G) \times T(G)} \\
	{G \times T(G)_u \times G \times T(G)_u} && {G \times T(G)_u \times G \times T(G)_u} \\
	{G \times T(G)_u} && {G \times T(G)_u} \\
	T(G) && T(G)
	\arrow["{-}"', from=4-1, to=4-3]
	\arrow["{1 \times T(\iota)}"', from=3-1, to=3-3]
	\arrow["\phi"', from=3-1, to=4-1]
	\arrow["\phi", from=3-3, to=4-3]
	\arrow["{T(m)}"', from=2-1, to=3-1]
	\arrow["{T(m)}", from=2-3, to=3-3]
	\arrow["{1 \times T(\iota)\times 1 \times T(\iota)}", from=2-1, to=2-3]
	\arrow["{- \times -}", from=1-1, to=1-3]
	\arrow["{\phi^{-1} \times \phi^{-1}}"', from=1-1, to=2-1]
	\arrow["{\phi^{-1} \times \phi^{-1}}", from=1-3, to=2-3]
\end{tikzcd}\]
the top and bottom square commute by the definition of $-$. The commutativity of the central square is that 
$(g^{-1}vg)^{-1} w = g v^{-1} g^{-1} w$.
The first diagram commutes by the definition of $-$. Then the second diagram commutes because $+ = \phi_2 (1 \times T(m)_u) \phi^{-1}$.
\end{proof}

\subsection{Left-invariant vector fields}
A \textbf{left-invariant vector field} is a section $\xi : G \to T(G)$ of $p: T(G) \to G$ such that the diagram
\[\begin{tikzcd}
	{G \times G} & G \\
	{T(G)\times T(G)} & T(G)
	\arrow["\xi", from=1-2, to=2-2]
	\arrow["{T(m)}"', from=2-1, to=2-2]
	\arrow["{0\times \xi}"', from=1-1, to=2-1]
	\arrow["m", from=1-1, to=1-2]
\end{tikzcd}\]
commutes. As shown in Proposition \ref{prop:zero_is_left_invariant}, the zero vector field is an example for the left-invariant vector field.

Classically, the vector space of left-invariant vector fields is an alternate description of a Lie group's Lie algebra. In this case, one can prove that the vector space of left-invariant vector fields is isomorphic to $T(G)_u$. Unfortunately, we were not able to fully generalize this result to the general tangent category setting. At least, it is true for the set of left-invariant vector fields and the set of elements, i.e. morphisms $1 \to T(G)_u$.

Technically, the word "elements´´ is used here with two meanings, morphisms $1 \to T(G)_u$ and the classical meaning of elements of the underlying set of the manifold. These two meanings correspond to each other by identifying a morphism $1 \to T(G)_u$ with its image (remembering that 1 has the one point-set as its underlying set).

\begin{proposition}[Left-invariant vector fields]\label{prop:left-invariant_isomorphic}
The set of left invariant vector fields $\xi: G \to T(G)$ is isomorphic to the set of elements $V: 1 \to T(G)_u $ of $T(G)_u$.
\end{proposition}
\begin{proof}
The morphism $\mathbb X_\mathrm{Linv} (G , T(G)) \to \mathbb X (1 , T(G)_u)$ is given by $\xi \mapsto V_\xi := \langle 1_1, u \xi \rangle: 1 \to T(G)_u$.

The morphism $\mathbb X (1 , T(G)_u) \to \mathbb X_\mathrm{Linv} (G , T(G))$ is given by sending $V: 1 \to T(G)_u$ to the composition of 
$$
G \cong G \times 1 \xrightarrow{1_G \times V} G \times T(G)_u \xrightarrow{ 0 \times p^*_u} T(G) \times T(G) \xrightarrow{T(m)}T(G),
$$
i.e $V \mapsto \xi_V := (0 \times V p^*_u) T(m)$.

These two constructions are inverse to each other as
$$
\xi_{V_\xi} = (0 \times \langle 1_1, u \xi \rangle p^*_u)T(m) = (0 \times u \xi )T(m) = (1 \times u)(0\times \xi) T(m) = (1 \times u) m \xi = \xi
$$
and 
$$
V_{\xi_V} = \langle 1_1 , u (0 \times V p^*_u) T(m) \rangle = \langle 1_1 , (u 0 \times V p^*_u)T(m) \rangle = \langle 1_1, V p^*_u \rangle = V
$$
where the second to last equality holds true as $u0 = 0T(u)$ is the unit of the multiplication $T(m)$ by Lemma \ref{lemma:tangent_group}.
\end{proof}
This gives a new perspective on Theorem \ref{thm:group_tangent}. Theorem \ref{thm:group_tangent} says that the tangent space $T(G)$ is $G \times T(G)_u$. Classically, an element of $T(G)$ is a vector $v \in T(G)_u$ which is "shifted´´ by an action of $g \in G$. Now Proposition \ref{prop:left-invariant_isomorphic} says that furthermore, we can think of v as a left-invariant vector field $\xi_v$ rather than a vector, namely the left-invariant vector field that is given by "shifting´´ the vector $v \in T(G)_u$ to any group element $g\in G$. The inverse is given by sending $(g,\xi)$ to the vector $\xi(g) \in T(G)$.

Given two vector fields $\xi, \zeta: G \to T(G)$, we call $\langle \xi , \zeta \rangle +$ the addition of $\xi$ and $\zeta$. We call $\xi -$ the negation of $\xi$. Together with the zero vector field this turns the set of all vector fields into an abelian group. The following Proposition \ref{prop:zero_is_left_invariant} shows that the left-invariant vector fields are a subgroup.

\begin{proposition}\label{prop:zero_is_left_invariant}
Let $\mathbb X$ be a Cartesian tangent category and let $G$ be a group object such that the pullback $T(G)_u$ exists. Then
\begin{enumerate}[(i)]
\item The zero vector field is left-invariant,
\item The addition of two left-invariant vector fields is left-invariant, and 
\item The negation of a left-invariant vector field is left invariant.
\end{enumerate}
\end{proposition}
\begin{proof}~
\begin{enumerate}[(i)]
\item The zero vector field is left-invariant as 
\[\begin{tikzcd}
	{G \times G} & G \\
	{T(G)\times T(G)} & T(G)
	\arrow["m", from=1-1, to=1-2]
	\arrow["Tm"', from=2-1, to=2-2]
	\arrow["0", from=1-2, to=2-2]
	\arrow["{0\times 0}"', from=1-1, to=2-1]
\end{tikzcd}\]
commutes by the naturality of $0$.
\item In order to show that the addition of the two left-invariant vector fields $\xi$ and $\zeta$ is left-invariant, we need to show that
\[\begin{tikzcd}
	{G \times G} & G \\
	{T(G)\times T(G)} & T(G)
	\arrow["m", from=1-1, to=1-2]
	\arrow["Tm"', from=2-1, to=2-2]
	\arrow["{\langle \xi,\zeta\rangle +}", from=1-2, to=2-2]
	\arrow["{0 \times \langle \xi,\zeta\rangle +}"', from=1-1, to=2-1]
\end{tikzcd}\]
commutes. It commutes because
\begin{align*}
(0 \times \langle \xi , \zeta \rangle + ) Tm =& (\langle 0,0 \rangle + \times \langle \xi , \zeta \rangle + ) Tm 
\\=& \langle (0 \times \xi)Tm , (0 \times \zeta) Tm \rangle + 
\\=&\langle m \xi , m \zeta \rangle + 
\\=& m \langle \xi , \zeta \rangle +.
\end{align*}
\item The negation preserves left-invariance because in the diagram
\[\begin{tikzcd}
	{G \times G} & G \\
	{T(G) \times T(G)} & T(G) \\
	{T(G) \times T(G)} & T(G)
	\arrow["Tm"', from=3-1, to=3-2]
	\arrow["Tm"', from=2-1, to=2-2]
	\arrow["m", from=1-1, to=1-2]
	\arrow["{0 \times \xi}"', from=1-1, to=2-1]
	\arrow["\xi", from=1-2, to=2-2]
	\arrow["{- \times -}"', from=2-1, to=3-1]
	\arrow["{-}", from=2-2, to=3-2]
\end{tikzcd}\]
the upper part commutes because of left-invariance, the lower part commutes because of Lemma \ref{lemma:minuscommutes} and the whole diagram shows the left-invariance of $\xi -$ since $0- = 0$.
\end{enumerate}
\end{proof}

\subsection{Defining the Lie bracket}
As the tangent bundle $T(G)$ has negatives it is possible to define a Lie-bracket on sections. 
For two vector fields $w_1, w_2: G \to T(G)$, \cite{Cockett2014DifferentialST}*{Definition 3.14} defines the Lie bracket as 
$$
[w_1, w_2] = \{ \langle w_1 T(w_2), w_2 T(w_1) c -\rangle+ \}
$$
where the brackets turn a morphism $f:A \to T^2G$ equalizing $pp0$ and $T(p)$ to a morphism $\{f\}: A \to T(G)$ given by the composition of the induced morphism into the equalizer $T_2G$ and the projection $\pi_0: T_2G \to T(G)$. This means $\{f\}$ it is the unique morphism such that $f = \langle \{f\} \ell , f p 0\rangle  T(+) $.

We will now use this to define a Lie bracket on the left-invariant vector fields mimicking classical differential geometry where the left-invariant vector fields form the Lie-Algebra of a Lie group.
\begin{lemma}\label{lem:preservation_of_invariance}
If the vector fields $w_1, w_2 : G \to T(G)$ are left-invariant, then $[w_1, w_2]$ is left-invariant.
\end{lemma}
\begin{proof}
We need to show that $m [w_1, w_2] = (0 \times [w_1 , w_2] ) T(m)$ using the fact that $w_1$ and $w_2$ are left-invariant, i.e. $m w_1 = (0 \times w_1) T(m)$.

\begin{align*}
& m [w_1, w_2]  =
\\& m \{\langle w_1 T(w_2), w_2 T(w_1) c - \rangle +\}
\\=&\{\langle m w_1 T(w_2), m w_2 T(w_1) c - \rangle +\} &&\text{$f\{g\}=\{fg\}$}
\\=&
\{\langle (0 \times w_1) T(m) T(w_2), ( 0 \times w_2) T(m) T(w_1) c - \rangle +\}&&\text{left-invariance}
\\=&
\{\langle (0 \times w_1)  (0 \times T(w_2)) T^2(m), ( 0 \times w_2) (0 \times T(w_1)) T^2(m) c - \rangle +\} && \text{left-invariance}
\\=&
\{\langle (0 \times w_1)  (0 \times T(w_2)), ( 0 \times w_2) (0 \times T(w_1)) c - \rangle +  T^2(m)\}
&& \text{Lemma \ref{lemma:minuscommutes}}
\\=& 
\{00 \times \langle  w_1  T(w_2), w_2 T(w_1) c - \rangle +  T^2(m)\}  && \text{composition}
\\=& 
\{0 0 \times \langle  w_1  T(w_2), w_2 T(w_1) c - \rangle + \} T(m)
&& \text{\cite{Cockett2014DifferentialST}*{Lemma 2.14}} 
\\=&
(0 \times [w_1, w_2]) T(m)&& \text{\cite{Cockett2014DifferentialST}*{Lemma 2.14}}
\end{align*}
\end{proof}

This defines a Lie-bracket on the set of left invariant vector fields. Using Proposition \ref{prop:left-invariant_isomorphic} we may also consider the Lie-bracket to be a pairing on $\mathbb X(1, T(G)_u)$:
$$
[\bullet,\bullet]_{1\to T(G)_u}:
\mathbb X (1, T(G)_u) \times \mathbb X (1, T(G)_u) \to \mathbb X (1, T(G)_u)
$$ 
That is a pairing on the elements of $T(G)_u$, i.e. the morphisms from the terminal object $1$ into $T(G)_u$. We call this the external Lie Algebra.

\begin{theorem}\label{thm:Lie_Algebra}
    Let $G$ be a group object in a Cartesian tangent category $\mathbb X$ for which the pullback $T(G)_u$ exists. Then the set $\mathbb X(1, T(G)_u)$ is a Lie-Algebra over the ring of integers $\mathbb Z$.
\end{theorem}
\begin{proof}
    Instead of $\mathbb X(1, T(G)_u)$, we use the set of left-invariant vector fields in this proof. By Proposition \ref{prop:left-invariant_isomorphic} this is isomorphic to $\mathbb X(1, T(G)_u)$. The addition $+$, the zero section $0$ and the negation from Theorem \ref{thm:addition_and_multiplication_on_TG} (ii) induce an abelian group structure on the set of left-invariant vector fields which induces a $\mathbb Z$-module structure. It follows from Lemma \ref{lem:preservation_of_invariance} that the bracket is an operation on the set of left-invariant vector fields. It follows from Theorem 3.17 of \cite{Cockett2014DifferentialST} that the Lie-bracket is bilinear with respect to this $\mathbb Z$-module structure. The Jacobi identity has been proved by in Theorem 4.2 of \cite{cockett_cruttwell_jacobi_identity}. Thereby the set of left-invariant vector fields is a Lie Algebra.
\end{proof}

Theorem 7.5 of \cite{Loryaintablian2025differentiablegroupoidobjectsabstract} is a related result, though with different assumptions and a different perspective. 

A natural question to ask is whether this external Lie algebra structure comes from a internal Lie-bracket on $T(G)_u$:
Is there an antisymmetric morphism fulfilling the Jacobi identity
$
[\bullet, \bullet]_{T(G)_u} : T(G)_u \times T(G)_u \to T(G)_u
$
such that $[\chi,\xi]_{1\to T(G)_u} = \langle \chi, \xi \rangle [\bullet, \bullet]_{T(G)_u}$?
We doubt that there is such an internal Lie-bracket this Lie-Algebra structure comes from. The reason for this is that the external Lie-bracket is defined on left-invariant vector fields and this only gives a Lie-bracket that is a set map. This does not relate to the type of objects in an arbitrary Cartesian tangent category. In some well known special cases it does work:
\begin{corollary}\label{cor:Diffgeo_Alggeo_Lie_theory} ~
\begin{enumerate}[(i)]
\item Differential geometry: Let $G$ be a Lie group (a group object in the category of smooth manifolds). Then its tangent space $T(G)_u$ at the unit has an internal Lie-bracket $T(G)_u \times T(G)_u \to T(G)_u$ that is bilinear with respect to the addition of tangent vectors.
\item Algebraic geometry: Let $G$ be an algebraic Group (a group object in the category of algebraic varieties) . Then its tangent space $T(G)_u$ at the unit has an internal Lie-bracket $T(G)_u \times T(G)_u \to T(G)_u$  that is bilinear with respect to the addition of tangent vectors.
\end{enumerate}
\end{corollary}

The key property in the two setups of Corollary \ref{cor:Diffgeo_Alggeo_Lie_theory} is that an arbitrary linear map of elements is a morphism in the category and that two such morphisms coincide if they coincide on elements.

\begin{proof}
Part (i) follows because the category of smooth manifolds is a tangent category in which tangent spaces exist. In this category the group objects are Lie groups. The Lie bracket, defined on points, is a linear map. Since it is linear, it is a smooth map $T(G)_u \otimes T(G)_u \to T(G)_u$ and therefore the Lie bracket is a smooth map between manifolds.
Part (ii) follows because the category of algebraic varieties and algebraic maps is a tangent category in which tangent spaces exist. The group objects in this category are algebraic groups. The Lie bracket is a linear map. As it is linear, it is algebraic and therefore the Lie bracket is an internal morphism.
\end{proof}

Part (i) is the recovered classical result from differential geometry, however in more generality.

In classical differential geometry the Lie algebra determines the Lie group locally via the exponential map. We doubt that this is the case in full generality because in general there is no exponential morphism in an arbitrary tangent category.

\section{Principal bundles}\label{sec:principal_bundles_part}
Physical systems that have an inherent symmetry or redundancy are often described as principal $G$-bundles where $G$ is the system's symmetry group.
For example, the electromagnetic interaction between charged particles naturally displays symmetry of the complex unit circle, $U(1)$.  The details of this example can be found in \cite{Nakahara_geo_top_phy_2003}*{Chapter 10}.  The electromagnetic field has an associated $U(1)$-principal bundle.  This bundle is important in electromagentism - the electromagnetic potential is the expression in local coordinates of a connection on this bundle.  Physically inspired examples such as this one make it clear that principal bundles are important.

The central idea of principal bundles is that locally they are the Cartesian product $M \times G$ of a base-manifold $M$ and the symmetry group $G$, while globally they may have a nontrivial geometry.
The purpose of this section is to generalize this geometric construction to an arbitrary join restriction category. The structure of join restriction categories allows for a particularly nice realization of fibre bundles in which local trivializations are expressed as partial isomorphisms.

\subsection{Fibre Bundles}\label{sec:fibre_bundles}
In differential geometry, given manifolds $M$ and $F$, a fibre bundle over $M$ with typical fibre $F$ is a manifold $E$ with a smooth map $q:E \to M$ such that for every $p\in M$ there is an open neighbourhood $U\ni p$ such that $q^{-1}(U)$ is diffeomorphic to $U \times F$. 
In this section we generalize this definition into the setting of join-restriction categories.

In the following definition it will be useful to recall from Definition \ref{def:total_map_partial_inverse} that a total morphism is a morphism $f:A \to B$ in a restriction category fulfilling $\bar f = 1_A$ and that a morphism $g:A \to B$ is called a partial isomorphism if it has a partial inverse, denoted $g^*$. Recall in addition from Remark \ref{rmrk:commutative_diagrams_in_restriction} that a diagram commutes if the compositions along all paths are equal. Two partial maps are equal if both their values and their domains coincide.
\begin{definition}\label{def:fibre_bundle}
Let $M$ and $F$ be objects in a join restriction category $\mathbb X$.
\begin{enumerate}[(i)]
\item a \textbf{fibre bundle} over $M$ with typical fibre $F$ in a join restriction category $\mathbb X$ is an object $E$ with a total morphism $q:E \rightarrow M$ and partial isomorphisms $\alpha_i: E \rightarrow M \times F$, called \textbf{local trivializations}, such that the diagram
\begin{equation}\label{diagram:fibre_bundle_alpha_i}
\begin{tikzcd}
	E && {M \times F} \\
	E && M
	\arrow["{\alpha_i}", from=1-1, to=1-3]
	\arrow["{\bar \alpha_i}"', from=1-1, to=2-1]
	\arrow["{\pi_0}", from=1-3, to=2-3]
	\arrow["q"', from=2-1, to=2-3]
\end{tikzcd}
\end{equation}
commutes and $\bigvee_{i \in I} \bar \alpha_i = 1_E$ and there exists a restriction idempotent $e_i: M \to M$ such that $\bar \alpha_i^* = e_i \times 1_F$. 
\item A morphism between bundles $(E,M,F,q,(\alpha_i)_{i\in I}) \to (E',M',F',q',(\alpha'_k)_{k\in K})$ is a pair of morphisms $\Phi: E \to E', \phi : M \to M'$ such that
\begin{center}
\begin{tikzpicture}
\path 
(0,1.5) node (a) {$E$}
(3,1.5) node (b) {$E'$}
(0,0) node(c) {$M$}
(3,0) node (d) {$M'$}
;
\draw [->] (a) -- node[above] {$\Phi$} (b);
\draw [->] (c) -- node[below] {$\phi$} (d);
\draw [->] (a) -- node[left] {$q$} (c);
\draw [->] (b) -- node[right] {$q'$} (d);
\end{tikzpicture}
\end{center}
commutes.
\item A fibre bundle $(E,M,q,(\alpha_i)_{i \in I})$ is called \textbf{totally fibred} if, for all $i \in I$,
$$
\bar \alpha_i = \overline{qe_i}.
$$
\end{enumerate}
\end{definition}
The morphisms $\Phi$ and $\phi$ forming the principal bundle morphisms are not total, in general.
Fibre bundles and fibre bundle morphisms, as in Definition \ref{def:fibre_bundle}, form a join restriction category $FibBun(\mathbb X)$.
The restrictions structure is given by $\overline {(\Phi , \phi )} = (\bar \Phi, \bar \phi)$ and the join is $\bigvee_i (\Phi_i, \phi_i)= (\bigvee_i \Phi_i, \bigvee_i \phi_i)$.

The condition for a bundle to be totally fibred will be relevant in Section \ref{sec:classical_bundles}. 
Remember from Remark \ref{rmk:partial_set_map_interpretations},(ii) that the restriction idempotent of a composition can be interpreted as a preimage. 
In the category \ParMfld{} of smooth manifolds and principal bundles total fibredness describes that the local trivialization $\alpha_i: E \to M \times F$ is not just defined on an arbitrary open subset $V \subset E$, but that, in fact, it is of the form $q^{-1}(U)$ for some $U \subset M$.

\begin{remark}\label{rmk:totally_fibred_equivalent}
    An equivalent characterization of a fibre bundle that is totally fibred is that the diagram
    \[\begin{tikzcd}
    	E & {M \times G} \\
    	M & M
    	\arrow["{\alpha_i}", from=1-1, to=1-2]
    	\arrow["q"', from=1-1, to=2-1]
    	\arrow["{\pi_0}", from=1-2, to=2-2]
    	\arrow["{e_i}"', from=2-1, to=2-2]
    \end{tikzcd}\]
    commutes. This shows how total fibredness forces the morphism $\alpha_i$ to be total on fibres and the partiality to be contained in the basespace component.
\end{remark}
The formal reason why the diagram in Remark \ref{rmk:totally_fibred_equivalent} holds, is that Diagram \ref{diagram:fibre_bundle_alpha_i} commutes, that in the totally fibred case $\bar \alpha_i = \overline{qe_i}$ and that with this the left path in the diagram simplifies to $q \bar e_i$. It is an equivalent condition since taking the restriction idempotent of $\alpha_i \pi_0 = q e_i$ gives exactly the totally fibred condition $\bar \alpha_i = \overline{q e_i}$.

Recall from Definition \ref{def:atlas} that an atlas is an index set $I$ together with objects $(U_i)_{i \in I}$ and morphisms $u_{ij}:U_i \to U_j$ for each $i,j \in I$.
The object $M\times F$ together with a collection of morphisms $u_{ij} = \alpha_i^* \alpha_j : M \times F \to M \times F$, indexed by the index set $I$ of $(\alpha_i)_{i \in I}$ and given by the composition of the partial inverses $\alpha_i^*$ and the local trivializations $\alpha_j$, form an atlas. The atlases produced by this construction are the same as the atlases which are characterized in Definition \ref{def:fibre_atlases}.  
 Theorem \ref{corollary:equivalence_fibre}, which states that if gluings exist, bundle atlases are equivalent to fibre bundles, will provide a proof of this statement. First, we need some preliminary definitions and properties.

\begin{definition}\label{def:fibre_atlases}
Let $M$ and $F$ be objects in a join restriction category $\mathbb X$. Then a \textbf{bundle atlas} is an atlas $(u_{ij}: M \times F \to M\times F)_{i,j \in I}$ such that 
\begin{itemize}
\item for all $i,j \in I$, $u_{ij} = \langle\pi_0, u_{ij} \pi_1\rangle$, and
\item for $i=j$, $ u_{ii} = e_{i} \times 1_F$ for a restriction idempotent $e_{i} = \bar e_{i}$.
\end{itemize}
\end{definition}

Here the expression $u_{ij} = \langle \pi_0, u_{ij} \pi_1\rangle$ does not mean that $u_{ij} \pi_0 = \pi_0$ or that $u_{ij}$ is total in some first component sense. In fact, due to the definition of restriction limits in Definition \ref{def:restriction_limit}(ii), $\langle f, g\rangle \pi_0 = \bar g f$ which in our case means $u_{ij} \pi_0 = \bar{u_{ij}} \pi_0$.

Recall the notion of a gluing from Definition \ref{def:gluing}.  In the following lemma, we show that $E$ is the gluing of a bundle atlas.
\begin{lemma}\label{lemma:fibrebundleatlasonedirection}

For a fibre bundle $(E,M,q,F,(\alpha_i)_{i \in I})$ in a join restriction category, $E$ is a gluing of the atlas $u_{ij} = \alpha_i^* \alpha_j$. This atlas is a bundle atlas as defined in Definition \ref{def:fibre_atlases}.
\end{lemma}
\begin{proof}
The partial isomorphisms $\alpha_i^*: M \times F \to E$ fulfill the required properties of Definition \ref{rmk:gluing_explicit} because the atlas $u_{ij}$ was constructed from the partial isomorphisms $\alpha_i$. The components $g_j$ from Remark \ref{rmk:gluing_explicit} are the inverses $\alpha_i^*$. They fulfill \ref{rmk:gluing_explicit}(ii) by the definition of $u_{ij}$. They fulfill (iii) because $\bigvee_{i \in I} \bar \alpha_i = \bigvee_{i \in I} \alpha_i \alpha_i^* = 1_E$.

For any $i,j \in I$ is of the form $u_{ij} = (\pi_0, u_{ij} \pi_1)$ since 
$$
u_{ij} \pi_0 = \alpha_{i}^* \alpha_j \pi_0 = 
\alpha_i^* \bar \alpha_j q = 
\overline{\alpha_i^* \bar \alpha_j} \bar \alpha_i^* \alpha_i^* q = 
\overline{\alpha_i^*  \alpha_j}  \alpha_i^* \bar \alpha_i q = 
\bar{\alpha_i^*}  \alpha_i^* \alpha_i \pi_0  = 
\bar{\alpha_i^*}  \bar \alpha_i^* \pi_0 .
$$
For $i=j$, it is of the form $ u_{ii} = e_{i} \times 1_F$ for an $e_{i} = \bar e_{i}$ as
$$
u_{ii}= \alpha_i^* \alpha_i =\bar \alpha_i^* = e_i \times 1.
$$
\end{proof}
\begin{example}
A simple class of examples are the bundles that are given by the projection morphisms $\pi_0 : M \times F \to M$. In this case the structure is given by a single local trivialization:
$$
q=\pi_0 : M \times F \to M , \alpha_0 = 1_{M \times F}
$$
This bundle's atlas is $\mathrm{At}(1_{M \times F})$, consisting of the object $M \times F$ and the identity morphism $1_{M \times F}$.
\end{example}
A well known nontrivial example is the Möbius strip \ParMfld{}.
\begin{example}
Let $E=\mathrm{Mo}$ be the Möbius strip. While defined as an abstract manifold, it can be embedded into $\mathbb R^3$ as
$$
\left\lbrace \left( \sin(2\alpha)(3\!+\!b\sin(\alpha)), \cos(2\alpha)(3\!+\!b\sin(\alpha)) , b \cos(\alpha) \right)\!\in\!\mathbb R^3 \! \left| \! \alpha \!\in\! [0,2 \pi] , b \!\in\! \left(-1,1 \right) \right\rbrace  \right.\!\subset\!\mathbb R^3.
$$
Let $M=S^1$ be the circle and $F = (-1,1)$ the open interval.
\begin{figure}
    \centering
    \includegraphics[width=  \linewidth]{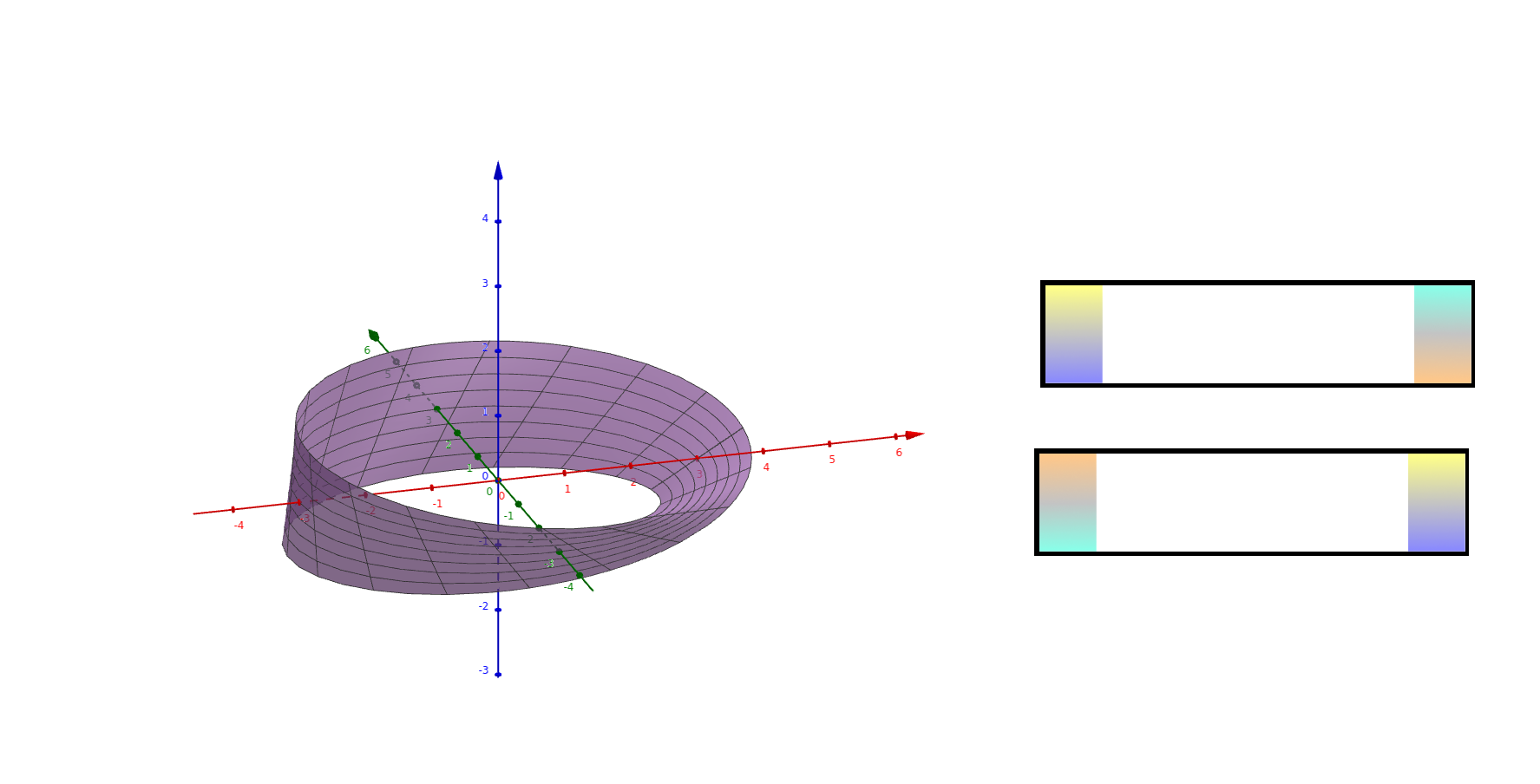}
    \caption{On the left is an embedding of the moebius strip into $\mathbb R^3$. On the left is the atlas, out of which it can be glued together. The two stripes are the domains of $u_{00}$ and $u_{11}$ respectively. The colours show how the change of chart partial isomorphisms $u_{ij}$ glue together parts of their domains. The figure was created with GeoGebra3D, Gimp and LibreOffice.}
    \label{fig:moebius_strip}
\end{figure}
Then there is a map $E \to S^1$ given by the angle to the vertical axis. If you embed $S^1$ into $\mathbb R^2$ as 
$$S^1=\{(\sin(\gamma),\cos(\gamma)) \in \mathbb R^2 | \gamma \in [0,2 \pi]\} \subset \mathbb R^2,$$ this projection map sends a point $(x,y,z) \in \mathrm{Mo} \subset \mathbb R^3$ to 
$$(x/\sqrt{x^2+y^2} , y / \sqrt{x^2 + y^2})\in S^1 \subset \mathbb R^2.$$ Let $\varepsilon$ be in $(0,\pi/2)$, for example 1/2. The maps $\alpha_0$ and $\alpha_1$ are the partial isomorphisms that identify the left and the right half of the Möbius strip with $(-\varepsilon , \pi + \varepsilon) \times F$ and $(\pi-\varepsilon , 2\pi + \varepsilon) \times F$, both seen as a subset of $S^1 \times F$.

The moebius strip is defined by an atlas, $U_0 = U_1 = S^1 \times F$, where we describe the circle by its angle
$S^1 = \mathbb R/\sim$ with $2 \pi n \sim 0 \, \forall n \in \mathbb N$.  The change of chart maps $u_{ij}$ are 
$$
u_{10}(x,y) = \left\lbrace
\begin{matrix}
(x,y) & \text{if }x \in(-\epsilon,\epsilon)\\
(x,-y) & \text{if }x \in(\pi-\epsilon, \pi + \epsilon)
\end{matrix}\right., \quad
u_{00} = 1|_{(-\epsilon,\pi+\epsilon)\times F},
$$$$
u_{01}(x,y) = \left\lbrace
\begin{matrix}
(x,y) & \text{if }x \in(-\epsilon,\epsilon)\\
(x,-y) & \text{if }x \in(\pi-\epsilon, \pi + \epsilon)
\end{matrix}
\right.,\text{ and }
u_{11}= 1|_{(\pi-\epsilon, 2\pi+\epsilon)\times F}
.
$$
This atlas can be seen as having a straight connection on the one side and a twisted connection on the other side. Its gluing is the Möbius strip $\operatorname{Mo}$, as depicted in Figure \ref{fig:moebius_strip}.
\end{example}

If we are in a category where gluings exist, the implication of Lemma \ref{lemma:fibrebundleatlasonedirection} becomes a one to one correspondence between atlases and bundles. In the following two lemmas we will establish more precisely what that means.  First, we show that one can produce a fibre bundle from an atlas.

\begin{lemma}\label{lemma:fibrebundleatlas} Let $M$ and $F$ be objects in a join restriction category $\mathbb X$.
Let $u_{ij}: M \times F \to M \times F$ be a bundle atlas.
Then if the gluing of the atlas $(M \times F ,u_{ij})$ exists (see definition \ref{def:gluing}), it is a fibre bundle $E \xrightarrow{q} M$ over $M$ with typical fibre $F$ such that
$$
u_{ij} = \alpha_i^* \alpha_j
$$
\end{lemma}
\begin{proof}
Let $E$ be a gluing and let $(\alpha_i^*)_{i \in I}: (M\times F, u_{ij}) \to \mathrm{At}(1_E)$ be the resulting bundle morphism as in Definition \ref{def:gluing}. 
We must check that the $\alpha_i^*$ fulfill the conditions in Definition \ref{def:fibre_bundle}. Due to the remarks following Definition \ref{def:gluing}, the morphisms $\alpha_i^*$ are partial isomorphisms. We denote their partial inverses as $\alpha_i$. In addition the remarks following Definition \ref{def:gluing} guarantee that
 $\bigvee_i \alpha_i  \alpha_i^* = 1$ and $u_{ij} = \alpha_i^* \alpha_j$. Due to the second condition in Definition \ref{def:fibre_atlases}, $\bar \alpha_i^*  = \bar u_{ii} = (e_{ii} \times 1)$.

Since $f_i = \bar u_{ii}\pi_0: M \times F \to M$ is an atlas morphism from ($M \times F , u_{ij}$) to the atlas $\mathrm{At}(1_M)$ there is an induced morphism $q: E \to M$ such that $f_i = \alpha_i^* q$. Precomposing $f_i=\alpha_i^*q$ with $\alpha_i$ ensures that Diagram \ref{diagram:fibre_bundle_alpha_i} commutes. Taking the join, $q = \bigvee_i (\alpha_i \pi_0)$, thus $\bar q = \bigvee_i \bar{\alpha_i} =1_E$. Thus $q$ is total.

Therefore $(E,M,q,(\alpha_i)_{i \in I})$ is a fibre bundle.
\end{proof}

Next, we show that fibre bundle morphisms correspond to atlas morphisms of the form $a_{ik} = (\pi_0 \phi, A_{ik} \pi_1 )$. This is the complement of Lemma \ref{lemma:fibrebundleatlas} on morphisms.

\begin{lemma}\label{lemma:fibreatlasmorphism} Let $\mathbb X$ be a join restriction category.
\begin{enumerate}[(i)]
\item For each fibre bundle morphism 
$$
(\Phi, \phi) : ( E,M,F,q,(\alpha_i)_{i \in I} )
\to ( E',M',F',q',(\alpha_k')_{k \in K} )
$$
the morphisms
$$
A_{ik} = \alpha_i^* \Phi \alpha_k' : M \times F \to M' \times F'
$$
are atlas morphism whose first components are $\phi$, i.e. atlas morphisms of the form $A_{ik}=(\pi_0 \phi, A_{ik} \pi_1)$ between the atlases given by $u_{ij} = \alpha_i^* \alpha_j$ and $u'_{kl} = \alpha'^*_k \alpha'_l$.


\item Let $u_{ij}$ and $u'_{kl}$ be fibre atlases which have gluings. For each atlas morphism 
$$
A_{ik}: (M \times F,u_{ij}) \to (M' \times F',u_{kl}')
$$
of the form $(\pi_0 \phi, A_{ik} \pi_1)$,
with $ u_{ij} $ and $u_{kl}'$
there is a unique morphism between their gluings $\Phi: G_U \to G_{U'}$ such that $(\Phi, \phi)$ is a fibre bundle morphism.
$$
A_{ik} = \alpha_i^* \Phi \alpha_k : M \times F \to M' \times F'
$$
\end{enumerate}
\end{lemma}
\begin{proof}
\begin{enumerate}[(i)]
\item We have to check the 5 conditions in the definition of an atlas morphism. All of them follow directly from the definition $A_{ik} = \alpha_i^* \Phi \alpha_k'$.
\begin{align*}
u_{ii}A_{ik} &= \alpha_i^* \alpha_i \alpha_i^* \Phi \alpha_k' = \alpha_i^* \Phi \alpha_k' = A_{ik}
\\
A_{ik}u_{kk}' &= \alpha_i^* \Phi \alpha_k' \alpha_k'^* \alpha_k' = \alpha_i^* \Phi \alpha_k' = A_{ik}
\\
u_{ij} A_{jk} &= \alpha_i^* \alpha_j \alpha_j^* \Phi \alpha_k' = \alpha_i^* \bar \alpha_j \Phi \alpha_k' \leq \alpha_i^* \Phi \alpha_k' = A_{ik}
\\
A_{ik} u_{kl}' &= \alpha_i^* \Phi \alpha_k' \alpha_k'^* \alpha_l' = \alpha_i^* \Phi \bar \alpha_k' \alpha_l' \leq \alpha_i^* \Phi \alpha_k'
\\
A_{ik} u_{kl}' &= \alpha_i^* \Phi \bar \alpha_k' \alpha_l' = \overline{\alpha_i^* \Phi \alpha_k'} \alpha_i^* \Phi \alpha_l' = \bar A_{ik} A_{il}
\end{align*}
Since $(\Phi, \phi)$ is a fibre bundle morphism, 
$$
A_{ik} \pi_0 = \alpha_i^* \Phi \alpha'_k \pi_0  \leq \alpha_i^* \Phi q' = \alpha_i^* q \phi \leq \pi_0 \phi
$$
Thus $A_{ik}= \bar A_{ik}(\pi_0 \phi, A_{ik} \pi_1)= (\pi_0 \phi, A_{ik} \pi_1)$.
\item Let $G_U$ be the gluing of $u_{ij}$ and $G_{U'}$ the gluing of $u_{kl}$. Because of the previous lemma, they form a fibre bundle. Since $A_{ik}$ is an atlas morphism from $u_{ij}$ to $u_{kl}'$ and $\alpha_k^*$ form an atlas morphism $u_{kl}'$ to $\mathrm{At}(E')$, their composition has target $\mathrm{At}(E')$. Therefore, by the universal property of the gluing $E$, there is a unique morphism $\Phi: E \to E'$ fulfilling 
$$
\alpha_i^* \Phi = A_{ik} \alpha_k^*
$$
Thus 
\begin{align*}
\alpha_i^* \Phi \bar \alpha_k q' = A_{ik} \alpha_k^* \bar \alpha_k q' = A_{ik} \alpha_k^* \alpha_k \pi_0 = \overline{A_{ik} \alpha_k^*} A_{ik} \pi_0 = \overline{A_{ik} \alpha_k^*} \pi_0 \phi
\end{align*}
Precomposing with $\alpha_i$ gives:
$$
\bar \alpha_i \Phi \bar \alpha_k q' = \overline{\alpha_i A_{ik} \alpha_k^*} \alpha_i \pi_0 \phi
$$
As $\alpha_i \pi_0 = \bar \alpha_i p$ this becomes 
$$
\overline {\bar \alpha_i  \Phi \alpha_k} \Phi q' = \overline{\alpha_i A_{ik} \alpha_k^*} q \phi
$$
Because $\alpha_i^* \Phi = A_{ik}\alpha_k^*$, we obtain
$$
\overline {\alpha_i A_{ik} \alpha_k^*} \Phi q' = \overline {\alpha_i A_{ik} \alpha_k^*} \pi_0 \phi
$$
and taking the join over all $i$ and $k$ gives that $\Phi p' = p \phi$.
\end{enumerate}
\end{proof}
Lemma \ref{lemma:fibrebundleatlas} and Lemma \ref{lemma:fibreatlasmorphism} imply together that there is an equivalence of categories between fibre bundles and atlases that fulfill special conditions.
Recall from the beginning of section \ref{sec:fibre_bundles} that we call the category of fibre bundles with fibre bundle morphisms FibBun($\mathbb X$).
We also define the category BunAtlas($\mathbb X$) which has
\begin{itemize}
\item as objects atlases fulfilling the conditions of Lemma \ref{lemma:fibrebundleatlas}
\item as morphisms fibre bundle morphisms fulfilling the conditions of Lemma \ref{lemma:fibreatlasmorphism}.
\end{itemize}

\begin{theorem}[Classification of fibre bundles]\label{corollary:equivalence_fibre}
Let $\mathbb X$ be a join restriction category with gluings.
There is an equivalence of categories between 
FibBun($\mathbb X$) and BunAtlas($\mathbb X$)
\end{theorem}
\begin{proof}
We consider the functor sending an atlas in BunAtlas to its gluing. According to Lemmas \ref{lemma:fibrebundleatlasonedirection} and \ref{lemma:fibrebundleatlas} this is essentially surjective and according to Lemma \ref{lemma:fibreatlasmorphism} it is fully faithful.\end{proof}
\subsection{G-bundles and principal G-bundles}\label{sec:g-bundles}
Now we will take a group object $G$ which, by Definition \ref{def:group_object}, is an object with multiplication $m: G \times G \to G$, unit $u:1 \to G$ and inverse $\iota: G \to G$. We use this group object to determine the structure of a fibre bundle by acting on the fibre.
Having such an action, the fibre is a $G$-object:
\begin{definition}
Let $G$ be a group object in a Cartesian category $\mathbb X$. 
\begin{enumerate}[(i)]
    \item 
A \textbf{$G$-object} is an object $F$ together with a left $G$-action, i.e. a total morphism $a: G \times F \to F$ fulfilling 
\[\begin{tikzcd}
	{G \times G \times F} & {G \times F} & F & {G \times F} \\
	{G \times F} & F && F
	\arrow["{1_G \times a}", from=1-1, to=1-2]
	\arrow["{m \times 1_F}"', from=1-1, to=2-1]
	\arrow["a", from=1-2, to=2-2]
	\arrow["{u \times 1_F}", from=1-3, to=1-4]
	\arrow["{1_F}"', from=1-3, to=2-4]
	\arrow["m", from=1-4, to=2-4]
	\arrow["a"', from=2-1, to=2-2]
\end{tikzcd}\]
    \item A \textbf{morphism of $G$-objects} from $(F,a_F)$ to $(E,a_E)$ is a morphism $f: F \to E$ in the underlying restriction category, such that 
    \[\begin{tikzcd}
	{G \times F} & {G\times E} \\
	F & E
	\arrow["{1_G \times f}", from=1-1, to=1-2]
	\arrow["{a_F}"', from=1-1, to=2-1]
	\arrow["{a_G}", from=1-2, to=2-2]
	\arrow["f"', from=2-1, to=2-2]
    \end{tikzcd}\]
    commutes.
    \item With the composition, unit and restriction of the underlying category, $G\mathrm{-obj}(\mathbb X)$ is the category of $G$-objects and morphisms of $G$-objects in a given restriction category $\mathbb X$. 
\end{enumerate}

\end{definition}
This definition is analogous to the beginning of Section 5.10 in \cite{michor2008topics}.
In topology and differential geometry, $G$-spaces are commonly used and they denote G-objects in Top and Mfld respectively.
\begin{example}
Given a topological group $G$, a $G$-space is a topological space $Y$ equipped with a continuous action map $a: G \times Y \to Y$, often denoted like $ a(g,y)=: g \rhd y$ such that $1 \rhd y = y$, and $(g \cdot g') \rhd y = g \rhd (g' \rhd y)$ for all $g,g' \in G$ and $y \in Y$. In the category $\mathrm{Top}$ of topological spaces and continuous maps, $G$-spaces are exactly $G$-objects.
\end{example}
An important example is the group object itself.
\begin{example}
A group object $G$ is an example of a $G$-object. The action is given by
$$
a = m : G \times G \to G .
$$
\end{example}
Now we are going to use these in a Cartesian join restriction category. There we will always demand for the morphisms $u,m,\iota,a$ to be total.
\begin{definition}
Let $G$ be a group object and $M$ any other object in a join restriction category $\mathbb X$. Then
a  \textbf{$G$-atlas} on $M$ consists of partial morphisms
$\tau_{ij}: M \to G$ such that
$$
(\tau_{ij}, \tau_{jk})m \leq \tau_{ik} \qquad \tau_{ii} \leq \, !u \qquad \tau_{ji} = \tau_{ij} \iota
$$
\end{definition}

\begin{definition}\label{def:G-bundle}
Let $G$ be a group object in a Cartesian join restriction category and $M$ another object. Then
\begin{enumerate}[(i)]
\item A \textbf{$G$-bundle} over $M$ with typical fibre $F$ consist of a $G$-object $F$
a fibre bundle $E \to M$ with fibre $F$ and a $G$-atlas $\tau_{ij}$ on $M$ such that 
$$
u_{ij} = \alpha_i^* \alpha_j = (\pi_0, (\pi_0 \tau_{ji}, \pi_1)a ): M\times F \to M \times F
$$
\item A \textbf{principal $G$-bundle} over $M$ is a $G$-bundle with fibre $G$ and the multiplication from the left $a=m: G \times G \to G$ as the action.
\end{enumerate}
\end{definition}
In Definition \ref{def:PBun_morphism} we will define morphisms in a way so that we will obtain a join-restriction category of $G$-bundles and a join-restriction category of principal $G$-bundles.

First we construct a $G$-action on a principal bundle:
\begin{theorem}[Total right action on principal bundles]\label{thm:right_action_principal}
Let $P \to M$ be a principal $G$-bundle. Then there exists a total morphism (a right action)
$$
r: P \times G \to P
$$
such that 
$$
(r \times 1)r = (1 \times m)r: P \times G \times G \to P, \quad (1 \times e) r = 1_P
\quad \text{and} \quad 
rq = \pi_0 q: P \times G \to M.
$$
\end{theorem}
\begin{proof}
In order to define the morphism $r$ we define the partial morphisms 
$$
r_i = (\alpha_i \times 1_G) (1_M \times m) \alpha_i^* : P \times G \to G.
$$
This means, $r_i$ is the following composition:
$$
P \times G \xrightarrow{\alpha_i \times 1_G} M \times G \times G \xrightarrow{1_M \times m} M \times G \xrightarrow{\alpha_i^*} P
$$
If the join $r = \bigvee_{i \in I} r_i$ is defined, it will automatically fulfill all the required properties, as its components fulfill it. It will be a total morphism as
\begin{align*}
\bar r_i &= \overline{(\alpha_i \times 1)(1 \times m) \alpha_i^*}
\\&=\overline{(\alpha_i \times 1_G)(1_M \times m) \bar \alpha_i^*}
\\&=\overline{(\alpha_i \times 1_G)(1_M \times m) (e \times 1_G)}
\\&=\overline{(\alpha_i \times 1_G) (e \times 1_{G\times G})(1 \times m)}
\\&= \overline{(\alpha_i \times 1_G) (e \times 1_{G\times G})}
\\&= \overline{(\alpha_i \times 1_G) (\bar \alpha_i^* \times 1)}
\\&= \overline{(\alpha_i \times 1_G) (\alpha_i^* \times 1)}
\\&= \overline{(\bar \alpha_i \times 1_G)}
\\&= \bar \alpha_i \times 1_G
\end{align*}
and $\bigvee_i (\bar \alpha_i^* \times 1_G) = \bigvee_i \bar \alpha^*_i \times 1_G = 1_{P} \times 1_G$, and thus 
$$
\overline {\bigvee r_i} = \bigvee \bar r_i = \bigvee (\bar \alpha_i \times 1) = 1_{P \times G}.
$$
Let's now prove that the join exists, i.e. that $r_i \smile r_j$. For this we need to show that $\bar r_i r_j = \bar r_j r_i$. Observe that
\begin{align*}
\bar r_i r_j &= \overline{(\alpha_i \times 1_G) (1_M \times m) \alpha_i^*} (\alpha_j \times 1_G) (1_M \times m) \alpha_j^*
\\
&= \overline{(\bar \alpha_i \times 1_G)(\alpha_i \times 1_G) (1_M \times m) \alpha_i^*} \, \overline{(\alpha_j \times 1_G)} (\alpha_j \times 1_G) (1_M \times m) \alpha_j^*
\\
&= \overline{(\bar \alpha_j \times 1_G)(\alpha_i \times 1_G) (1_M \times m) \alpha_i^*}  (\bar \alpha_i \alpha_j \times 1_G) (1_M \times m) \alpha_j^*
\\
&= \overline{( \alpha_j u_{ji}  \times 1_G) (1_M \times m) \alpha_i^*}  (\bar \alpha_i  \alpha_j \times 1_G) (1_M \times m) \alpha_j^*
\\
&= \overline{( \alpha_j \times 1_G) (\pi_0,(\pi_0 \tau_{ij} , \pi_1)m, \pi_2) (1_M \times m) \alpha_i^*}  (\bar\alpha_i \alpha_j \times 1_G) (1_M \times m) \alpha_j^* . 
\intertext{By associativity this equals}
&= \overline{( \alpha_j \times 1_G) (\pi_0,\pi_0 \tau_{ij} , (\pi_1, \pi_2)m) (1_M \times m) \alpha_i^*}  (\bar \alpha_i  \alpha_j \times 1_G) (1_M \times m) \alpha_j^*
\\&= \overline{( \alpha_j \times 1_G) (1 \times m)((\pi_0,\pi_0 \tau_{ij} ) \times 1_G) (1_M \times m) \alpha_i^*}  (\bar \alpha_i  \alpha_j \times 1_G) (1_M \times m) \alpha_j^* . 
\intertext{As the multiplication with $\tau_{ij}$ gives $u_{ij}=\alpha_j^* \alpha_i \alpha_i^*$ this equals}
&= \overline{( \alpha_j \times 1_G) (1 \times m) \alpha_j^* \alpha_i \alpha_i^*}  (\bar \alpha_i  \alpha_j \times 1_G) (1_M \times m) \alpha_j^*
\\&= \overline{( \alpha_j \times 1_G) (1 \times m) \alpha_j^* \bar \alpha_i}  (\alpha_i u_{ij} \times 1_G)  (1_M \times m) \alpha_j^*
\\&= \overline{( \alpha_j \times 1_G) (1 \times m) \alpha_j^* \bar \alpha_i}  (\alpha_i  \times 1_G) ((1, \tau_{ij})\times 1_G \times 1_G) (1_M \times m \times 1_G)  (1_M \times m) \alpha_j^*.
\intertext{By associativity this is}
&= \overline{( \alpha_j \times 1_G) (1 \times m) \alpha_j^* \bar \alpha_i}  (\alpha_i  \times 1_G) ((1_M, \tau_{ij})\times m)(1_M \times m) \alpha_j^*
\\&= \overline{( \alpha_j \times 1_G) (1 \times m) \alpha_j^* \bar \alpha_i}  (\alpha_i  \times 1_G) (1_M \times m) \alpha_i^* \alpha_j \alpha_j^*
\\&= \overline{( \alpha_j \times 1_G) (1 \times m) \alpha_j^* \bar \alpha_i}  (\alpha_i  \times 1_G) (1_M \times m) \alpha_i^* \bar \alpha_j
\\&\leq \overline{( \alpha_j \times 1_G) (1 \times m) \alpha_j^*}  (\alpha_i  \times 1_G) (1_M \times m) \alpha_i^* 
\\&=\bar r_j r_i.
\end{align*}
Thus $\bar r_i r_j \leq \bar r_j r_i$. As $i$ and $j$ were arbitrary, we also obtain $\bar r_i r_j \geq \bar r_j r_i$. Thus they are equal.
\end{proof}
It is not possible to define such a right $G$ action for arbitrary $G$-bundles because $F$ does not have a right action, in general.

Now let us define morphisms for $G$-bundles and principal $G$-bundles.
\begin{definition}\label{def:PBun_morphism}
For a fixed group-object $G$, \textbf{a morphism of principal $G$-bundles} is a morphism of fibre bundles
$$
(\Phi, \phi) : (P,M) \to (P',M')
$$
such that
$$
r_P \Phi = (\Phi \times 1_G) r_{P'}
$$
i.e. the right action commutes with the morphism.
\end{definition}
In a join-restriction category $\mathbb X$, the principal $G$-bundles and morphisms of principal $G$-bundles form the restriction category $\mathrm{PBun}_G(\mathbb X)$.

We will now relate this to the way the principal bundle is built out of an atlas.
\begin{proposition}\label{prop:principal_morphism} Let $G$ be a group object in a join-restriction category $\mathbb X$.
Giving a morphism of principal $G$-bundles  $(\Phi, \phi) : (P , M)$ to $(P' \to M')$ is equivalent to giving
\begin{itemize}
\item the morphism $\phi: M \to M'$
\item and morphisms $T_{ik}: M \to G$ such that
\begin{itemize}
    \item $T_{ik} \geq (T_{jk},\tau_{ij})m$
    \item $T_{ik} \geq (\phi \tau_{lk}', T_{il})m$
    \item $\bar T_{ik} T_{ih} = (\phi \tau_{kh},T_{ik})m$ (this actually implies the second point).
\end{itemize}
\end{itemize}
The composition is given by:
$$
(\phi,T_{ik})(\tilde \phi , \tilde T_{kl}) = (\phi \tilde \phi, \bigvee_k (T_{ik} , T_{kl})m)
$$
\end{proposition}
The idea of this is that $\Phi$ is locally given by the left-multiplication with $T_{ik}$.
\begin{proof}
Let $\Phi,\phi$ be a principal bundle morphism from $P$ to $P'$. Then $\Phi$ is a morphism between gluings and therefore induces an atlas morphism $A_{ik}: M\times G \to M' \times G$ between the atlas $u_{ij}$ of $P$ and the atlas $u_{kl}$ of $P'$. Explicitly $A_{ik}=\alpha_i^* \Phi \alpha_k$. As the diagram
\begin{center}
\begin{tikzpicture}
\path 
(0,1.5) node (a1) {$M \times G$}
(3,1.5) node (a2) {$P$}
(6,1.5) node (b1) {$P'$}
(9,1.5) node (b2) {$M \times G'$}
(0,0) node (c1) {$M\times G$}
(3,0) node (c2) {$P$}
(6,0) node (d1) {$P'$}
(9,0) node (d2) {$M'\times G$}
;
\draw [->] (a1) -- node[above] {$\alpha_i^*$} (a2);
\draw [->] (a2) -- node[above] {$\Phi$} (b1);
\draw [->] (b1) -- node[above] {$\alpha_k$} (b2);
\draw [->] (c1) -- node[above] {$=$} (c2);
\draw [->] (c2) -- node[above] {$\phi$} (d1);
\draw [->] (d1) -- node[above] {$=$} (d2);
\draw [->] (a1) -- node[left] {$\pi_0$} (c1);
\draw [->] (a2) -- node[left] {$1_P$} (c2);
\draw [->] (b1) -- node[left] {$1_{p'}$} (d1);
\draw [->] (b2) -- node[left] {$\pi_0$} (d2);
\end{tikzpicture}
\end{center}
commutes in the sense that the lower path $\geq$ the upper path we know that $A_{ik} \pi_0 \leq \pi_0 \phi$. Thus $A_ik$ can be written in the form
$A_{ik}=(\pi_0 \phi, \tilde A_{ik})$, with $\tilde A_{ik}: M \times G \to G$. As the right action commutes with the morphism, we have that
$$
(\tilde A_ik \times 1_G)(1_M \times m) = (1_M \times m)\tilde A_ik .
$$
But because of the unit property of the group object $G$, we have that
$$
\tilde A_{ik} = (1_M \times u \times 1_G) (1_M \times m) \tilde A_{ik} = (1_M \times u \times 1_G)  (\tilde A_{ik}\times 1_G) (1_M \times m) = (T_{ik} \times 1_G) (1_M \times m)
$$
where $T_{ik} = (1_M \times u) \tilde A_{ik}: M \to G$ fulfills all three properties that we were asking for, because $A_{ik} = (\phi, (T_{ik}\times 1_G)m)$ is an atlas morphism.

Let now conversely $\phi$ and $T_{ik}$ be given. Then because of the properties of $T_{ik}$, $A_{ik}:=(\phi, (T_{ik}\times 1_G)m)$ is an atlas morphism, thus it induces a morphism $\Phi$ between the gluings.

The composition is given as the composition of atlas morphisms as in \cite{Cockett2014DifferentialST}.
\end{proof}
We now generalize the characterization from Proposition \ref{prop:principal_morphism} as the definition for morphisms between arbitrary (not just principal) $G$-bundles. 
The morphism $\psi: F \to F'$ allows to define morphisms between fibre bundles with different fibres, as long as $\psi$ preserves the group actions.
\begin{definition}
Let $(E,M,F,q,\alpha_i,\tau_{ij})$ and $(E',M',F',q',\alpha_i',\tau_{kl}')$ be $G$-bundles. Then a \textbf{morphism of $G$-bundles} between them consists of
\begin{itemize}
\item A morphism $\phi: M \to M'$
\item A morphism $\psi: F \to F'$ such that $(\psi \times 1_G)a' = a \psi$
\item A collection of morphisms $T_{ik}: M \to G$ such that
\begin{itemize}
    \item $T_{ik} \geq (T_{jk},\tau_{ij})m$
    \item $T_{ik} \geq (\phi \tau_{lk}', T_{il})m$
    \item $\bar T_{ik} T_{ih} = (\phi \tau_{kh},T_{ik})m$
\end{itemize}
\item A morphism $\Phi: E \to E'$ such that
\begin{center}
\begin{tikzpicture}
\path 
(0,1.5) node (a) {$M \times F$}
(5,1.5) node (b) {$M' \times F'$}
(0,0) node(c) {$E$}
(5,0) node (d) {$E'$}
;
\draw [->] (a) -- node[above] {$(\phi , (T_{ik} \times \psi)a' )$} (b);
\draw [->] (c) -- node[below] {$\Phi$} (d);
\draw [->] (a) -- node[left] {$\alpha_i$} (c);
\draw [->] (d) -- node[right] {$\alpha_k^*$} (b);
\end{tikzpicture}
\end{center}
commutes.
\end{itemize}
\end{definition}
We denote the restriction category of $G$-bundles and morphisms of $G$-bundles as $\mathrm{Bun}_G(\mathbb X)$.

Summarizing, we have the following categories which are all join restriction categories:
\begin{enumerate}[(i)]
\item FibBun$(\mathbb X)$: Objects are fibre bundles, morphisms are fibre bundle morphisms.
\item $G\mathrm{-obj}(\mathbb X)$ : Objects are $G$-objects, morphisms are morphisms commuting with the group action.
\item Bun$_G(\mathbb X)$ : Objects are G-bundles, morphisms are the G-bundle morphisms from the above definition.
\item PBun$_G (\mathbb X)$ : Objects are principal G-bundles, morphisms are the G-bundle morphisms from the above definition.
\end{enumerate}

Given a join-restriction category $\mathbb X$ with gluings, there are the following restriction functors (as defined in Definition \ref{def:restriction_categories}(iv)):
\begin{itemize}
\item fibre: Bun$_G \to G$-space that sends a $G$-bundle to its fibre $F$ and a morphism of $G$-bundles to its fibre component $\psi$.
\item principal: Bun$_G \to$ PBun$_G$ that sends a $G$-bundle $E$ to the principal bundle $P_E$ that is the gluing of the atlas 
$$
u_{ij}:=(\tau_{ij} \times 1_G ) (1_M \times m) : M \times G \to M \times G
$$
and a morphism $(\phi, \psi, T_{ik}, \Phi)$ to $(\phi, T_{ik})$ which is, according to Proposition \ref{prop:principal_morphism} a way of giving a principal bundle morphism.
\item build: PBun$_G \times G$-Space $\to$ Bun$_G$ that sends a principal bundle $P$ and a $G$-object $F$ to the $G$-bundle given as the gluing of the atlas 
$$
u_{ij}:= ((1_M, \tau_{ij}) \times 1_F) (1_M \times a): M \times F \to M \times F.
$$
where $(1_M,\tau_{ij}): M \to M\times G$ gives a group element that acts through $a: G \times F \to F$, the $G$-action on $F$.
For a morphism $((\phi,T_{ik}),(\psi))$ of PBun$_G \times G$-Space, consider the atlas morphism given by 
$$
A_{ik}=((\phi,T_{ik})\times \psi)(1_{M'} \times a')
$$
where $(\phi,T_{ik})\times \psi): M \times F \to M' \times G \times F'$ produces a group element that acts through $(1_{M'} \times a'):M' \times G \times F' \to M' \times F'$. This atlas morphism induces a morphism $\Phi$ between the gluings. The result of the functor is then the morphism of $G$-bundles given by $(\phi, \psi, T_{ik}, \Phi)$.
\end{itemize}
All this works as the $T_{ik}$ in the definition of a morphisms of $G$-bundles was required to fulfill the conditions of an atlas-morphism.
\begin{theorem}
The functors
\begin{align*}
\mathrm{Bun}_G \xrightarrow{(\mathrm{principal},\mathrm{fibre})} & \mathrm{PBun}_G \times G\mathrm{-space}
\\
\mathrm{Bun}_G \xleftarrow{\mathrm{build}} &\mathrm{PBun}_G \times G\mathrm{-space}
\end{align*}
form an equivalence of categories between $\mathrm{Bun}_G$ and $\mathrm{PBun}_G \times G\mathrm{-Space}$.
\end{theorem}
This means that all possible $G$-bundles can already be described by its fibre and a principal bundle containing the bundle information.
\begin{proof}
In order to prove that these functors form an equivalence, we need to show that their composition is naturally isomorphic to the identity.
Let $(P,F)\in \mathrm{PBun}_G \times G\mathrm{-Space}$ be given. Then build$(P,F)$ has fibre $F$ and the same $\tau_{ij}$ as $P$. Thus 
$$
(\mathrm{principal},\mathrm{fibre})(\mathrm{build}(P,F))=(\tilde P, F)
$$
where $\tilde P$ is the gluing of $(\tau_{ij}\times 1_G)(1_M \times m)$. But $P$ is also a gluing of this. Thus there is an isomorphism $\varphi: \tilde P \to P$ with $\alpha_i = \varphi \tilde \alpha_i$.

by Proposition \ref{prop:principal_morphism} this morphism is completely described by $(1_M,!u)$ (it is the identity atlas morphism, given by the multiplication with the group unit). Let $\phi,T_{ik}$ describe any principal bundle morphism. Then the compositions equal:
$$
(1_M, !u);(\phi,T_{ik}) = (\phi, T_{ik}) = (\phi, T_{ik})(1_{M'},!u)
$$
As build;principal does not change what a morphism does on atlases, this proves that $\varphi$ is natural. Thus $(\varphi \times 1_{G\mathrm{-Space}})$ is the required natural isomorphism.

In the other direction let $E$ be a $G$-bundle. Then $\mathrm{build}(\mathrm{fibre}(E),\mathrm{principal}(E))$ is by definition the gluing of the atlas 
$$
u_{ij}:= (\tau_{ij} \times 1_F)(1_M \times a): M \times F \to M \times F
$$
But so is $E$, so there is an isomorphism $\varphi_E: \mathrm{build}(\mathrm{fibre}(E),\mathrm{principal}(E)) \to E$. Naturality again follows by considering the atlas description of $\varphi_E$, where it is just the identity and therefore natural.
\end{proof}
\subsection{Examples}\label{sec:classical_bundles}
We will now explore how the principal bundles we defined relate to the classical notions of principal bundles in different aspects of geometry. Differential geometry, point-set topology and algebraic geometry each have their version of principal bundles. All these notions have in common that a principal bundle is locally isomorphic to the Cartesian product of a base space with a group.

We directly recover principal bundles from differential geometry and point-set topology. For algebraic geometry, future work is needed, but we strongly suspect that principal bundles in algebraic geometry can also be described using the language of restriction categories. A way to include Grothendieck topologies and, in particular, étale topology in restriction categories still needs to be further developed though.

\subsubsection{Principal and fibre bundles in differential geometry}\label{subsubsec:classical}
Here we will connect the notions defined above with the classical notions of fibre bundles and principal bundles from differential geometry.
\begin{definition}~
Let \ParMfld{} be the following join restriction category:
\begin{enumerate}[(i)]
\item Objects are smooth manifolds
\item Morphisms from a manifold $M$ to a manifold $M'$ are are smooth maps from an open subset $U \subset M$ to $M'$. Here the smooth structure on the open subset $U$ is defined as a restriction of the smooth structure of $M$.
\item Composition, identity and restrictions are analogous to the category of partial maps on sets. In particular the restriction of a map $f:A \to B$ that is defined on $U \subset A$ is the inclusion $U \to A$. The composition gives a smooth map defined on an open subset since the composition of smooth maps is smooth, composing partial maps intersects domains of definition and finite intersections of open sets are open.
\end{enumerate}
\end{definition}
This restriction category of manifolds with smooth partial maps is the category where the categorical constructions defined in Sections \ref{sec:fibre_bundles} and \ref{sec:g-bundles} correspond to constructions from classical differential geometry.
\begin{definition}[Definition 20.1 in \cite{michor2008topics}]
A \textbf{classical fibre bundle} $(E, q, M, F)$ consists of manifolds $E, M , F$ and
a smooth mapping $q : E \to M$ ; furthermore each $x \in M$ has an open neighbourhood
$U$ such that $E|_U := q^{-1} (U )$ is diffeomorphic to $U \times  F$ via a fibre respecting diffeomorphism, $\alpha$:
\[\begin{tikzcd}
	{E|_U} && {U \times F} \\
	& U
	\arrow["q"', from=1-1, to=2-2]
	\arrow["{\pi_0}", from=1-3, to=2-2]
	\arrow["\alpha", from=1-1, to=1-3]
\end{tikzcd}\]
\end{definition}
This is the notion of fibre bundles from classical differential geometry that Definition \ref{def:fibre_bundle} is meant to generalize. 
However, a fibre bundle as in Definition \ref{def:fibre_bundle}{(i)} is not the same as a classical fibre bundle. We need two extra conditions. Interestingly, the results of section \ref{sec:fibre_bundles} and \ref{sec:g-bundles} hold without assuming these extra conditions.
Recall from Definition \ref{def:fibre_bundle}{(iii)} that a fibre bundle $(E,M,q,F,(\alpha_i)_{i \in I}, (e_i)_{i \in I})$ is totally fibred if $\overline{q e_i} = \bar \alpha_i$ for all $i \in I$.

\begin{theorem}\label{thm:classical_fiber_correspondence_diffgeo}
	Let $M$ and $F$ be smooth manifolds.
	A classical fibre bundle $E$ over $M$ with fibre $F$ is exactly the same as a categorical bundle $(E,M,q,F,(\alpha_i)_{i \in I}, (e_i)_{i \in I})$ in \ParMfld{} which
	\begin{itemize}
		\item is totally fibred, and
		\item for which that the projection $q: E \to M$ is an epimorphism.
	\end{itemize}
\end{theorem}
\begin{proof}
	Suppose $(E,M,q,F,(\alpha_i)_{i \in I}, (e_i)_{i \in I})$ is a totally fibred categorical principal bundle in \ParMfld{} for which $q$ is an epimorphism (i.e. surjective). Then for every $x \in M$ there is a $y \in q^{-1}(\{x\})$. Because $\bigvee_{i \in I} \bar \alpha_i = 1_E$, there exists an $i \in I$ such that $\alpha_i(y)$ is defined. 

    Recall from Remark \ref{rmk:partial_set_map_interpretations} that restriction idempotents describe the subset that is their domain of definition.   
    
    Since $\bar \alpha_i^*$ is of the form $e_i \times 1_F$, it follows that the domain of definition of $\alpha_i^*$ is of the form $U \times F$ where $U$ is the domain of definition of $e_i$. Thus the partial isomorphism $\alpha_i$ is an isomorphism of open subsets $V \xrightarrow{\cong} U \times F$ where $V$ is the domain of definition for $\alpha_i$.
    
	Since $\bar \alpha_i$ describes $V$, $\bar e_i$ describes $U$ and the restriction idempotent of a composition describes a preimage, $\bar \alpha_i = \overline{q e_i}$ implies $V = q^{-1}(U)$. Thus $\alpha_i$ is an isomorphism between $E|_U$ and $U \times F$. Since $y \in V = q^{-1}(U)$, it follows that $x \in U$.
	
	We have thus constructed for every $x \in M$ a open neighbourhood $U \subset M$ and a isomorphism $\alpha_i: q^{-1}(U) \to U \times F$ such that
	\[\begin{tikzcd}
		{E|_U} && {U \times F} \\
		& U
		\arrow["q"', from=1-1, to=2-2]
		\arrow["{\pi_0}", from=1-3, to=2-2]
		\arrow["\alpha", from=1-1, to=1-3]
	\end{tikzcd}\]
	commutes. Therefore $(E,q,M,F)$ is a classical fibre bundle, showing that every totally fibred categorical fibre bundle with an epimorphism $q$ is a classical fibre bundle.
	
	Conversely, let $p: E \to M$ be a classical fibre bundle. Then we take the set $M$ of points in the underlying manifold as our index-set. By definition of the classical fibre bundle, for every point $i\in M$ there is a partial isomorphism $\alpha_i: E \to M \times F$ with domain $p^{-1}(U)$.
	As the domain of $\alpha_i^*$ is $U \times F$, its restriction idempotent is of the form $\bar \alpha_i^* = e \times 1$.
	As every point $i \in M$ has an $\alpha_i$ defined in its preimage, the join satisfies $\bigvee_i \alpha_i = 1_E$.
	Thus every classical fibre bundle is a categorical fibre bundle.
\end{proof}

Similarly we will prove that a classical principal bundle is exactly a totally fibred categorical principal bundle with surjective $q$-map in the category of manifolds.

We will consider the classical definitions from \cite{michor2008topics} or \cite{Waldmann2021} (these are the same) and prove that they are equivalent to Definition \ref{def:G-bundle}.

\begin{definition}[Definition 21.1 in \cite{michor2008topics}]
Let $G$ be a Lie group and let $(E, q, M, F)$ be a classical fibre bundle. A \textbf{classical $G$-bundle} structure on the fibre bundle consists of the following data.
\begin{enumerate}[(i)]
    \item A left action $a : G \times  F \to F$ of the Lie group on the standard fibre.
    \item A fibre bundle atlas $(U_i , \alpha_i )$ whose transition functions $(u_{ij} )$ act on $F$
via the $G$-action. That is, there is a family of smooth mappings $(\tau_{ij} : U_i \cap U_j \to G)$
which satisfies the cocycle condition $\tau_{ij}(x) \tau_{jk} (x) = \tau_{ik} (x)$ for $x \in U_i \cap U_j \cap U_k$
and $\tau_{ii} (x) = u$, the unit in the group, such that $u_{ij} (x, f) = a(\tau_{ij}(x), f)$.
\end{enumerate}
A fibre bundle with a $G$-bundle structure is called a $G$-bundle.
\end{definition}
Comparing this with Definition \ref{def:G-bundle} we see that a classical $G$-bundle structure is exactly the same as a $G$-bundle structure in \ParMfld.
\begin{proposition}
A classical $G$-bundle is exactly the same as a $G$-bundle in the join-restriction category \ParMfld{},
\begin{itemize}
    \item that is totally fibred, and
    \item where the projection map $q:E \to M$ is surjective.
\end{itemize}
\end{proposition}
This proposition follows from Theorem \ref{thm:classical_fiber_correspondence_diffgeo} since categorical $G$-bundles are defined from categorical fibre bundles the same way as classical $G$-bundles from classical fibre bundles.

Now we consider the definition of principal bundles in classical differential geometry. 

\begin{definition}(Classical principal bundle)
Let $M$ be a manifold and let $G$ be a Lie group. Then
a principal fibre bundle with structure group $G$ over $M$ is a $G$-bundle over $M$ where the typical fibre is $G$ endowed with the left multiplications as left action.
\end{definition}
Again we see that this is the same as part 2 of Definition \ref{def:G-bundle}.
\begin{proposition}\label{prop:diffgeo_principal_bundles}
A classical principal $G$-bundle is exactly the same as a principal $G$-bundle in the join-restriction category \ParMfld{},
\begin{itemize}
    \item that is totally fibred, and
    \item where the projection map $q:E \to M$ is surjective.
\end{itemize}
\end{proposition}

\subsubsection{Principal bundles in point-set topology}\label{def:classical_topology}
In this section we will look at the definition of a principal bundle in point-set topology and prove that it coincides with the notion developed here. 

A \textbf{right-$G$-object} is a topological space $P$ with a right-group action $a: P \times G \to P$. Given two right-$G$-objects $P, P'$, a \textbf{$G$-equivariant} map is a smooth map $f: P \to P'$ such that 
\[\begin{tikzcd}
	{P\times G} & {P' \times G} \\
	P & {P'}
	\arrow["f"', from=2-1, to=2-2]
	\arrow["{f \times 1_G}", from=1-1, to=1-2]
	\arrow["a"', from=1-1, to=2-1]
	\arrow["{a'}", from=1-2, to=2-2]
\end{tikzcd}\]
commutes.
In \cite{mitchell_2011} a principal bundle is defined as follows.
\begin{definition}
Suppose that $P$ is a right $G$-object equipped with a $G$-map $q: P \to M$ where $G$ acts trivially on $M$. It is a \textbf{principal $G$-bundle} over $M$ if $q$ satisfies the following local triviality condition. $B$ has a covering by open sets $(U_i)_{i \in I}$ such that there exist $G$-equivariant homeomorphisms $(\phi_i : q^{-1}(U_i) \to U_i \times G)_{i \in I}$ such that the diagram
\[\begin{tikzcd}
	{q^{-1}(U_i)} & {U_i \times G} \\
	{U_i}
	\arrow["{\phi_i}", from=1-1, to=1-2]
	\arrow["{\pi_0}", from=1-2, to=2-1]
	\arrow["q"', from=1-1, to=2-1]
\end{tikzcd}\] commutes.
Here $U \times G$ has the right $G$-action $(u,g)h = (u,gh)$.

A morphism $\sigma: P \to Q$ of principal bundles over $M$ is an equivariant morphism $\sigma: P \to Q$ over the identity of $M$, i.e. a morphism $\sigma : P \to Q$ such that
\[\begin{tikzcd}
	{P\times G} & {Q\times G} && P & Q \\
	P & Q &&& M
	\arrow["{G\text{-action}}"', from=1-1, to=2-1]
	\arrow["{G\text{-action}}", from=1-2, to=2-2]
	\arrow["\sigma"', from=2-1, to=2-2]
	\arrow["{\sigma \times 1_G}", from=1-1, to=1-2]
	\arrow["{q_P}"', from=1-4, to=2-5]
	\arrow["{q_Q}", from=1-5, to=2-5]
	\arrow["\sigma", from=1-4, to=1-5]
\end{tikzcd}\]
commute.
\end{definition}
We will now show this definition actually is the same as a surjective, totally fibred principal bundle in the join-restriction category of topological spaces. For this, we first define this join restriction category.
\begin{definition}
Let $ParTop$ be the following join restriction category:
\begin{enumerate}[(i)]
    \item Objects are topological spaces
    \item Morphisms $A \to B$ are continuous maps from an open subset $U \in A$ to $B$.
    \item Composition, identity and restrictions are like for partial maps of sets. In particular the restriction of a map $f: A \to B$ that is defined on $U \subset A$ is the inclusion $U \to A$.
\end{enumerate}
\end{definition}
\begin{lemma}
Given a classical principal bundle $(P,M,q,(U_i)_{i \in I},(\phi_i)_{i \in I})$ like in Definition \ref{def:classical_topology} the change of charts is given by the left-multiplication by a group element: There exists a map $\tau_{ij}: U_i \cap U_j \to G$ such that 
$$
\varphi_i^{-1} \varphi_j = (\langle  1_{U_i \cap U_j} , \tau_{ij} \rangle \times 1_G )(1_{U_i \cap U_j} \times m) : (U_i \cap U_j) \times G  \to (U_i \cap U_j) \times G .
$$

\end{lemma}\label{lemma:Top_change_of_charts_group_mul}
\begin{proof}
We will define $\tau_{ij}$ as the composition of 
$$U_i \cap U_j \xrightarrow{\langle 1, u \rangle } (U_i \cap U_j) \times G \xrightarrow{\phi_i^{-1}\phi_j} (U_i \cap U_j) \times G \xrightarrow{\pi_1} G . $$
In other words, $\tau_{ij}$ is the $G$-component of $\phi_i^{-1}\phi_j$ evaluated on the unit of the group.
In order to show that $\varphi_i^{-1} \varphi_j = (\langle  1_{U_i \cap U_j} , \tau_{ij} \rangle \times 1_G )(1_{U_i \cap U_j} \times m)$, we will evaluate it on an element $(x,g)$ of $(U_i \cap U_j)\times G $:
$$
\varphi_i^{-1} \varphi_j (x,g)\!=\! (1_{U_i \cap U_j}) \phi_i^{-1} \phi_j (x,u,g)\! =\! m (\phi_i^{-1} \phi_j(x,u), g) \! = \! m(x, \tau_{ij}(x) , g) \! = \! (x,m(\tau_{ij}, g)),
$$
where the second equality holds as both $\phi_i^{-1}$ and $\phi_j$ are $G$-equivariant.
\end{proof}

\begin{theorem}\label{prop:topology_principal_bundles}
Let $G$ be a total group object in the join restriction category ParTop, i.e. a topological group.
The category of classical principal $G$-bundles over $M$ (defined in Definition \ref{def:classical_topology}) is isomorphic to the subcategory of categorical principal bundles (defined in Definition \ref{def:G-bundle}) where
\begin{itemize}
    \item The base space of every object is $M$.
    \item The underlying fibre bundle is totally fibred.
    \item The projection map $q: P \to M$ is surjective.
    \item Every morphism $(f,g): (P,M,q) \to (P',M,q')$ has the identity as its second component: $g = 1_M$
\end{itemize}
\end{theorem}
\begin{proof}
Given a categorical principal bundle $P,M,q,(\alpha_i)_{i \in I}$ we obtain a classical principal bundle by taking $U_i$ to be the subset of $M$ where $\alpha_i^*: M \times G \to P$ is defined and $\phi_i:= \alpha_i: P \to M \times G$, defined on $q^{-1}(U_i)$ because the underlying fibre bundle is totally fibred. The space $P$ is a right-$G$-object by Theorem \ref{thm:right_action_principal} and $q$ is a $G$-map to the trivial $G$-object $M$ as $rq = \pi_0 q: P \times G \to M$.

Given a morphism $(f,1_M): (P \xrightarrow{q} M) \to (P' \xrightarrow{q'} M)$, the first component $f: P \to P'$ forms a morphism of classical principal bundles as, by Definition \ref{def:PBun_morphism} it commutes with the $G$-action.

Conversely, given a classical principal bundle $E \xrightarrow{q} M$, then it is also a categorical principal bundle with the $\alpha_i$ constructed from the $\phi_i$ as the partial map given by the span $P \hookleftarrow q^{-1}(U_i) \xrightarrow{\phi_i} M \times G$. As the $(U_i)_{i \in I}$ form a covering, $\bigvee_{i \in\iota} \bar a_i = 1$.
The partial inverse $\alpha_i^*$ is given by the span $M \times G \hookleftarrow U_i \times G \xrightarrow{\phi_i^{-1}} E $. The change of charts is given by the left-multiplication with a group element according to Lemma \ref{lemma:Top_change_of_charts_group_mul}. 
\end{proof}
\subsubsection{Prinicipal bundles in algebraic geometry}
In \cite{Achar2021PerverseSA}, Definition 6.1.6, a notion of principal bundle over an algebraic variety is defined through algebraic properties instead of the local triviality condition used in Definition \ref{def:fibre_bundle} and the differential geometry literature. These principal bundles are not necessarily locally trivial in the Zariski-topology. However they are locally trivial in the étale topology. Whether or not there is a join-restriction category where the categorical principal bundles are the principal bundles defined in \cite{Achar2021PerverseSA} remains an open problem.

The difficulty here is that for a general Grothendieck topology and, in particular, for étale topology open subsets are described by morphisms that are not necessarily embeddings. Thus, if one follows the span construction of \cite{cockett_lack_2007}, the restriction idempotents are not idempotent, $\bar f \bar f \neq \bar f$. The resulting partial map category is not a restriction category.

We still think it is possible to describe principal bundles in algebraic geometry because it should be possible to describe Grothendieck topologies in the language of restriction categories using density relations.

\subsection{The torsor property}
Let $P$ be a principal bundle in a restriction category $\mathbb X$.
In classical differential geometry, the group action $r: P \times G \to P$ coming from Theorem \ref{thm:right_action_principal} is known to be free and proper and act transitively on fibres, as mentioned for example in the proof of Theorem 21.18(3) in \cite{michor2008topics}.
In algebraic geometry, free and transitive group actions play such an important role that they were given a special name, torsors, in \cite{Milne_Etale_cohomology}*{Proposition III.4.1}.
 In \cite{nlab:principalbundle} principal bundles are even introduced as torsors in the slice category.
Often torsors are also thought of as groups without a specified unit, which is encoded as a ternary operation. In \cite{street_torsors_herds_flocks} these definitions are shown to be equivalent.

While there are multiple different definitions of torsors, we will use a generalization of the definition in \cite{street_torsors_herds_flocks}*{Section 2}.
A torsor is an object with a free and transitive group action, which is described by the morphism in the definition below being an isomorphism.

\begin{definition}\label{def:Torsor_vanilla}
    Given a group object $G$ in a category, a \textbf{right $G$-torsor} is an object $T$ with a right $G$-action $a: T \times G \to G$ such that the morphism
    $$
    T \times G \xrightarrow{\langle \pi_0, a\rangle} T \times T
    $$
    is an isomorphism.
\end{definition}

The Definition of \cite{street_torsors_herds_flocks} in addition requires the map from $T$ to the terminal object to be a regular epimorphism in order to exclude the empty set from being a torsor. However, this would require us to make additional conditions on colimits which would be artificial in our setup. Thus we do not require this condition.

We will adapt Definition \ref{def:Torsor_vanilla} slightly in the restriction category setting and distinguish partial and total torsors. Depending on the situation, we will either obtain partial or total torsors from principal bundles.

\begin{definition}\label{def:partial_and_total_torsor}
    Let $G$ be a group object in a restriction category and let $T$ an object with a right $G$-action $a: T \times G \to G$.
    \begin{enumerate}[(i)]
        \item Then $(T,a)$ is a \textbf{partial right $G$-torsor} if the morphism
        $$
        T \times G \xrightarrow{\langle \pi_0, a\rangle} T \times T
        $$
        is a partial isomorphism, i.e. it has a partial inverse.
        \item Then $(T,a)$ is a \textbf{total right $G$-torsor} if the morphism
        $$
        T \times G \xrightarrow{\langle \pi_0, a\rangle} T \times T
        $$
        is a total isomorphism, i.e. it is total and has a total inverse.
    \end{enumerate}
\end{definition}
Since $\langle \pi_0, a\rangle$ is total, a partial right $G$-torsor is total if and only if the inverse $T \times T \to T \times G$ is total. 
We now want to investigate if a principal bundle in a join restriction category is a torsor in some kind of slice category.
For this we define a restriction slice category. This definition stems from current work in progress by Robin Cockett, Goeff Cruttwell, Jonathan Gallagher and Dorette Pronk.

\begin{definition}
Given an object $M$ in a restriction category $\mathbb X$, we define the \textbf{restriction slice category}
$\mathbb X \downsquigarrow M$ as follows:
\begin{enumerate}[(i)]
    \item objects are pairs $(A,f)$ consisting of an object $A$ and a morphism $f: A \to M$,
    \item morphisms from $(A,f)$ to $(B,g)$ are morphisms $\varphi: A \to B$ of $\mathbb X$ such that $f \geq \varphi g$, and
    \item the restriction idempotent $\bar \varphi$ in the restriction slice category $X \downsquigarrow M$ is the same as the restriction idempotent $\bar \varphi$ in the underlying category $\mathbb X$.
\end{enumerate}
Composition and identity are given by the composition and identity operations of the underlying category.

\end{definition}
A restriction slice category is, in particular, a restriction category. 
In order to describe $G$-torsors in the slice category the next lemma constructs a group object in the slice category $\mathbb X \downsquigarrow M$ out of a group object in the underlying category $\mathbb X$.
\begin{lemma}\label{lem:group_object_in_slice}
    Given a group object $(G,m,u,\iota)$ in a Cartesian restriction category $\mathbb X$, the pair $(M \times G, \pi_0)$ is a group object in the restriction slice category $\mathbb X \downsquigarrow M$. The group multiplication, unit and inverse are 
    \begin{align*}
        1_M \times m &: M \times G \times G\cong (M \times G) \times_M (M \times G) \to M \times G,
        \\
        1_M \times u &: M \times 1  \cong M \to M \times G , \text{ and}
        \\
        1_M \times \iota &: M \times G \to M \times G .
    \end{align*}
\end{lemma}
\begin{proof}
The morphisms $1_M \times m ,  1_M \times u $ and $1_M \times \iota$ are morphisms in the slice category since they preserve the projection to the group object,
$$
(1_M \times m ) \pi_0 = \pi_0~, ~~ (1_M \times u) \pi_0 = \pi_0 ~,\text{and}~~ (1_M \times \iota) \pi_0 = \pi_0 .
$$
The fact that $1_M \times m ,  1_M \times u $ and $1_M \times \iota$ fulfill the identities required in Definition \ref{def:group_object} follows from the fact that $m,u$ and $\iota$ satisfy the same identities.
\end{proof}
In light of Lemma \ref{lem:group_object_in_slice} we will call a $(G \times M, \pi_1)$-torsor in the slice category $\mathbb X \downsquigarrow M$ a $G$-torsor in $\mathbb X \downsquigarrow M$ for brevity.

With these definitions, we can now formulate the main results of this section, that principal bundles are in fact torsors in the slice category $\mathbb X \downsquigarrow M$. However, we need to distinguish between the general case and the totally fibred case. In the general case, we obtain only a partial right torsor, but in the totally fibred case, we will obtain a total right torsor.

In the category of sets, every group action on is transitive on a certain subset, namely the orbit of a point.
Therefore, in our opinion, total torsors are more useful than partial torsors, since they can be interpreted as group actions that are transitive, not just transitive on a subset.

\begin{proposition}\label{prop:partial_right_torsor}
    Let $\mathbb X$ be a Cartesian join restriction category,
    let $G$ be a group object in $\mathbb X$, let $(P,M,q,(\alpha_i)_{i \in I},(\tau_{ij})_{i,j \in I})$ be a principal bundle in $\mathbb X$ such that $P \times_M P$ exists and let $r:P\times G \to P$ be the group action from Theorem \ref{thm:right_action_principal}. Then $(P,q)$ is a partial right $G$-torsor in the restriction slice category $\mathbb X \downsquigarrow M$.
    The partial inverse $d^*$ of 
    $$
    \langle \pi_0 , r\rangle : P \times G \cong  P \times_M(M \times G) \to P \times_M  P
    $$ 
    consists of components 
    $$
    P \times_M P \xrightarrow{\alpha_i \times_M\, \alpha_i} M \times G \times G \xrightarrow{\langle \pi_0, \pi_1, \langle \pi_1 \iota , \pi_2 \rangle \rangle} M \times G \times G \xrightarrow{ \alpha_i^* \times 1_G} P \times G \cong P \times_M(M \times G).
    $$
\end{proposition}
In order to prove this we will use the following technical observations about join restriction categories:
\begin{lemma}\label{lem:technicalities_on_inverse_joins}~
\begin{enumerate}[(i)]
    \item In a join restriction category, let $(f_i)_{i \in I}$ and $(g_i)_{i \in I}$ be families of compatible functions indexed by the same set $I$, then 
    $$
    \bigvee_{i \in I} f_i \bigvee_{j \in I} g_j \geq \bigvee_{i \in I} f_i g_i .
    $$
    \item In a restriction category, let $fg \geq \bar f$. Then $fg= \bar f$.
    \item In a restriction category, let $\bar f g \leq f$. Then $\bar f g = \bar g f$ and therefore $f \smile g$.
\end{enumerate}
\end{lemma}
\begin{proof}[Proof of Lemma \ref{lem:technicalities_on_inverse_joins}.]~
    \begin{enumerate}[(i)]
        \item This follows since the join is associative and $\bigvee_{i\in I} f_i \bigvee_{j \in I} g_j =\bigvee_{i\in I} \bigvee_{j \in I} f_i g_j$ is a join over $I \times I$ and $\bigvee_{i \in I} f_i g_i$ is the join over the diagonal subset $\{(i,i)| i \in I\} \subset I \times I$. 
        \item This is true since  $fg \geq \bar f$ means by definition of $\leq$ that $\bar f = \bar {\bar f} fg$ and therefore
        $$
        fg = \bar f fg = \bar{ \bar f} fg = \bar f .
        $$
        \item This is true since $\bar f g \leq f$ means by definition of $\leq$ that
        $$
            \bar f g = \overline{\bar f g} f = \bar f \bar g f = \bar g f.
        $$
    \end{enumerate}
\end{proof}

\begin{proof}[Proof of Proposition \ref{prop:partial_right_torsor}.]
In order to show that $(P,q)$ is a partial right $G$-torsor in $\mathbb X \downsquigarrow M$ we need to show that 
$$
\langle \pi_0 , r\rangle : P \times G \to P \times_M P
$$
is a partial isomorphism in $\mathbb X$. We do this by explicitly constructing its inverse as
$$
d^* = \bigvee_{i \in I} (\alpha_i \times_M \alpha_i) \langle \pi_0 , \pi_1 , \langle \pi_1 \iota, \pi_2 \rangle m \rangle (\alpha_i^* \times 1_G)
$$
Before we show that this join is well-defined, we first show that, if it exists, it is a partial inverse for $\langle \pi_0 , r\rangle$.
We expand the definition of $r$ from Theorem \ref{thm:right_action_principal} and rewrite it in a form more useful for this proof.
\begin{align*}
\langle \pi_0 , r\rangle & =  \left\langle 
\pi_0  \bigvee_{i \in I} \alpha_i \alpha_i^* , 
\bigvee_{i \in I} (\alpha_i \times 1_G)(1_M \times m) \alpha_i^*
\right\rangle
\\
& = \bigvee_{i \in I} (\alpha_i \times 1_G) \langle \langle \pi_0 , \pi_1 \rangle  \alpha_i ^* , \langle \pi_0 , \langle \pi_1 , \pi_2 \rangle m  \alpha_i ^* \rangle \rangle 
\end{align*}
We show that $\langle \pi_0 , r \rangle d^* = 1_{P\times G} = \overline{\langle \pi_0 , r \rangle}$. We expand 
\begin{align*}
    &~\langle \pi_0 , r \rangle d^*  =
    \\
    &
    \bigvee_{i \in I} (\alpha_i \times 1_G) \langle \langle \pi_0 , \pi_1 \rangle  \alpha_i ^* , \langle \pi_0 , \langle \pi_1 , \pi_2 \rangle m  \alpha_i ^* \rangle \rangle 
    \bigvee_{j \in I} (\alpha_j \times_M \alpha_j) \langle \pi_0 , \pi_1 , \langle \pi_1 \iota , \pi_2 \rangle m \rangle (\alpha_j^* \times 1_G).
    \\
    \intertext{Applying Lemma \ref{lem:technicalities_on_inverse_joins}(i) to the previous expansion we obtain the inequality}
    & \langle \pi_0 , r \rangle d^* \geq 
    \\& \bigvee_{i \in I} (\alpha_i \times 1_G) \langle \langle \pi_0 , \pi_1 \rangle  \alpha_i ^* , \langle \pi_0 , \langle \pi_1 , \pi_2 \rangle m  \alpha_i ^* \rangle \rangle 
    (\alpha_i \times_M \alpha_i) \langle \pi_0 , \pi_1 , \langle \pi_1 \iota, \pi_2 \rangle m \rangle (\alpha_i^* \times 1_G)
    \\ =& \bigvee_{i \in I} (\alpha_i \times 1_G) \Bigg\langle \langle \pi_0 , \pi_1 \rangle \alpha_i^* \alpha_i , \Big\langle \langle \pi_0 , \pi_1 \rangle \alpha_i^* \alpha_i \pi_1 \iota , \langle \pi_0 , \langle \pi_1, \pi_2\rangle m \rangle \alpha_i^* \alpha_i \pi_1 \Big\rangle m \Bigg\rangle (\alpha_i^*\times 1_G).
    \\ \intertext{Since $\alpha_i^* \alpha_i =\bar {\alpha_i^*} = e_i \times 1_G$, this equals}
    &
    \bigvee_{i \in I} (\alpha_i \times 1_G)
    \Bigg\langle \pi_0e_i , \pi_1 , \Big\langle \langle \pi_0 e_i, \pi_1 \rangle \pi_1 \iota , \langle \pi_0 e_i , \langle \pi_1, \pi_2\rangle m \rangle \pi_1 \Big\rangle m \Bigg\rangle
    (\alpha_i^* \times 1_G)
    \\ =&
    \bigvee_{i \in I} (\alpha_i \times 1_G)
    \overline{\pi_0 e_i} \Bigg\langle
    \pi_0 , \pi_1, \Big\langle \pi_1 \iota, \langle \pi_1, \pi_2 \rangle m \Big\rangle m \Bigg\rangle 
    (\alpha_i^* \times 1_G).
    \\ \intertext{By associativity of $m$ this equals}
    &
    \bigvee_{i \in I} (\alpha_i \times 1_G)
    \overline{\pi_0 e_i} \langle \pi_0 , \pi_1, \pi_2\rangle
    (\alpha_i^* \times 1_G)
    \\  \intertext{and using $\bar \alpha_i^* = e_i \times 1_F = \overline{\pi_0 e_i}$ we obtain}
    &
    \bigvee_{i \in I} (\alpha_i \times 1_G)
    (\bar{\alpha_i^*} \times 1_G)
    (\alpha_i^* \times 1_G)
    \\=& \bigvee_{i\in I} (\bar \alpha_i \times 1_G)
    \\=& \bigvee_{i\in I}\overline {\pi_0 \alpha_i}
    =1_{P \times G}.
    \end{align*}
    This equations shows that $\langle \pi_0 , r \rangle d^* \geq 1_{P\times G} = \overline{\langle \pi_0 , r \rangle}$ which implies by Lemma \ref{lem:technicalities_on_inverse_joins}(i) that $\langle \pi_0 , r \rangle d^* = 1_{P\times G} = \overline{\langle \pi_0 , r \rangle}$.
    Next we will show that $d^* \langle \pi_0 , r \rangle = \bar{d^*}$.
    First we simplify the left hand side into    
    \begin{align*}
        &~d^* \langle \pi_0 , r \rangle  =
        \\
        &
    \bigvee_{j \in I} (\alpha_j \times_M \alpha_j) \langle \pi_0 , \pi_1 , \langle \pi_1 \iota , \pi_2 \rangle m \rangle (\alpha_j^* \times 1_G)
    \bigvee_{i \in I} (\alpha_i \times 1_G) \langle \langle \pi_0 , \pi_1 \rangle  \alpha_i ^* , \langle \pi_0 , \langle \pi_1 , \pi_2 \rangle m  \alpha_i ^* \rangle \rangle 
    \\
    \geq & \bigvee_{i \in I} (\alpha_i \times_M \alpha_i) \Bigg\langle \langle \pi_0, \pi_1\rangle \alpha_i^* , \Big\langle \pi_0, \big\langle \pi_1 , \langle \pi_1 \iota,\pi_2\rangle m \big\rangle m \Big \rangle \alpha_i^* \Bigg\rangle
    \\
    = & \bigvee_{i \in I} (\alpha_i \times_M \alpha_i) \Bigg\langle \langle \pi_0, \pi_1\rangle \alpha_i^* , \langle \pi_0, \pi_2 \rangle \alpha_i^* \Bigg\rangle
    \\
    = & \bigvee_{i \in I} \langle \pi_0 \alpha_i \alpha_i^* , \pi_1 \alpha_i \alpha_i^*\rangle
    \\
    = & \bigvee_{i \in I}\overline{\pi_0 \alpha_i} \, \overline{\pi_1 \alpha_i}.
    \\
    \intertext{The right hand side is simplified to}
        \bar {d^*} &= \overline{\bigvee_{i \in I} (\alpha_i \times_M \alpha_i) \langle \pi_0 , \pi_1 , \langle \pi_1 \iota, \pi_2 \rangle m \rangle (\alpha_i^* \times 1_G)}
        \\
        &= \bigvee_{i \in I} \overline{\big\langle \pi_0 \alpha_i \alpha_i^* , (\alpha_i \times_M \alpha_i) \langle \pi_1 \iota, \pi_2 \rangle m \big\rangle}
        \\
        &= \bigvee_{i \in I} \overline{\pi_0 \alpha_i} \, \overline{\langle \pi_0 \alpha_i , \pi_1 \alpha_i \rangle}
        \\
        &= \bigvee_{i \in I} \overline{\pi_0 \alpha_i} \, \overline{\pi_1 \alpha_i}.
    \end{align*}
    Therefore we showed that, assuming $d^*$ is defined, it is a partial inverse to $\langle \pi_0 , r \rangle$. It remains to show that the join
    $$
    \bigvee_{i \in I} (\alpha_i \times_M \alpha_i) \langle \pi_0 , \pi_1 , \langle \pi_1 \iota , \pi_2 \rangle m \rangle (\alpha_i^* \times 1_G)
    $$
    is defined, i.e. that for
    $$
    d_i^* = (\alpha_i \times_M \alpha_i) \langle \pi_0 , \pi_1 , \langle \pi_1 \iota , \pi_2 \rangle m \rangle (\alpha_i^* \times 1_G)
    $$
    the compatibility 
    $
    d_i^*\smile d_j^*
    $
    holds for all $i,j \in I$. By Lemma \ref{lem:technicalities_on_inverse_joins}(iii) this amounts to showing the inequality
    $$
    \bar d_i^* d_j^* \leq d_i^*.
    $$
    For this we use that $\bar {d_i^*} = (\bar \alpha_i \times_M \bar \alpha_i)$. We see that
    \begin{align*}
        \bar d_i^* d_j^* =& (\bar \alpha_i \times_M \bar \alpha_i)  (\alpha_j \times_M \alpha_j) \langle \pi_0 , \pi_1 , \langle \pi_1 \iota, \pi_2 \rangle m \rangle (\alpha_j^* \times 1_G)
        \\ 
        =&\Big\langle \pi_0 \bar \alpha_i \bar \alpha_j , (\alpha_i \alpha_i^* \alpha_j \times_M \alpha_i \alpha_i^* \alpha_j) \langle \pi_1 \iota, \pi_2\rangle m \Big\rangle
        \\
    \intertext{Since $(P,M,q,(\alpha_i)_{i \in I})$ is a principal bundle, we know that $\alpha_i^* \alpha_j \pi_1 = \langle \pi_0 \tau_{ij} , \pi_1 \rangle m$. Therefore the previous expression equals}
        =& \Bigg\langle \pi_0 \bar \alpha_i \bar \alpha_j ,(\alpha_i \times_M \alpha_i) \Big\langle \langle \pi_0 \tau_{ij} , \pi_1 \rangle m \iota , \langle \pi_0 \tau_{ij} , \pi_2 \rangle m \Big\rangle \Bigg\rangle 
    \\
    \intertext{where the indices of the projections stem from the identification of $(M \times G)\times_M (M \times G))$ with $M \times G \times G$. The inverse of a product $\langle f, g\rangle m \iota$ is the same as the reversed product of the inverses $\langle g \iota, f \iota \rangle m$ and therefore this equals}
        & \Bigg\langle \pi_0 \bar \alpha_i \bar \alpha_j ,(\alpha_i \times_M \alpha_i) \Big\langle \langle  \pi_1  \iota  , \pi_0 \tau_{ij} \iota\rangle m,  \langle \pi_0 \tau_{ij} , \pi_2 \rangle m \Big\rangle \Bigg\rangle 
    \\
    =& \Big\langle \pi_0 \bar \alpha_i \bar \alpha_j ,(\alpha_i \times_M \alpha_i)\overline{\pi_0 \tau_{ij}} \langle \pi_1 \iota , \pi_2 \rangle m \Big \rangle
    \\
    \leq & \Big\langle \pi_0 \alpha_i \alpha_i^* , (\alpha_i \times_M \alpha_i)  \langle \pi_1 \iota , \pi_2 \rangle m \Big \rangle
    \\
     =&  (\alpha_i \times_M \alpha_i) \langle \pi_0 , \pi_1 , \langle \pi_1 \iota , \pi_2 \rangle m \rangle (\alpha_i^* \times 1_G) = d_i.
    \end{align*}
    The preceding calculation shows that $\bar d_i^* d_j^* \leq d_i^*$. Now it follows from Lemma \ref{lem:technicalities_on_inverse_joins}(iii) that $d_i^* \smile d_j^*$ which concludes the proof of Theorem \ref{prop:partial_right_torsor}.
\end{proof}
While Proposition \ref{prop:partial_right_torsor} gives a partial isomorphism, this isomorphism is, in general, not total. In order to obtain a total right $G$-torsor, as in Definition \ref{def:partial_and_total_torsor}(iii), we need to require $(P,M,q,(\alpha_i)_{i \in I})$ to be totally fibred, as in Definition \ref{def:fibre_bundle}(iii).
\begin{theorem}\label{thm:total_right_torsor}
    Let $(P,M,q,(\alpha_i)_{i \in I})$ be a totally fibred principal $G$-bundle in a join-restriction category $\mathbb X$. Then $P$ is a total right $G$-torsor in the slice category $\mathbb X \downsquigarrow M$ via the isomorphism in Proposition \ref{prop:partial_right_torsor}.
\end{theorem}
\begin{proof}
    Since the map $\langle\pi_0, r\rangle$ is total, we only need to show that 
    $$
    d^* = \bigvee_{i \in I} (\alpha_i \times_M \alpha_i) \langle \pi_0 , \pi_1 , \langle \pi_1 \iota, \pi_2 \rangle m \rangle (\alpha_i^* \times 1_G)
    $$
    is total too.
    As we calculated in the proof of Proposition \ref{prop:partial_right_torsor}.
    $$
    \bar d^* = \bigvee_{i \in I} \bar d_i^* = \bigvee_{i \in I} (\bar \alpha_i \times_M \bar \alpha_i).
    $$
    Since we assumed that $(P,M,q,(\alpha_i)_{i \in I})$ is totally fibred, $\bar \alpha_i = \overline {q e_i}$ implies
    $$
        \bar d^* = \bigvee_{i \in I} (\overline {q e_i} \times _M \overline {q e_i}).
    $$
    Since $\pi_0 q = \pi_1 q : P \times_M P \to M$, 
    $$
        \bar d^* = \bigvee_{i \in I} \overline {\pi_0 q e_i} =  \bigvee_{i \in I} \overline {\pi_0 \bar \alpha_i} =  \overline { \bigvee_{i \in I}\pi_0 \bar \alpha_i} = \overline { \pi_0 \bigvee_{i \in I} \bar \alpha_i} = \bar \pi_0 = 1_{P \times_M P}.
    $$
    This shows that $d^*$ is a partial inverse to $\langle \pi_0 ,r\rangle$ that is total and therefore a total inverse.
\end{proof}

\subsection{The vertical bundle of a principal bundle}\label{sec:vertical_bundle}
Given a group object $G$ and a principal $G$-bundle $P\xrightarrow{q} M$ in a Cartesian tangent join restriction category $\mathbb X$, the \textbf{vertical bundle} $T_0P$ is defined to be the pullback
\[\begin{tikzcd}
	{T_0(P)} & {T(P)} \\
	M & {T(M).}
	\arrow["{Tq^*(0)}", from=1-1, to=1-2]
	\arrow["pq"', from=1-1, to=2-1]
	\arrow["\lrcorner"{anchor=center, pos=0.125}, draw=none, from=1-1, to=2-2]
	\arrow["{T(q)}", from=1-2, to=2-2]
	\arrow["0"', from=2-1, to=2-2]
\end{tikzcd}\]
It is known, for example from Section 21.18 of \cite{michor2008topics}, that in classical
differential geometry the vertical bundle is isomorphic to $P \times T(G)_u$, where $T(G)_u$ is the tangent space of $G$ at the unit given by the (restriction) pullback
\[\begin{tikzcd}
	{T(G)_u} & {T(G)} \\
	1 & G
	\arrow["p_u^*", from=1-1, to=1-2]
	\arrow["!", from=1-1, to=2-1]
	\arrow["\lrcorner"{anchor=center, pos=0.125}, draw=none, from=1-1, to=2-2]
	\arrow["p", from=1-2, to=2-2]
	\arrow["u"', from=2-1, to=2-2]
\end{tikzcd}\]
as defined in Section \ref{sec:trivializing_the_tangent_bundle}. We now prove the same in the more general setting of a tangent restriction category.

\begin{theorem}\label{thm:vertical_bundle}
Given a group object $G$ and a principal $G$-bundle $P\xrightarrow{q} M$ in a Cartesian tangent join restriction category, if $T(G)_u$ exists, the vertical bundle exists and is isomorphic to $P \times T(G)_u$.
\end{theorem}

In local charts (i.e. when considering partial isomorphisms) $P\cong M \times G$ and thus $T(P) \cong T(M) \times T(G) \cong T(M) \times G \times T(G)_u$ due to Theorem \ref{thm:group_tangent}.  Thus in local charts the diagram is
\[\begin{tikzcd}
	{M \times G \times T(G)_u} && {T(M) \times G \times T(G)_u} \\
	M && {T(M)}
	\arrow["{0 \times 1_G \times 1_{T(G)_u}}", from=1-1, to=1-3]
	\arrow["{\pi_0}"', from=1-1, to=2-1]
	\arrow["{\pi_0}", from=1-3, to=2-3]
	\arrow["0"', from=2-1, to=2-3]
\end{tikzcd}\]
which is a pullback. This shows that the statement of Theorem \ref{thm:vertical_bundle} makes sense in local coordinates.  The result of Theorem \ref{thm:vertical_bundle} is stronger though. It states that there is a total isomorphism going globally between $T_0(P)$ and $P \times T(G)_u$. 

\begin{proof}
We will show that the following is a pullback square:
\[\begin{tikzcd}
	{P \times T(G)_u} && {T(P)} \\
	M && {T(M)}
	\arrow["{(0 \times p^*_u)T(r)}", from=1-1, to=1-3]
	\arrow["{\pi_0q}"', from=1-1, to=2-1]
	\arrow["{T(q)}", from=1-3, to=2-3]
	\arrow["0_M"', from=2-1, to=2-3]
\end{tikzcd}\]
The diagram commutes because the right action $r$ preserves the basepoint (i.e. $rq=\pi_0 q$) and therefore
$$
(0_P \times p^*(u)) T(r) T(q) = (0_P \times p^*(u)) \pi_0 T(q) = \pi_0 0_P T(q) = \pi_0 q 0_M.
$$
Since all morphisms in the square are total we can use Lemma \ref{lem:restriction_pullback_of_total} to verify that the diagram is a pullback. We assume there is an object $Z$ with (not necessarily total) morphisms $f: Z \to T(P)$ and $g: Z \to M$ such that $g0 = f T(q)$. We need to show that there is a unique morphism $\phi: Z \to P \times T(G)_u$ such that the diagram 
\begin{equation}
\begin{tikzcd}
	Z \\
	& {P \times T(G)_u} && {T(P)} \\
	& M && {T(M)}
	\arrow["\phi"{description}, dashed, from=1-1, to=2-2]
	\arrow["f", curve={height=-6pt}, from=1-1, to=2-4]
	\arrow["g"', curve={height=6pt}, from=1-1, to=3-2]
	\arrow["{(0 \times p^*_u)T(r)}", from=2-2, to=2-4]
	\arrow["{\pi_0q}"', from=2-2, to=3-2]
	\arrow["{T(q)}", from=2-4, to=3-4]
	\arrow["0"', from=3-2, to=3-4]
\end{tikzcd} \label{diag:limit_for_vertical_bundle}
\end{equation}
commutes.
We construct the morphism $\phi$ using maps 
$$
\xi_i = \langle p ,\langle !,  T(\alpha_i) \pi_1 \langle p \iota 0,1\rangle T(m) \rangle \rangle : T(P) \to P \times T(G)_u.
$$
Define the morphism $\phi$ to be 
$$
\phi = \bigvee_{i \in I} f \xi_i.
$$
It may be surprising that we do not need to use $g$. However, since $f T(q) p = g$, $f$ contains all the information we need.
The remainder of this proof consists of the following three steps.
\begin{enumerate}[(1)]
    \item In order for $\phi$ to be well-defined we need to check that the morphisms $f \xi_i$ are compatible with each other, i.e. that $\overline{f \xi_i} f \xi_j = \overline{f \xi_j} f \xi_i$. 
    \item The morphism $\phi$ must make the outer triangles of Diagram \ref{diag:limit_for_vertical_bundle} commute, i.e. $\phi (0 \times p_u^*) T(r) = f$ and that $\phi \pi_0q = g$.
    \item The morphism $\phi$ must be the unique such morphism.
\end{enumerate}
For step (1) we will show that $\overline{f \xi_i} f \xi_j \leq f \xi_i$ which implies compatiblity by Lemma \ref{lem:technicalities_on_inverse_joins}. We simplify and expand $\overline{f \xi_i} f \xi_j$. By R.4 of Definition \ref{def:restriction_categories},
\begin{align*}
\overline{f \xi_i} f \xi_j =& f \bar \xi_i \xi_j
\\
=&f \overline{T(\alpha_i)} \langle p , \langle ! ,  T(\alpha_j) \pi_1 \langle p \iota 0 , 1 \rangle T(m)\rangle \rangle .
\\
\intertext{Here we used $\bar \xi_i = \overline{T(\alpha_i)}$ and the definition of $\xi_i$ in order to simplify and expand the expression. In order to use the interaction between $\alpha_i$ and $\alpha_j$ we now compose}
&\langle fp , \langle !, f \overline{T(\alpha_i)} T(\alpha_j) \pi_1 \langle p \iota 0 , 1 \rangle T(m)\rangle \rangle
\\
=& \langle fp , \langle !, f T(\alpha_i) T(\alpha_i^*) T(\alpha_j) \pi_1 \langle p \iota 0 , 1 \rangle T(m)\rangle \rangle.
\\
\intertext{Since $P$ is a principal bundle we can rewrite $\overline {T(\alpha_i)} T(\alpha_j)$ into an expression of $T(\alpha_i)$ using Definition \ref{def:G-bundle}. Since $P$ is a principal bundle, this equals}
& \langle fp , \langle !, f T(\alpha_i) \langle \pi_0 T(\tau_{ij}) , \pi_1 \rangle T(m) \pi_1 \langle p \iota 0 , 1 \rangle T(m)\rangle \rangle
\\
=& \langle fp , \langle !, \langle f T(\alpha_i) \pi_1 p \iota 0 , f T(\alpha_i) \pi_0 T(\tau_{ij}) p \iota 0 , f T(\alpha_i) \pi_0 T(\tau_{ij}) , f T(\alpha_i) \pi_1 \rangle  T(m_4) \rangle \rangle.
\\
\intertext{Now we have an expression containing only $T(\alpha_i)$ and we will continue by simplifying this expression. Since $\alpha_i \pi_0 = q$ the previous expression produces the inequality}
\overline{f \xi_i} f \xi_j \leq & \langle fp , \langle !, \langle f T(\alpha_i) \pi_1 p \iota 0 , f T(q) p 0 T(\tau_{ij}) T(\iota)  , f T(q) T(\tau_{ij}) , f T(\alpha_i) \pi_1 \rangle  T(m_4) \rangle \rangle.
\\
\intertext{Since $f T(q) p 0 = g 0 = f T(q)$ the previous expression produces the inequality}
\overline{f \xi_i} f \xi_j
\leq & \langle fp , \langle !, \langle f T(\alpha_i) \pi_1 p \iota 0 , f T(\alpha_i) \pi_1 \rangle  T(m_4) \rangle \rangle 
\\
= & f \langle p , \langle ! , T(\alpha_i) \pi_1 \langle p \iota 0 , 1 \rangle T(m) \rangle \rangle = f \xi_i.
\end{align*}
Using Lemma \ref{lem:technicalities_on_inverse_joins}(iii) this implies that the morphisms $(f\xi_i)_{i \in I}$ are compatible with each other.

For step (2) we directly show the first equation $\phi (0 \times p_u^*) T(r) = f$ by calulating
\begin{align*}
&\phi (0 \times p_u^*)T(r)
\\=& \bigvee_{i \in I} f \xi_i (0 \times p_u^*) T(r)
\\
=&  \bigvee_{i \in I} f \langle p , \langle !, T(\alpha_i) \pi_1 \langle p \iota 0 , 1_{T(G)} \rangle T(m) \rangle \rangle (0 \times p_u^*) T(r).
\intertext{For every morphism $h$ into $T(G)$, satisfying $hp=!u$, the universal property of the pullback guarantees that the equation $\langle !,h\rangle p_u^* = h$ holds.  Using $h= T(\alpha_i) \pi_1 \langle p \iota 0 , 1_{T(G)}\rangle T(m) ,$ the previous line equals 
}
& \bigvee_{i \in I} f \langle p 0 , T(\alpha_i) \pi_1 \langle p \iota 0 , 1_{T(G)} \rangle T(m)
\rangle T(r).
\intertext{Now, expanding the definition of $r$ from Theorem \ref{thm:right_action_principal} the previous line equals}
& \bigvee_{i \in I}\bigvee_{j \in I} f \langle p 0 , T(\alpha_i) \pi_1 \langle p \iota 0 , 1_{T(G)} \rangle T(m)
\rangle (T(\alpha_j) \times 1_{T(G)})(1_{T(M)} \times T(m)) T(\alpha_j^*)
\\
\geq& \bigvee_{i \in I} f \langle p 0 , T(\alpha_i) \pi_1 \langle p \iota 0 , 1_{T(G)} \rangle T(m)
\rangle (T(\alpha_i) \times 1_{T(G)})(1_{T(M)} \times T(m)) T(\alpha_i^*)
\\=& \bigvee_{i \in I} f \langle T(\alpha_i) p 0 , T(\alpha_i) \pi_1 \langle p \iota 0 , 1_{T(G)} \rangle T(m) \rangle (1_{T(M)} \times T(m)) T(\alpha_i^*)
\\
=& \bigvee_{i \in I} f \langle T(\alpha_i) \pi_0 p 0, T(\alpha_i) \pi_1 p 0 , \langle T (\alpha_i) \pi_1 p \iota 0 , T(\alpha_i) \pi_1 \rangle T(m) \rangle (1 \times T(m)) T(\alpha_i^*) .
\intertext{Now due to the associativity, the inverse and the neutral property of the multiplication 
$ T(\alpha_i) \pi_1 p 0$ and $T (\alpha_i) \pi_1 p \iota 0$
cancel each other out. In addition $\alpha_i \pi_0 = \bar \alpha_i q$ and thus the previous expression simplifies to}
& \bigvee_{i \in I} f \langle T(q) p 0, T(\alpha_i) \pi_1  \rangle T(\alpha_i^*)
\\
=& \bigvee_{i \in I} \langle g 0 p 0, f T(\alpha_i) \pi_1  \rangle T(\alpha_i^*)
\\
=& \bigvee_{i \in I} \langle g 0, f T(\alpha_i) \pi_1  \rangle T(\alpha_i^*)
\\
=& f \bigvee_{i \in I} \langle T(\alpha_i) \pi_0, T(\alpha_i) \pi_1  \rangle T(\alpha_i^*) = f \bigvee_{i \in I} T(\alpha_i) T(\alpha_i^*) = f
\end{align*}
where the last equality is due to Lemma \ref{lem:join_and_tangent}.
Since $\bar \phi = \bigvee_{i \in I} \overline{f \xi_i} \leq \bar f$, we also have that $\overline{\phi (0 \times p_u^*)T(r)} \leq \overline f$. Together 
$$
\phi (0 \times p_u^*)T(r) \geq f ~ \text{ and } ~ \overline{\phi (0 \times p_u^*)T(r)} \leq \overline f
$$
imply that $\phi (0 \times p_u^*)T(r) = f$ as required.
Now we use $g 0 = f T(q)$ and $\pi_0 q 0 = (0 \times p_u^*) T(r) T(q)$ to show the second equation 
$$
\phi \pi_0 q = \phi \pi_0 q 0 p = \phi (0 \times p_u^*) T(r) T(q) p = fT(q) p= g0p = g .
$$
Lastly, step (3), the uniqueness of $\phi$, needs to be shown. Suppose there is another morphism $\phi': Z \to P \times T(G)_u$ making these diagrams commute, then $\phi' (0 \times p_u^*)T(r) \xi_i \smile \phi' (0 \times p_u^*)T(r) \xi_j$ for all $i,j \in I$ because $\phi' (0 \times p_u^*)T(r) \xi_i = f \xi_i$ and $f \xi_i \smile f \xi_j$. 
Then we have  
\begin{equation}\label{eqn:phi_phi'}
\phi' = \bigvee_{i \in I} \phi' (0 \times p_u^*)T(r) \xi_i = \bigvee_{i \in I} f \xi_i = \phi .
\end{equation}
The first equality in \ref{eqn:phi_phi'} holds because
\begin{align*}
&\bigvee_{i \in I} \phi' (0 \times p_u^*) T(r) \xi_i
\\
=& \bigvee_{i \in I} \phi' (0 \times p_u^*) (T(\alpha_i) \times 1_{T(G)})(1_{T(M)} \times T(m)) T(\alpha_i^*) \langle p , \langle ! , T(\alpha_i) \pi_1 , \langle p \iota 0 , 1_{T(G)}\rangle T(m) \rangle \rangle 
\\
=& \bigvee_{i \in I} \phi' \langle T\pi_0 0 T(\alpha_i) T(\pi_0), \langle \pi_0 0 T(\alpha_i) \pi_1 , \pi_1 p_u^* \rangle T(m) \rangle T(\alpha_i^*)
\langle p , \langle ! , T(\alpha_i) \pi_1 \langle p \iota 0 , 1_{T(G)}\rangle T(m) \rangle \rangle .
\end{align*}
We will now calculate the two components of this morphism separately and show that $\bigvee_{i \in I} \phi' (0 \times p_u^*) T(r) \xi_i = \phi'$ by showing that the first component is $\phi' \pi_0$ and the second component is $\phi'\pi_1$. The first component is
\begin{align*}
    &\bigvee_{i \in I } \phi' \Big\langle \pi_0 0 T(\alpha_i) T(\pi_0) , \langle \pi_0 0 T(\alpha_i) T(\pi_1) , \pi_1 p_u^* \rangle T(m) \Big\rangle T(\alpha_i^*) p.
    \\
    \intertext{Due to naturality of $p$, $0p=1$ and $p_u^*p = !u$, the previous line simplifies to}
    &\bigvee_{i \in I} \phi' \Big\langle \pi_0 \alpha_i \pi_0 , \langle \pi_0 \alpha_i \pi_1 , ! u\rangle m \Big\rangle \alpha_i^* .
    \\
    \intertext{Since $u$ is the unit for $m$ this equals}
    & \bigvee_{i \in I} \phi' \langle \pi_0 \alpha_i \pi_0 , \pi_0 \alpha_i \pi_1 \rangle \alpha_i^* = \bigvee_{i \in I} \phi' \pi_0 \alpha_i \alpha_i^* = \phi'\pi_0.
\end{align*}
The second component of $\bigvee_{i \in I} \phi' (0 \times p_u^*) T(r) \xi_i$ is
\begin{align*}
    &\bigvee_{i \in I} \Big\langle \pi_0 0 T(\alpha_i) T(\pi_0) , \langle \pi_0 0 T(\alpha_i) T(\pi_1), \pi_1 p_u^* \rangle T(m) \Big\rangle T(\alpha_i^*) \Big\langle ! , T(\alpha_i) \pi_1 \langle p \iota 0 , 1 \rangle T(m) \Big\rangle
    \intertext{which equals}
    &
    \bigvee_{i \in I} \Bigg\langle ! , \phi' \Big\langle \pi_0 0 T(\alpha_i) T(\pi_0) , \langle \pi_0 0 T(\alpha_i) T(\pi_1) , \pi_1 p_u^*\rangle T(m) \Big\rangle T(\alpha_i^*) T(\alpha_i) \pi_1 \langle p \iota 0 , 1 \rangle T(m) \Bigg\rangle .
    \\
    \intertext{Since $\alpha_i^*\alpha_i = \bar \alpha_i^* = e_i \times 1_G$ and for all morphisms $a$ and $b$ $\langle a ,b\rangle \pi_1 = \bar a b$, this equals}
    &
    \bigvee_{i \in I} \Big\langle ! , \phi' \overline{\pi_0 0 T(\alpha_i) T(\pi_0)T(e_i)} 
    \langle \pi_0 0 T(\alpha_i) T(\pi_1) , \pi_1 p_u^* \rangle T(m) 
    \langle p \iota 0 , 1 \rangle T(m) \Big\rangle.
    \intertext{Using naturality and totality of $0$, $\pi_0 e_i = \bar \alpha_i^* \pi_0$ and $\alpha_i \bar \alpha_i^* = \alpha_i$ this equals}
    &
    \bigvee_{i \in I} \Bigg\langle ! , \phi' \overline{\pi_0 \alpha_i}
    \Big\langle
    \langle \pi_0 0 T(\alpha_i) T(\pi_1) , \pi_1 p_u^* \rangle T(m) p \iota 0
    ,
    \langle \pi_0 0 T(\alpha_i) T(\pi_1) , \pi_1 p_u^* \rangle T(m)
    \Big\rangle T(m) \Bigg\rangle.
    \\
    \intertext{Using naturality of $0$, that $0$ is a section of $p$, and the fomula $\langle a ,b \rangle m \iota = \langle b \iota , a \iota \rangle m$ for the inverse of a product, this becomes}
    &
    \bigvee_{i \in I}\Bigg\langle ! , \phi' \overline{\pi_0 \alpha_i}
    \Big\langle
    \langle  \pi_1 p_u^* p 0 T(\iota) , \pi_0 0 T(\alpha_i \pi_1 \iota) \rangle T(m)
    ,
    \langle \pi_0 0 T(\alpha_i) T(\pi_1) , \pi_1 p_u^* \rangle T(m)
    \Big\rangle T(m) \Bigg\rangle.
    \\
    \intertext{Due to the associativity and inverse property of $m$ and the defining property of $p_u^*$, $p_u^* p 0 = !u0 = !T(u)$, this equals}
    & \bigvee_{i \in I} \Big\langle ! , \phi' \overline{\pi_0 \alpha_i} \langle ! T(u) T(\iota) , \pi_1 p_u^* \rangle T(m) \Big\rangle .
    \intertext{Since $\bigvee_{i \in I} \bar \alpha_i = 1_P$ and $u\iota = u$ is the unit for m this equals}
    & \langle ! , \phi' \pi_1 p_u^*\rangle = \phi' \pi_1 \langle ! , p_u^* \rangle = \phi' \pi_1
\end{align*}
where the last equality holds since $!$ and $p_u^*$ are the projections out of the pullback $T(G)_u$ and thus $\langle ! , p_u^*\rangle = 1_{T(G)_u}$.
The two components together imply that $\phi' = \bigvee_{i \in I} \phi' (0 \times p_u^*)T(r) \xi_i$ which was missing to prove equation \ref{eqn:phi_phi'} which concludes the proof by proving uniqueness of $\phi$.

\end{proof}

\section{Concluding remarks}

The goal of this paper was to understand group objects and principal bundles in the general categorical settings of tangent categories and join restriction categories. These structures have been introduced and we have shown that many results from classical differential geometry can be recovered. In particular our construction is a generalization of principal bundles in differential geometry and topology.


Aintablian and Blohmann have independently developed a categorical approach to Lie groupoids and Lie algebroids in \cite{Loryaintablian2025differentiablegroupoidobjectsabstract} which is similar to some of our results in Section 4.  In particular, Theorem 7.5 of \cite{Loryaintablian2025differentiablegroupoidobjectsabstract} is (up to small differences) equivalent to Theorem \ref{thm:Lie_Algebra} in this paper.  The main focus of \cite{Loryaintablian2025differentiablegroupoidobjectsabstract} is Lie groupoids and their behavior, while the main focus of our work is group objects and principal bundles in restriction categories. In particular, in Section 4, we were primarily interested in examining the tangent bundle of a group object. 
These results are not contained in \cite{Loryaintablian2025differentiablegroupoidobjectsabstract} since it works in the more general setup of groupoid objects, where these results may not hold.
We relate the result about vector fields to the tangent space in Theorem \ref{thm:Lie_Algebra}. A technical difference between the constructions of \cite{Loryaintablian2025differentiablegroupoidobjectsabstract} and our construction is that we do not assume that the tangent bundles have negatives (the fibres are monoids, not groups) nor scalar multiplication, though Theorem \ref{thm:eckmann_hilton_for_m_and_plus} shows that the relevant negatives always exist.

In future work, we hope, in particular, to explore how the notion of group objects and principal bundles behaves in algebraic geometry.

In addition we are interested to see in which cases we can recover a group object from its Lie algebra. For this we need a notion of differential equations, like the one provided in \cite{cockett2019diffeq}.

Another important continuation will be the connection with differential bundles. In particular we would like to generalize the classical result that principal $\mathrm{GL}_n$-bundles correspond to vector bundles by showing that $G$-bundles, whose fibre is a differential object, are differential bundles.
\section*{Declarations}
\textbf{Conflict of interest:} The authors have no Conflict of interest to declare that are relevant to the content of this
article.

\noindent
\textbf{Funding:} Florian Schwarz was funded by the Alberta Innovates Graduate Student Scholarship. Robin Cockett was partially funded by NSERC.
\bibliography{sample}
\end{document}